\documentclass{amsart}
\usepackage{enumerate, amsmath, color, amsfonts, amsthm, amssymb}

\usepackage[svgnames]{xcolor}
\usepackage{tikz}
\usepackage{hyperref}
 
\title[Computation of Extension Spaces Using Planar Curves]
{Computation of extension spaces for the path algebra of type $\tilde A(n-1,1)$ using planar curves} 

\author[H. A. Werth]{Heather Anna Werth}
\address{Department of Mathematics, University of Alabama,
Tuscaloosa, AL 35487, U.S.A.}
\email{hawerth@crimson.ua.edu}

\numberwithin{equation}{section}
\theoremstyle{plain}
\newtheorem{Thm}[equation]{Theorem}
\newtheorem{Prop}[equation]{Proposition}

\newtheorem{Lem}[equation]{Lemma}

\theoremstyle{definition}
\newtheorem{Def}[equation]{Definition}
\newtheorem{Exa}[equation]{Example}
\newtheorem{Rmk}[equation]{Remark}

\begin{document}

\begin{abstract} $Q$ is a quiver of type $\tilde A(n-1,1)$ if its graph is of affine type $\tilde A_{n-1}$ and if its arrows have a certain orientation. We develop a bijection between the set of indecomposable $kQ$-modules whose dimension vectors are positive real roots of the root system associated to $Q$ and a certain set of planar curves. We prove that the number of self-intersections of the curve which corresponds to the module $M$ is equal to the dimension of $\text{Ext}^1_{kQ}(M,M)$. We also prove that, for many pairs of modules $(M,N)$, the number of intersections between the corresponding two curves is equal to the dimension of $\text{Ext}^1_C (M,N)$, where $C$ is the cluster category of $kQ$-mod.
\end{abstract}

\maketitle

\section{Introduction} 

\noindent
Fix $k$ an algebraically closed field and $Q$ a finite, acyclic, connected quiver. The category of modules over the path algebra $kQ$ is equivalent to the category of representations of the quiver $Q$ over $k$. This category is Krull-Schmidt, so classifying the indecomposable objects is of high importance. A fundamental step in this direction was taken by V. Kac (\cite{Kac1,Kac2}; see also \cite{Kac}), when he associated to the undirected graph of $Q$ a certain root system, and showed that the dimension vectors of indecomposable representations are in bijection with the positive roots of the root system. A positive root is called real if it can be obtained through action of the Weyl group on a simple root of the system. Kac proved that if the dimension vector of an indecomposable module is a positive real root, then it is the unique module (up to isomorphism) with that dimension vector. 

The dimension vectors of $kQ$-modules are therefore well understood in terms of the root system. Their homological properties, however, are dependent upon the orientation of $Q$, and so one must look beyond the root system to study them. One approach to computing the homological invariants is to interpret $kQ$-modules as familiar geometric objects, in such a way that homological data are clearly expressed by geometric or combinatorial properties. This method of geometrically modelling important categories has often been used successfully in the representation theory of algebras -- for a small selection of references, see \cite{Schroll, labardini, He2020AGM, Schiffler1, Schiffler2}. In this paper, we are interested in the special situation in which indecomposable $kQ$-modules are interpreted as curves whose intersections count the dimensions of extension spaces.

In \cite{FST}, a triangulation of a marked surface $(\mathcal{S},\mathcal{M})$ was associated to a quiver $Q$ of finite type $A$ or $D$ or of affine type $\tilde A$ or $\tilde D$. When $(\mathcal{S},\mathcal{M})$ is an unpunctured surface (this occurs for types $A$ and $\tilde A$), the authors of \cite{article} construct a potential on the quiver $Q$. From the quiver with potential one can construct a Jacobian algebra (\cite{article}, Sections 2.1 and 2.2) and by Proposition 4.2 in \cite{article}, curves in $(\mathcal{S},\mathcal{M})$ correspond to indecomposable modules over that algebra. Furthermore, using the Jacobian algebra, one can associate to $(\mathcal{S},\mathcal{M})$ a generalized cluster category $C(\mathcal{S},\mathcal{M})$ (see, e.g., the Introduction of \cite{BrustleZhang2011}), whose set of indecomposable objects includes the indecomposable modules over the Jacobian algebra. Therefore, a curve in $(\mathcal{S},\mathcal{M})$ corresponds to an indecomposable object in $C(\mathcal{S},\mathcal{M})$. Now, Theorem 3.4 in \cite{zhang2012cotorsion} states that the minimal number of intersections between two curves (up to homotopy) in $(\mathcal{S},\mathcal{M})$ is equal to the dimension of the Ext space of the corresponding objects in $C(\mathcal{S},\mathcal{M})$. For completeness, we mention that \cite{labardini_potential} defines a quiver with potential for any surface (punctured or unpunctured), and Theorem 5.5 in \cite{qiu_zhou_2017} is the analogous intersection result for punctured $(\mathcal{S},\mathcal{M})$. 

When $Q$ is acyclic (without potential), the Jacobian algebra is just $kQ$  and $C(\mathcal{S},\mathcal{M})$ is just $C$, the cluster category of $kQ$-mod (see, e.g., \cite{Amiot}, Corollary 3.12, or \cite{Keller}, Section 6.11). Thus, in this case, curves in $(\mathcal{S},\mathcal{M})$ correspond to indecomposable modules over $kQ$ and their intersections count dimensions of extension spaces between those modules in $C$. This fact will be crucial to proving the main result of this paper.

In this paper, we will use the well-known surface $(\mathcal{S},\mathcal{M})$ to increase understanding of a newer geometric interpretation of $kQ$-mod which comes from \cite{lee}. In that paper, the authors defined a correspondence between (dimension vectors of) indecomposable $kQ$-modules and curves on a Riemann surface $\Gamma$. It should be noted that \cite{lee} defined curves only for the positive real roots. However, the formulation of \cite{lee} applies to all acyclic quivers, whereas the surface in \cite{FST} is only defined for quivers of types $A$, $D$, $\tilde A$, and $\tilde D$. 

By \cite{lee}, a positive real root $\alpha$ of the root system of an acyclic quiver $Q$ corresponds to a family of plane curves $\mathcal{C}(\alpha)$. The authors conjectured that $\mathcal{C}(\alpha)$ contains a curve without self-intersections if and only if $M_\alpha$, the unique module whose dimension vector is $\alpha$, has a vanishing self-extension space in $kQ$-mod\footnote{If $M_\alpha$ satisfies this condition, $\alpha$ is called a real Schur root of the root system of $Q$.} (\cite{lee}, Conjecture 1.3).  Multiple authors have investigated this conjecture. It is proved in \cite{lee} that for $Q$ a rank 3, 2-complete\footnote{A quiver is called ``2-complete" if it is acyclic and has at least 2 arrows between every pair of vertices.} quiver, and $\alpha$ a real Schur root of the corresponding root system, $\mathcal{C}(\alpha)$ consists of one curve which has no self-intersections. The authors of \cite{Felikson_2018} reached the same conclusion for $Q$ any 2-complete quiver. The conjecture was proved for various other types of acyclic quivers in \cite{Nguyen} and \cite{Hong_2021}.

So far, the only curves in $\Gamma$ which have been studied are those in $\mathcal{C}(\alpha)$ for some real Schur root $\alpha$. Furthermore, the only question asked about a curve in $\mathcal{C}(\alpha)$ has been whether it contains any self-intersections. In this paper, we study the set $\mathcal{C}(\alpha)$ for $\alpha$ any positive real root of a quiver $Q$ of affine type $\tilde A(n-1,1)$. We show that if $\text{dim Ext}^1_{kQ}(M_\alpha,M_\alpha)=n \geq 0$, then $\mathcal{C}(\alpha)$ contains a curve $F(\alpha)$ with $n$ self-intersections. We also look at pairwise intersections between two curves $F(\alpha) \in \mathcal{C}(\alpha)$ and $F(\beta) \in \mathcal{C}(\beta)$. No homological information has previously been attached to an intersection between two different curves in $\Gamma$; we prove that, for certain pairs ($\alpha, \beta$), the number of intersections between $F(\alpha)$ and $F(\beta)$ is equal to $\text{dim Ext}^1_C (M_\alpha, M_\beta)$. Our main theorem is:

\begin{Thm}\label{mainthm0}
Let $Q$ be a quiver of type $\tilde A(n-1,1)$ and let $\text{Re}_{+}$ denote the set of all positive real roots of the associated root system. Then there exists an injective function
\[F:\text{Re}_{+} \to \bigcup_{\alpha \in \text{Re}_{+}} \mathcal{C}(\alpha)\]
with the properties 
\begin{enumerate}
\item $F(\alpha) \in \mathcal{C}(\alpha)$.
\item If $\alpha$ and $\beta$ are non-Schur roots in $\text{Re}_{+}$ such that $\beta- \alpha \in \{0\} \cup \text{Im}_{+}$\footnote{$\text{Im}_{+}=\{(\lambda,\dots,\lambda)\thinspace | \thinspace 0<\lambda \in \mathbb{Z}\}$.}, or if $\alpha$ and $\beta$ are Schur roots such that $\beta-\alpha \in \{(\lambda,\dots,\lambda)\thinspace | \thinspace 0 \leq \lambda <n\}$, then the number of pairwise intersections between $F(\alpha)$ and $F(\beta)$  is equal to the dimension of $\text{Ext}^{1}_{C}(M_{\alpha},M_{\beta})$.
\end{enumerate}
\end{Thm} 

\noindent
A corollary to this theorem is:

\begin{Thm} Let $\alpha \in \text{Re}_{+}$. The number of self-intersections of $F(\alpha)$ is equal to the dimension of $\text{Ext}^{1}_{kQ}(M_{\alpha},M_{\alpha})$.
\end{Thm} 

We will explicitly describe the curve $F(\alpha)$ and see that it shares many of the qualitative properties of the uniquely determined curve for $M_\alpha$ in the unpunctured surface $(\mathcal{S},\mathcal{M})$. The relationship between the curve we construct and the curve in $(\mathcal{S},\mathcal{M})$ will provide insight into why Theorem \ref{mainthm0} holds only for those particular pairs $(\alpha, \beta)$. 

The importance of understanding \cite{lee}'s correspondence goes beyond computing the dimensions of extension spaces in $kQ$-mod and $C$. The construction of the curves in $\Gamma$ relies explicitly on the expression of a dimension vector as a (positive) real root, while the same is not true of the curves in $(\mathcal{S},\mathcal{M})$. The fact that such curves can measure the dimensions of Ext spaces perhaps points to a link between the root system of $Q$ and the homological data of $kQ$-mod. Furthermore, the fact that this measurement may be explained through a connection between the surfaces of \cite{lee} and of \cite{FST} perhaps suggests a link between the root system and the other objects associated to $(\mathcal{S},\mathcal{M})$, such as the cluster algebra arising from it.

\subsection*{Acknowledgments} I thank Kyungyong Lee for introducing me to the representation theory of quivers and for many invaluable discussions on the subject. Also, I thank Fabian Haiden for pointing out the paper \cite{Schroll}.

\section{Background}
\subsection{Real Roots}
 
\noindent
Fix an algebraically closed field $k$ and an integer $n\geq 3$. A quiver $Q$ of affine type $\tilde A_{n-1}$ is one which has the underlying graph:
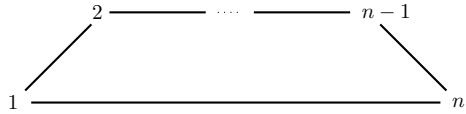
\vspace{2mm}
 \begin{center}
 \begin{tikzpicture}[scale=0.8]
\node[scale=.8] at (0,0){1};
\node[scale=.8] at (7.4,0){$n$};
\node[scale=.8] at (1.4,1.5){2};
\node[scale=.8] at (6.2,1.5){$n-1$};
\draw [-, thick] (0.2,0.2)--(1.3,1.3);
\draw [-, thick] (1.6,1.5)--(3.2,1.5);
\draw[dotted](3.4,1.5) -- (3.8,1.5);
\draw [-, thick] (4,1.5)--(5.6,1.5);
\draw [-, thick] (6.1,1.3)--(7.2,0.2);
\draw [-, thick] (0.3,0)--(7.1,0);
\end{tikzpicture}
\end{center}

Following \cite{Kac}, we define the positive real roots associated to this graph as follows. Associate to each vertex $i$, $1 \leq i \leq n$, the $n$-tuple $\alpha_i=(\delta_{i1},\dots,\delta_{in}) \in \mathbb{Z}^n$, where $\delta_{ij}$ is the Kronecker delta. Call $\alpha_i$ the $i$th simple root. Define a symmetric bilinear form on $\mathbb{Z}^n$ by 
\[
(\alpha_i,\alpha_j)=\delta_{ij}-\frac{1}{2}b_{ij}, 
\]
where $b_{ij}$ is the number of edges between the vertices $i$ and $j$ (if $i=j$, $b_{ij}=0$). Define the $i$th simple reflection $s_i$ by the matrix whose $j$th column is 
\[
s_i(\alpha_j)=\alpha_j-2(\alpha_i,\alpha_j)\alpha_i.
\]
 
In what follows, an explicit description of the simple reflection $s_i$ will be useful. Observe that 
\[
s_i(\alpha_j) = 
  \begin{cases}
    \alpha_i+\alpha_j & \text{if } i-j=\pm 1 \text{ or } i-j=\pm (n-1) \\
    -\alpha_i & \text{if } i=j \\
    \alpha_j & \text{otherwise}
  \end{cases} 
\]
If $i=1$, then the only $j$, $1\leq j \leq n$, which fall under the first case are $j=i+1$ and $j=n$. If $i=n$, then the only $j$ which fall under the first case are $j=i-1$ and $j=1$. If $1<i<n$, then the only $j$ which fall under the first case are $j=i-1$ and $j=i+1$. In terms of the columns of $s_i$, this means that if $i=1$, then the $n$th and the $(i+1)$st columns are equal to $\alpha_i+\alpha_n$ and $\alpha_i+\alpha_{i+1}$, respectively; if $i=n$, then the $1$st and the $(i-1)$st columns are equal to $\alpha_i+\alpha_1$ and $\alpha_i+\alpha_{i-1}$, respectively; if $1<i<n$, then the $(i-1)$st column is $\alpha_1+\alpha_{i-1}$ and the $(i+1)$st column is $\alpha_i+\alpha_{i+1}$. The columns which are equal to $\alpha_i+\alpha_j$ we shall call ``adjacent" to the $i$th column $s_i(\alpha_i)$. The adjacent columns have zeros for every component except the $i$th and $j$th, which are 1.

It is not hard to see that the rows of the matrix for $s_i$ are the same as those of the identity matrix, except for the $i$th row, which has $-1$ for the $i$th component, 1 for the two adjacent components, and 0s elsewhere. Therefore, for any root $\alpha$, $s_i(\alpha)$ differs from $\alpha$ only by the $i$th component. If we denote the $i$th component of $\alpha$ by $\alpha^i$ and the two adjacent components by $\alpha^{i-1}$ and $\alpha^{i+1}$, then the $i$th component of $s_i(\alpha)$ is $\alpha^{i-1}+\alpha^{i+1}-\alpha^{i}$.

\begin{Def}
A real root is an element of $\mathbb{Z}^n$ of the form $s_{i_1} \cdots s_{i_j}  (\alpha_k)$, for some $j$ simple reflections $s_{i_1}$ through $s_{i_j}$ and a simple root $\alpha_k$. A real root is called positive if all its components are nonnegative. 
\end{Def}  

\begin{Prop}\label{realroot} If $\alpha$ is a real root for the graph of type $\tilde A_{n-1}$, then $\alpha$ has one of the following forms:
\begin{enumerate}[1.] 
\item $(m,\dots,m,m+1,\dots,m+1,m,\dots,m)=(m^a,(m+1)^b,m^c) \quad $ 
\item $(m+1,\dots,m+1,m,\dots,m,m+1,\dots,m+1)=((m+1)^a,m^b,(m+1)^c)$
\item $-(m^a,(m+1)^b,m^c)$ or $-((m+1)^a,m^b,(m+1)^c)$
\end{enumerate}
for $m,a,b,c$ nonnegative integers such that $a+b+c=n$, $a+c \neq 0$, and $a+c \neq n$.
\end{Prop}

\begin{proof} Let us call the set of roots having these forms $\Delta$. Since $\alpha$ is obtained as a sequence of simple reflections applied to a simple root, and all of the simple roots are in $\Delta$, it suffices to prove that an application of a simple reflection to an element of $\Delta$ always gives another element of $\Delta$. So let $\mathbf{v}=(v^1,\dots,v^i,\dots,v^n) \in \Delta$ and let $s_i$ be an arbitrary simple reflection. The $i$th component of $s_i(\mathbf{v})$ is $v^{i-1}+v^{i+1}-v^{i}$. Since by assumption $v^{i-1}$, $v^{i}$, and $v^{i+1}$ are all equal, or exactly two of them are, $v^{i-1}+v^{i+1}-v^{i}$ must be one of the following options: 
 
\begin{center}
\begin{tabular}{ll} 
	Relation of $v^{i}$ to adjacent components & $i$th component of $s_i(\mathbf{v})$\\ 
	$v^{i}=v^{i-1}=v^{i+1}$ & $v^i$ \\
    $v^i=v^{i-1}-1=v^{i+1}-1$ & $v^i+2$ \\
    $v^i=v^{i-1}+1=v^{i+1}+1$ & $v^i-2$ \\
    $v^i=v^{i-1}-1=v^{i+1}$ & $v^i+1$ \\
    $v^i=v^{i-1}+1=v^{i+1}$ & $v^i-1$ \\
    $v^i=v^{i-1}=v^{i+1}-1$ & $v^i+1$ \\
    $v^i=v^{i-1}=v^{i+1}+1$ & $v^i-1$ \\
\end{tabular}
\end{center}

\noindent
Now with all the cases made explicit, it is easy to check that $s_i(\mathbf{v})\in \Delta$. Hence every real root is in $\Delta$.
\end{proof}

When we prove the main result of this paper, we will display a plane curve corresponding to any root of type 1 or 2. By the very definition of an admissible plane curve (to be given shortly), constructing a curve for $\mathbf{v}$ is equivalent to finding an expression for $\mathbf{v}$ of the form $s_{i_1} \cdots s_{i_j}  (\alpha_k)$. Hence the proof of Theorem \ref{mainthm0} implies that the roots of types 1 and 2 are (positive) real roots. It follows from Proposition \ref{realroot} that all positive real roots are of type 1 or 2, and, therefore, that Theorem \ref{mainthm0} holds for all elements of $\text{Re}_{+}$.

\subsection{Surfaces}
In this paper, we consider the quiver $Q$ of affine type $\tilde A_{n-1}$ which has the orientation: 

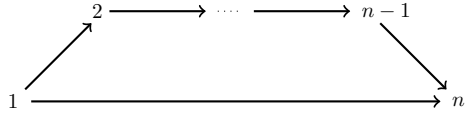
\vspace{2mm}
 \begin{center}
 \begin{tikzpicture}[scale=0.8]
\node[scale=.8] at (0,0){1};
\node[scale=.8] at (7.4,0){$n$};
\node[scale=.8] at (1.4,1.5){2};
\node[scale=.8] at (6.2,1.5){$n-1$};
\draw [->, thick] (0.2,0.2)--(1.3,1.3);
\draw [->, thick] (1.6,1.5)--(3.2,1.5);
   \draw[dotted](3.4,1.5) -- (3.8,1.5); 
\draw [->, thick] (4,1.5)--(5.6,1.5);
\draw [->, thick] (6.1,1.3)--(7.2,0.2);
\draw [->, thick] (0.3,0)--(7.1,0);
\end{tikzpicture}
\end{center} 

In \cite{FST}, $Q$ is said to be of type $\tilde A(n-1,1)$ to indicate that there are $n-1$ arrows in one direction and 1 in the other direction. A fact which will be used later is that the (real) Schur roots of this quiver are those of type 1 in which either $a$ or $c$ is 0, using the notation of Proposition \ref{realroot}. The triangulated marked surface corresponding to $Q$ (cf. section 2.1 of \cite{article}) is this annulus with $n-1$ marked points on the outer boundary and $1$ marked point on the inner boundary:

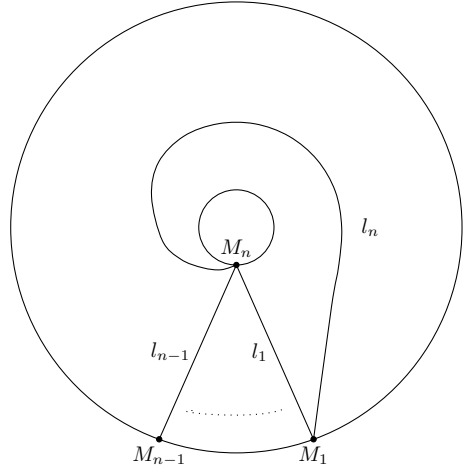
\vspace{2mm}

\begin{center}
\begin{tikzpicture}[pics/annulus/.style={code={
    \draw (0,0) circle (3cm)
   	circle (.5cm);
   	\draw[black] (0cm,-.5cm) -- (250:3cm)  node[midway,left, scale=.8]{$l_1$};
   	\draw[dotted] (256:2.5) to[out=-20,in=0](283:2.5);    
   	\draw[black] (0cm,-.5cm) -- (290:3cm)   node[midway,left, scale=.8]{$l_{n-1}$};
 	\draw [black] plot [smooth, tension=.6] coordinates { (270:.5) (290:.6)(310:.7) (330:.85) (350:1) (20:1.2) (40:1.35)(60:1.4)(80:1.4)(100:1.4) (120:1.4)(140:1.4) (160:1.4) (180:1.4) (200:1.45) (220:1.65) (250:3)}; 	 
   	\node[left, scale=.8] at (-2,0){$\thickspace l_n$};  
   	\draw[fill] (0cm,-.5cm) circle(1pt) node[above, scale=.8]{$M_n$}; 	 
   	\draw[fill] (250:3) circle(1pt) node[below, scale=.8]{$M_1$};	
   	\draw[fill] (290:3) circle(1pt)node[below, scale=.8]{$M_{n-1}$};}}]
   	\draw [scale=0.8]pic[xscale=-1] {annulus};
\end{tikzpicture}
\end{center}  

We will refer to this surface as $(\mathcal{S},\mathcal{M})$, where $\mathcal{M}$ denotes the set of boundary marked points, $\mathcal{M}=\{M_{1},\dots,M_n\}$. In this paper, a curve in $(\mathcal{S},\mathcal{M})$ is the homotopy class of an oriented noncontractible smooth curve which starts at some marked point and ends at another, not necessarily distinct, marked point. The number of intersections between two curves is, by definition, the minimal number of transverse intersections occurring between any two representatives of the respective homotopy classes. The curves which are part of the triangulation of $(\mathcal{S},\mathcal{M})$ shall be called arcs to avoid confusion, and labelled by $l_1,\dots,l_n$.

As mentioned in the Introduction, Proposition 4.2 of \cite{article} (see also the remark following it) states that, up to orientation, curves in $(\mathcal{S},\mathcal{M})$ are in bijection with the isomorphism classes of string modules over a Jacobian algebra which in our case is $kQ$, the path algebra of $Q$.  The bijection in Proposition 4.2 has the property that if a curve $\gamma$ in $(\mathcal{S},\mathcal{M})$ corresponds to an indecomposable $kQ$-module $M$, then the dimension vector of $M$ is equal to the $n$-tuple of integers whose $i$th component is the number of times $\gamma$ intersects $l_i$.

Fix a positive real root $\alpha$ of $Q$ and let $M_\alpha$ be the unique (up to isomorphism) module with dimension vector $\alpha$. In the course of this article, we will display a curve $\gamma$ in $(\mathcal{S},\mathcal{M})$ that intersects $l_i$ exactly $w$ times, where $w$ is  the $i$th component of $\alpha$. By Proposition 4.2, the module corresponding to $\gamma$ has dimension vector $\alpha$ -- hence it is $M_\alpha$. 

\begin{Exa} Consider the quiver $\tilde A(3,1)$: 

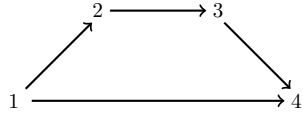
\vspace{2mm} 
 \begin{center}
 \begin{tikzpicture}[scale=0.8]
\node[scale=.8] at (0,0){1};
\node[scale=.8] at (4.7,0){$4$};
\node[scale=.8] at (1.4,1.5){$2$};
\node[scale=.8] at (3.4,1.5){$3$};
\draw [->, thick] (0.2,0.2)--(1.3,1.3);
\draw [->, thick] (1.6,1.5)--(3.2,1.5);
\draw [->, thick] (3.5,1.3)--(4.6,.2);
\draw [->, thick] (0.3,0)--(4.5,0);
\end{tikzpicture}
\end{center} 
$(\mathcal{S},\mathcal{M})$ in this case is:
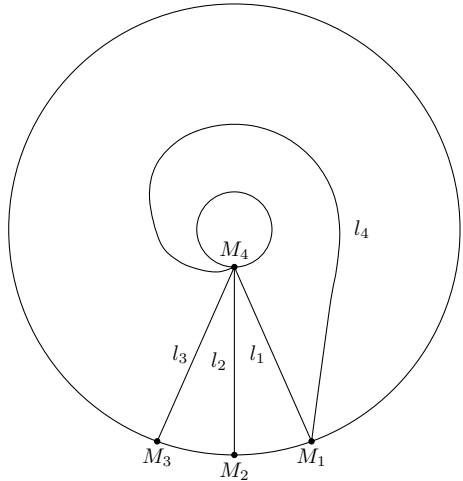
\vspace{2mm}

\begin{center}
\begin{tikzpicture}[pics/annulus2/.style={code={
   \draw (0,0) circle (3cm)
   	circle (.5cm);
   	\draw[black] (0cm,-.5cm) -- (250:3cm)  node[midway,left, scale=.8]{$l_1$}; 
   	\draw[black] (0cm,-.5cm) -- (290:3cm) node[midway, left, scale=.8]{$l_3$};
   	\draw [black] plot [smooth, tension=.6] coordinates { (270:.5) (290:.6)(310:.7) (330:.85) (350:1) (20:1.2) (40:1.35)(60:1.4)(80:1.4)(100:1.4) (120:1.4)(140:1.4) (160:1.4) (180:1.4) (200:1.45) (220:1.65) (250:3)}; 	 
   	\node[left, scale=.8] at (-1.9,0){$\thickspace l_4$}; 
   	\draw[fill] (0cm,-.5cm) circle(1pt) node[above, scale=.8]{$M_4$};
   	\draw[fill] (250:3) circle(1pt) node[below, scale=.8]{$M_1$};			
   	\draw[fill] (0cm,-3cm) circle(1pt) node[below, scale=.8]{$M_{2}$};  
   	\draw[fill] (290:3) circle(1pt)node[below, scale=.8]{$M_3$};
   	\draw[black] (0cm,-.5cm) -- (0cm,-3cm) node[midway,left, scale=.8]{$l_2$};}}]
   	\draw [scale=.8]pic[xscale=-1] {annulus2};
\end{tikzpicture}
\end{center}
The root $\alpha=(2,2,1,2)$ corresponds to the curve: 
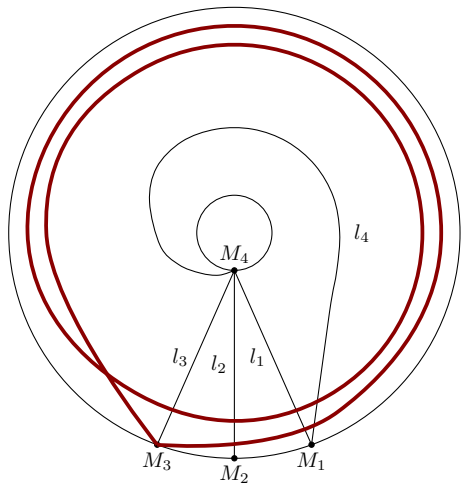
\vspace{2mm}
\begin{center}
\begin{tikzpicture}[pics/annulus3/.style={code={
  \draw (0,0) circle (3cm)
   	circle (.5cm);
   	\draw[black] (0cm,-.5cm) -- (250:3cm)  node[midway,left, scale=.8]{$l_1$};
   	\draw[black] (0cm,-.5cm) -- (290:3cm) node[midway, left, scale=.8]{$l_3$};
   	\draw [black] plot [smooth, tension=.6] coordinates { (270:.5) (290:.6)(310:.7) (330:.85) (350:1) (20:1.2) (40:1.35)(60:1.4)(80:1.4)(100:1.4) (120:1.4)(140:1.4) (160:1.4) (180:1.4) (200:1.45) (220:1.65) (250:3)}; 	 
   	\node[left, scale=.8] at (-1.9,0){$\thickspace l_4$}; 
   	\draw[fill] (0cm,-.5cm) circle(1pt) node[above, scale=.8]{$M_4$};
   	\draw[fill] (250:3) circle(1pt) node[below, scale=.8]{$M_1$};		
   	\draw[fill] (0cm,-3cm) circle(1pt) node[below, scale=.8]{$M_{2}$};  
   	\draw[fill] (290:3) circle(1pt)node[below, scale=.8]{$M_3$};
   	\draw[black] (0cm,-.5cm) -- (0cm,-3cm) node[midway,left, scale=.8]{$l_2$};
   	\draw [DarkRed, line width= .5mm] plot [smooth, tension=0.75] coordinates { (290:3) (360:2.5) (60:2.5) (120:2.5) (180:2.5) (240:2.5) (300:2.5) (360:2.75) (60:2.75) (120:2.75) (180:2.75) (240:2.75) (290:3) };}}]
   	\draw [scale=.8]pic[xscale=-1] {annulus3};
\end{tikzpicture}
\end{center}
\end{Exa}

The surface which \cite{lee} defines for the quiver $Q$ shall be called $\Gamma$. $\Gamma$ can be modelled (such as in \cite{Felikson_2018}) as the upper half-plane with a ``basepoint" $B$ on the real axis and $n$ marked points $\{p_1,\dots,p_n\}$, such that $\text{Im}(p_1)=\cdots=\text{Im}(p_n)$ and such that $p_i$ is the endpoint of a vertical ray $r_i$:

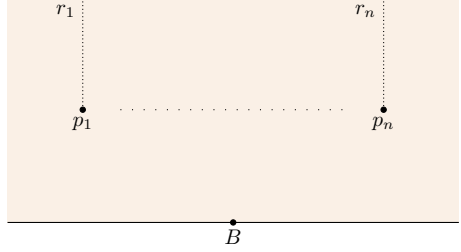
\vspace{2mm}

\begin{center}
\begin{tikzpicture}[scale=1]
\fill[Linen] (-1,0) rectangle (5,3);
\draw[fill] (0,1.5) circle(1 pt) node[below, scale=.8] {$p_1$}; 
\draw[loosely dotted] (.5,1.5) -- (3.5,1.5);
\draw[fill] (4,1.5) circle(1 pt) node[below, scale=.8] {$p_n$};
\draw[fill] (2,0) circle(1 pt) node[below, scale=.8] {$B$}; 
\draw[-] (-1,0) -- (5,0);
\draw[densely dotted] (0,1.5) -- (0,3) node[below left, scale=.8]{$r_1$}; 
\draw[densely dotted] (4,1.5) -- (4,3) node[below left, scale=.8]{$r_n$}; 
\end{tikzpicture}
\end{center}

An admissible curve, or simply a curve, in $\Gamma$ is the homotopy class of an oriented smooth curve which begins at some $p_i$ and ends at $B$. As with curves in $(\mathcal{S},\mathcal{M})$, intersections are considered up to homotopy. Fix  $\alpha\in \text{Re}_{+}$. Since $\alpha$ is real, it has at least one expression as $s_{i_1} \cdots s_{i_j} (\alpha_k)$. An element of $\mathcal{C}(\alpha)$ is a curve which begins at the point $p_k$ and intersects the rays $r_{i_j}$ through $r_{i_1}$, in right to left order, as it increases. $\mathcal{C}(\alpha)$ is usually quite a large set, since, first of all, the expression $s_{i_1} \cdots s_{i_j} (\alpha_k)$ for a positive real root $\alpha$ is typically not unique. Furthermore, even when one expression is fixed upon, it may be possible to construct several non-homotopic curves which start at $p_k$ and intersect the rays $r_{i_j}$ through $r_{i_1}$. 

\begin{Exa} Consider the quiver $\tilde A(2,1)$. Below are three curves in $\mathcal{C}(\alpha)$, where $\alpha=(2,2,1)$. Note that the first two curves are obtained through the same sequence of simple reflections.

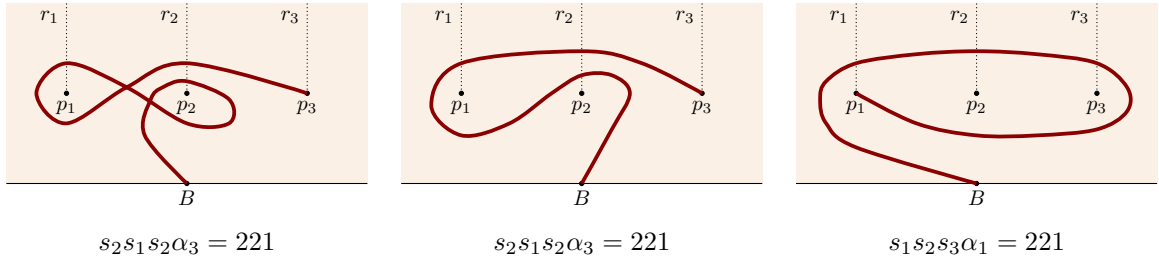
\vspace{2mm}

\begin{tikzpicture}[scale=.8]
\fill[Linen] (-1,0) rectangle (5,3);
\draw[fill] (0,1.5) circle(1 pt) node[below, scale=.8] {$p_1$}; 
\draw[fill] (2,1.5) circle(1 pt) node[below, scale=.8] {$p_2$};
\draw[fill] (4,1.5) circle(1 pt) node[below, scale=.8] {$p_3$};
\draw[fill] (2,0) circle(1 pt) node[below, scale=.8] {$B$}; 
\draw[-] (-1,0) -- (5,0);
\draw[densely dotted] (0,1.5) -- (0,3) node[below left, scale=.8]{$r_1$}; 
\draw[densely dotted] (2,1.5) -- (2,3) node[below left, scale=.8]{$r_2$}; 
\draw[densely dotted] (4,1.5) -- (4,3) node[below left, scale=.8]{$r_3$}; 
\draw [DarkRed, line width=0.5mm] plot [smooth, tension=0.6] coordinates { (4,1.5) (2,2) (1,1.6) (0,1) (-.5,1.5) (0,2) (1, 1.6) (2,1)(2.7,1) (2.7,1.4) (2,1.7) (1.6,1.6) (1.4,1.4) (1.3, .8) (2,0)};
\node at (2,-1) {$s_2s_1s_2\alpha_3=221$};
\end{tikzpicture}
\hspace{2mm}
\begin{tikzpicture}[scale=.8]
\fill[Linen] (-1,0) rectangle (5,3);
\draw[fill] (0,1.5) circle(1 pt) node[below, scale=.8] {$p_1$}; 
\draw[fill] (2,1.5) circle(1 pt) node[below, scale=.8] {$p_2$};
\draw[fill] (4,1.5) circle(1 pt) node[below, scale=.8] {$p_3$};
\draw[fill] (2,0) circle(1 pt) node[below, scale=.8] {$B$}; 
\draw[-] (-1,0) -- (5,0);
\draw[densely dotted] (0,1.5) -- (0,3) node[below left, scale=.8]{$r_1$}; 
\draw[densely dotted] (2,1.5) -- (2,3) node[below left, scale=.8]{$r_2$}; 
\draw[densely dotted] (4,1.5) -- (4,3) node[below left, scale=.8]{$r_3$};
\draw [DarkRed,line width=0.5mm] plot [smooth, tension=0.6] coordinates { (4,1.5) (3,2) (2,2.2) (0,2) (-.5,1.3) (0,.8) (.8,1) (1.5,1.5) (2,1.8) (2.5,1.8) (2.8,1.5) (2.5,.9) (2,0)};
\node at (2,-1) {$s_2s_1s_2\alpha_3=221$}; 
\end{tikzpicture}
\hspace{2mm}
\begin{tikzpicture}[scale=.8]
\fill[Linen] (-1,0) rectangle (5,3);
\draw[fill] (0,1.5) circle(1 pt) node[below, scale=.8] {$p_1$}; 
\draw[fill] (2,1.5) circle(1 pt) node[below, scale=.8] {$p_2$};
\draw[fill] (4,1.5) circle(1 pt) node[below, scale=.8] {$p_3$};
\draw[fill] (2,0) circle(1 pt) node[below, scale=.8] {$B$}; 
\draw[-] (-1,0) -- (5,0);
\draw[densely dotted] (0,1.5) -- (0,3) node[below left, scale=.8]{$r_1$}; 
\draw[densely dotted] (2,1.5) -- (2,3) node[below left, scale=.8]{$r_2$}; 
\draw[densely dotted] (4,1.5) -- (4,3) node[below left, scale=.8]{$r_3$}; 
\draw [DarkRed, line width=0.5mm] plot [smooth, tension=0.6] coordinates { (0,1.5) (1,1) (2,.8) (3,.8) (4,.9) (4.5,1.2) (4.5, 1.6) (4,2) (3,2.15) (2,2.2) (1,2.15) (0,2) (-.5,1.7) (-.6,1.4) (-.5,1)(0,.6) (2,0)}; 
\node at (2,-1) {$s_1s_2s_3\alpha_1=221$};
\end{tikzpicture}
\end{Exa}

An admissible curve $F$ corresponds to exactly one root $\alpha$ and to exactly one particular expression $s_{i_1} \cdots s_{i_j}  (\alpha_k)$ of $\alpha$. The two associations  $F \mapsto \alpha$  and $F \mapsto s_{i_1} \cdots s_{i_j}  (\alpha_k)$ are surjective functions which we will denote, respectively, by $R(F)=\alpha$ and $S(F)=s_{i_1} \cdots s_{i_j}  (\alpha_k)$. Then the following are equivalent: $F \in \mathcal{C}(\alpha)$, $R(F)=\alpha$, and $S(F)$ is an expression of $\alpha$.  

\subsection{Extension Spaces}

In this section, let $Q$ be any finite, acyclic quiver. We refer to \cite{CALDERO2006983}, section 2.3,  for the general construction of the cluster category $C$ of $kQ$-mod. The indecomposable objects of $kQ$-mod form a subset of the indecomposable objects of $C$; Proposition 1 in \cite{CALDERO2006983} gives a formula relating the Ext spaces in the two categories. Let $M, N$ be indecomposable objects in $kQ$-mod, and therefore also in $C$. Then (\cite{CALDERO2006983}),

\begin{equation}\label{ext}
\text{Ext}^1_{C}(M,N)=\text{Ext}^1_{kQ}(M,N) \oplus D\text{Ext}^1_{kQ}(N,M), 
\end{equation} 
where $ D\text{Ext}^1_{kQ}(N,M)$ is the dual vector space of  $\text{Ext}^1_{kQ}(N,M)$. In particular, taking dimensions of $k$-vector spaces gives

\begin{equation}\label{dim}
\text{dim Ext}^1_{C}(M,N)=\text{dim Ext}^1_{kQ}(M,N)+\text{dim Ext}^1_{kQ}(N,M) =\text{dim Ext}^1_{C}(N,M).
\end{equation}
When $N=M$,

\begin{equation}\label{dim2}
\text{dim Ext}^1_{C}(M,M)=2\text{dim Ext}^1_{kQ}(M,M).
\end{equation}  

Let $C(\mathcal{S},\mathcal{M})$ denote the cluster category associated to $(\mathcal{S},\mathcal{M})$. When $Q$ is acyclic, $C(\mathcal{S},\mathcal{M})$ is equivalent as a category to $C$.  Let $\gamma_\alpha$ and $\gamma_\beta$ be two curves in $(\mathcal{S},\mathcal{M})$ with corresponding indecomposable $kQ$-modules $M_\alpha$ and $M_\beta$. As just described, $M_\alpha$ and $M_\beta$ are indecomposable objects in $C(\mathcal{S},\mathcal{M})$. In \cite{zhang2012cotorsion}, $\text{Int}(\gamma_\alpha, \gamma_\beta)$ denotes the minimal number of intersections occurring between any two representatives of the respective homotopy classes of $\gamma_\alpha$ and $\gamma_\beta$. We recall Theorem 3.4 of \cite{zhang2012cotorsion}, which states that

\begin{equation}\label{int1}
\text{Int}(\gamma_\alpha,\gamma_\beta)=\text{dim Ext}^1_{C(S,M)}(M_\alpha, M_\beta).
\end{equation}
Combining this formula with \ref{dim} gives 

\begin{equation}\label{int2}
\text{Int}(\gamma_\alpha,\gamma_\beta)=\text{dim Ext}^1_{kQ}(M_\alpha,M_\beta)+\text{dim Ext}^1_{kQ}(M_\beta,M_\alpha).
\end{equation}
In particular, $\text{Int}(\gamma_\alpha, \gamma_\alpha)$ is to be interpreted as the number of pairwise intersections between two homotopic copies of $\gamma_\alpha$, which is evidently double the ordinary self-intersection number of $\gamma_\alpha$, which we will denote in this paper by $\text{Int}(\gamma_\alpha)$. Hence by \ref{dim2} and \ref{int2},

\begin{equation}\label{int3}
\text{Int}(\gamma_\alpha)=\text{dim Ext}^1_{kQ}(M_\alpha,M_\alpha).
\end{equation}

It is also true for an admissible plane curve $F$ that $\text{Int}(F,F)=2\text{Int}(F)$. 

Now let $Q$ once again denote a quiver of type $\tilde A(n-1,1)$. Because of equation \ref{int1} we can reformulate Theorem \ref{mainthm0} as:

\begin{Thm} \label{mainthm} Fix $n \geq 3$ and let $\text{Re}_{+}$ be the set of positive real roots of the root system associated to $Q$. For each $\alpha \in \text{Re}_{+}$ let $\gamma_\alpha$ be the unique curve for $\alpha$ in $(\mathcal{S},\mathcal{M})$. Let $\mathcal{P}$ denote the set of admissible plane curves in $\Gamma$. Then there is an injective function
\[F:\text{Re}_{+} \to \mathcal{P}\]
with the properties
\begin{enumerate}
\item $R(F(\alpha))=\alpha$.
\item If $\alpha$ and $\beta$ are non-Schur roots such that $\beta- \alpha \in \{0\} \cup \text{Im}_{+}$, or if $\alpha$ and $\beta$ are Schur roots such that $\beta-\alpha \in \{(\lambda,\dots,\lambda)\thinspace | \thinspace 0\leq \lambda <n\}$, then $\text{Int} (F(\alpha),F(\beta))=\text{Int}(\gamma_\alpha,\gamma_\beta)= \text{dim Ext}^{1}_{C}(M_{\alpha},M_{\beta})=\text{dim Ext}^1_C(M_\beta,M_\alpha)$.
\end{enumerate} 
\end{Thm}
Equation \ref{int3} then establishes the corollary for the case $\beta-\alpha = 0$:

\begin{Thm} For $\alpha \in \text{Re}_{+}$, $\text{Int }(F(\alpha))=\text{Int }(\gamma_\alpha)=\text{dim Ext}^1_{kQ}(M_\alpha,M_\alpha)$.
\end{Thm}

Let $Q^{\text{op}}$ denote the quiver all of whose arrows are reversed. Then the marked surface associated to $Q^{\text{op}}$ is $(\mathcal{S},\mathcal{M})^{\text{op}}$:

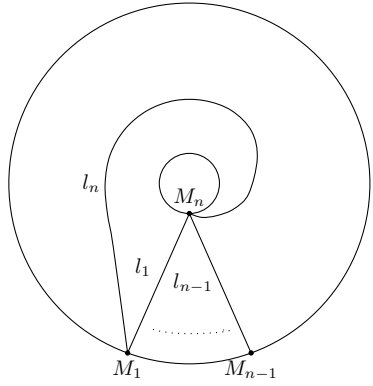
\begin{center}
\begin{tikzpicture}[scale=0.8]
    \draw (0,0) circle (3cm)
   	circle (.5cm);
   	\draw[black] (0cm,-.5cm) -- (250:3cm)  node[midway,above left, scale=.8]{$l_1$};
   	\draw[dotted] (256:2.5) to[out=-20,in=0](283:2.5);    
   	\draw[black] (0cm,-.5cm) -- (290:3cm)   node[midway,left, scale=.8]{$l_{n-1}$};
 	\draw [black] plot [smooth, tension=.6] coordinates { (270:.5) (290:.6)(310:.7) (330:.85) (350:1) (20:1.2) (40:1.35)(60:1.4)(80:1.4)(100:1.4) (120:1.4)(140:1.4) (160:1.4) (180:1.4) (200:1.45) (220:1.65) (250:3)}; 	 
   	\node[right, scale=.8] at (-2,0){$\thickspace l_n$};  
   	\draw[fill] (0cm,-.5cm) circle(1pt) node[above, scale=.8]{$M_n$}; 	 
   	\draw[fill] (250:3) circle(1pt) node[below, scale=.8]{$M_1$};	
  
   	\draw[fill] (290:3) circle(1pt)node[below, scale=.8]{$M_{n-1}$};
\end{tikzpicture}
\end{center}  

If $\alpha$ is a positive real root of $Q$, and $\gamma_\alpha$ is the curve for $\alpha$ is $(\mathcal{S},\mathcal{M})$, then $\gamma_\alpha^{\text{op}}$  is the curve in $(\mathcal{S},\mathcal{M})^{\text{op}}$ for $\alpha$. It is clear that for any $\alpha$ and $\beta$, $\text{Int}(\gamma_\alpha,\gamma_\beta)=\text{Int}(\gamma_\alpha^{\text{op}},\gamma_\beta^{\text{op}})$. 

We will prove Theorem \ref{mainthm} separately for three disjoint subsets of $\text{Re}_{+}$, whose union is $\text{Re}_{+}$: non-Schur roots of type 1, non-Schur roots of type 2, and Schur roots. In each case, we will first construct $F(\alpha)$ and show that it has the property $(1)$ stated in the theorem. Then, we will describe $\gamma_\alpha^{\text{op}}$ and prove the property $(2)$ by proving that $\text{Int} (F(\alpha),F(\beta))=\text{Int}(\gamma_\alpha^{\text{op}},\gamma_\beta^{\text{op}})$.

\section{Proof of Theorem \ref{mainthm} for Roots of Type 1} Fix $n \geq 3$ and nonnegative integers $m,a,b,c$ such that $a+b+c=n$, $a+c \neq n$, $a\geq 1$, and $c\geq 1$. Then 
\[
\alpha=(m,\dots,m,m+1,\dots,m+1,m,\dots,m)=(m^a,(m+1)^b,m^c)
\]
is a positive real root of type 1 which is not a Schur root, and all non-Schur roots of type 1 are of this form. Note that the first component which is $m+1$ is the $(a+1)$th component and the last component which is $m+1$ is the $(a+b)$th component. 

To construct $F(\alpha)$, go through the steps:
\begin{enumerate}
\item Let $F(\alpha)$ start at $p_{a+b}$ and, going to the right, intersect the rays $r_x$ for $a+b < x < n$. 
\item Going to the left, intersect the rays $r_x$ for $1\leq x\leq n$.
\item Going to the right, intersect the rays $r_x$ for $2\leq x\leq a+b$.
\item Repeat steps (1) through (3) $m-1$ times (for a total of $m$ times).
\item If $a+1 \leq a+b-1$, then intersect the rays $r_x$ for $a+1 \leq x\leq a+b-1$, going to the left. Otherwise, end the curve at $B$ without intersecting any rays.
\end{enumerate}

\begin{Exa} If $\alpha=1221$, then $a=1$, $b=2$, and $c=1$. $F(\alpha)$ is shown below with the different steps labelled.
\vspace{2mm}
\begin{center} 
\begin{tikzpicture}[scale=1.5]  
\fill[Linen] (-1,0) rectangle (7,3);
\draw[-] (-1,0) -- (7,0);
\draw[densely dotted] (0,1.5) -- (0,3) node[below left, scale=.8]{$r_1$}; 
\draw[densely dotted] (2,1.5) -- (2,3) node[below left, scale=.8]{$r_2$}; 
\draw[densely dotted] (4,1.5) -- (4,3) node[below left, scale=.8]{$r_3$}; 
\draw[densely dotted] (6,1.5) -- (6,3) node[below left, scale=.8]{$r_4$}; 
\draw [DarkRed,dash pattern = on 3mm off .5mm, line width=.5mm] plot [smooth, tension=0.6] coordinates {(4, 1.5) (5,1.2) (5.7, 1) (6.2, 1.1) (6.4, 1.4) (6.3, 1.8) (6,2) (5,2.2) (3.5,2.3)(2,2.3) (1,2.2) (0,1.9)};
\node[scale=0.8] at (5,2.4) {$(2)$};
\draw [DarkRed,densely dashed, line width=.5mm] plot [smooth, tension=0.6] coordinates {(0,1.9) (-.4,1.6) (-.45,1) (.2,.9) (1,1.6) (2,2) (3,2)(4,1.9)};
\node[scale=0.8] at (3.5,1.8) {$(3)$};  
\draw [DarkRed,  line width=.5mm] plot [smooth, tension=0.6] coordinates {(4,1.9) (4.5,1.8) (4.7,1.4)(4.5,1) (3.5,1) (2.7,1.4) (2,1.7) (1.6,1.6) (1.4,1.3) (1.6, .8) (3,0)};
\node[scale=0.8] at (1.3,.9) {$(5)$}; 
\draw[fill] (3,0) circle(1 pt) node[below, scale=.8] {$B$}; 
\draw[fill] (0,1.5) circle(1 pt) node[below, scale=.8] {$p_1$}; 
\draw[fill] (2,1.5) circle(1 pt) node[below, scale=.8] {$p_2$};
\draw[fill] (4,1.5) circle(1 pt) node[below, scale=.8] {$p_3$};
\draw[fill] (6,1.5) circle(1 pt) node[below, scale=.8] {$p_4$};
\draw[fill] (3,0) circle(1 pt) node[below, scale=.8] {$B$};  
\end{tikzpicture}
\end{center}  
If $\alpha=11211$, then $F(\alpha)$ is:
\vspace{2mm}
\begin{center} 
\begin{tikzpicture}[scale=1.5]
\fill[Linen] (-1,0) rectangle (9,3);
\draw[-] (-1,0) -- (9,0);
\draw[densely dotted] (0,1.5) -- (0,3) node[below left, scale=.8]{$r_1$}; 
\draw[densely dotted] (2,1.5) -- (2,3) node[below left, scale=.8]{$r_2$}; 
\draw[densely dotted] (4,1.5) -- (4,3) node[below left, scale=.8]{$r_3$}; 
\draw[densely dotted] (6,1.5) -- (6,3) node[below left, scale=.8]{$r_4$}; 
\draw[densely dotted] (8,1.5) -- (8,3) node[below left, scale=.8]{$r_5$}; 
\draw [DarkRed, loosely dashed, line width=.5mm] plot [smooth, tension=0.6] coordinates {(4, 1.5) (5,1.7) (6,1.8) (6.8, 1.5)(7.4,1.2)(7.4,1.2)  (7.9,1)(8.3,1.15) (8.4,1.5)(8,1.8)} ;
\node[scale=0.8] at (5,1.5) {$(1)$};  
\draw [DarkRed, dash pattern = on 3mm off .5mm, line width=.5mm] plot [smooth, tension=0.6] coordinates {  (8,1.8)(7,2) (5,2.2) (3,2.2)(1,2) (0,1.7)} ;
\node[scale=0.8] at (7.2,2.15) {$(2)$};  
\draw [DarkRed, densely dashed, line width=.5mm] plot [smooth, tension=0.6] coordinates {(0,1.7)(-.2,1.5)(-.3,1.1)(.1,.8)(.7,1) (1.3,1.35)(2,1.7)(3,1.9)(4,1.9)} ; 
\draw [DarkRed,line width=.5mm] plot [smooth, tension=0.6] coordinates {(4,1.9)(4.6,1.5)(4.2,.7) (4,0) };
\node[scale=0.8] at (3,1.7) {$(3)$};  
\node[scale=0.8] at (4,.6) {$(5)$};  
\draw[fill] (4,0) circle(1 pt) node[below, scale=.8] {$B$}; 
\draw[fill] (0,1.5) circle(1 pt) node[below, scale=.8] {$p_1$}; 
\draw[fill] (2,1.5) circle(1 pt) node[below, scale=.8] {$p_2$};
\draw[fill] (4,1.5) circle(1 pt) node[below, scale=.8] {$p_3$};
\draw[fill] (6,1.5) circle(1 pt) node[below, scale=.8] {$p_4$};
\draw[fill] (8,1.5) circle(1 pt) node[below, scale=.8] {$p_5$};  
\end{tikzpicture}
\end{center}
\end{Exa}

\begin{Thm}[\ref{mainthm}(1)] Fix $n \geq 3$ and nonnegative integers $m,a,b,c$ such that $a+b+c=n$, $a+c \neq n$, $a\geq 1$, and $c\geq 1$, and consider the root $\alpha=(m^a,(m+1)^b,m^c)$. Then $R(F(\alpha))=\alpha$. 
\end{Thm} 

\begin{proof} Recall that the $i$th reflection $s_i$ acts on the root $\mathbf{v}=(v^1,\dots,v^{i-1},v^i,v^{i+1},\dots,v^n)$ by $s_i(\mathbf{v})=(v^1,\dots,v^{i-1}, v^{i-1}+v^{i+1}-v^i,v^{i+1},\dots,v^n)$. It follows that the steps to construct $F(\alpha)$ correspond to the following steps to construct the root $R(F(\alpha))$:
\begin{enumerate}
\item Apply the reflection $s_{n-1}\cdots s_{a+b+1}$ to the simple root $\alpha_{a+b}$. The result is the root $\sum_{j=a+b}^{n-1} \alpha_j$.
\item Now apply the reflection $s_1\cdots s_n$. The result is $\alpha_1+\sum_{j=a+b+1}^{n} \alpha_j$.
\item Now apply the reflection $s_{a+b}\cdots s_2$. The result is the root $(1,\dots,1)+\alpha_{a+b}$.
\item The result is $(m,\dots,m)+\alpha_{a+b}$.
\item If $a+1\leq a+b-1$, apply the reflection $s_{a+1}\cdots s_{a+b-1}$. The result now is $(m,\dots,m)+\sum_{j=a+1}^{a+b} \alpha_j$, which is exactly $\alpha$. Otherwise, $a+1>a+b-1$ which happens if and only if $b=1$. Therefore, the root constructed by step $(4)$ is $(m,\dots,m)+\alpha_{a+b}=(m,\dots,m)+\alpha_{a+1}$, which is indeed $\alpha$ when $b=1$.
\end{enumerate}

\end{proof}

Now, the curve $\gamma_\alpha^{\text{op}}$ in $(\mathcal{S},\mathcal{M})^{\text{op}}$ is obtained through the steps:

\begin{enumerate} 
\item Let $\gamma_\alpha^{\text{op}}$ start at $M_a$ and, going counter-clockwise around the inner boundary, intersect $l_a$ in $m$ points not including $M_a$.
\item End at $M_{a+b+1}$. (If $a+b+1=n$, end at $M_1$.)
\end{enumerate}

\begin{Exa} Let $\alpha=121$. The different steps are labelled and correspond to solid and dashed curve segments.
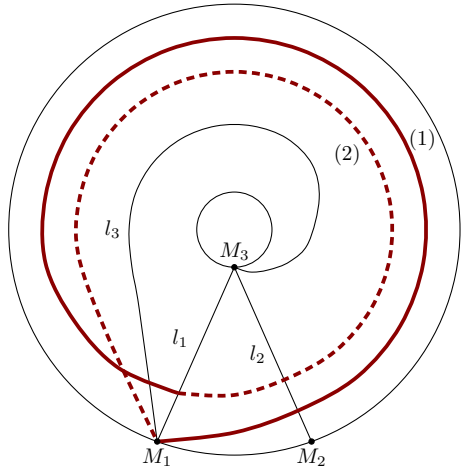
\vspace{2mm}
\begin{center}
\begin{tikzpicture}[scale=1]
   \draw (0,0) circle (3cm)
   	circle (.5cm);
   	\draw[black] (0cm,-.5cm) -- (250:3cm)  node[midway,above left, scale=.8 ]{$l_1$};
   	\draw[black] (0cm,-.5cm) -- (290:3cm) node[midway, left, scale=.8]{$l_2$};
   	\draw [black] plot [smooth, tension=.6] coordinates { (270:.5) (290:.6)(310:.7) (330:.85) (350:1) (20:1.2) (40:1.35)(60:1.4)(80:1.4)(100:1.4) (120:1.4)(140:1.4) (160:1.4) (180:1.4) (200:1.45) (220:1.65) (250:3)}; 	 
   	\node[right, scale=.8] at (-1.9,0){$\thickspace l_3$}; 
  \draw [DarkRed, line width=0.5mm] plot [smooth, tension=0.6] coordinates { (250:3) (270:2.7) (290:2.55) (310:2.55) (330:2.55) (350:2.55) (10:2.55) (30:2.55) (50:2.55) (70:2.55) (90:2.55) (110:2.55) (130:2.55) (150:2.55) (170:2.55) (190:2.55) (210:2.45) (230:2.4) (251:2.3)};  
  \draw [DarkRed, densely dashed, line width=0.5mm] plot [smooth, tension=0.6] coordinates { (251:2.3) (270:2.2)(290:2.1) (310:2.1) (330:2.1) (350:2.1) (10:2.1)(30:2.1)(50:2.1) (70:2.1)(90:2.1)(110:2.1)(130:2.1)(150:2.1) (170:2.1) (190:2.1) (220:2.2 ) (250:3) };
 \node[scale=0.8] at (2.5,1.2) {$(1)$}; 
   \node[scale=0.8] at (1.5,1) {$(2)$}; 
   	\draw[fill] (0cm,-.5cm) circle(1pt) node[above, scale=.8]{$M_3$};
   	\draw[fill] (250:3) circle(1pt) node[below, scale=.8]{$M_1$};		
   	\draw[fill] (290:3) circle(1pt)node[below, scale=.8]{$M_2$};	
\end{tikzpicture}
\end{center}
\end{Exa}

We now give another description of $\gamma_\alpha^{\text{op}}$ which will be used to prove Theorem \ref{type12}. Let $m>0$. Choose a point $w_i$ in $l_i - \{M_n,M_i\}$. Define $\text{Sp}_m (M_i,w_i)$ to be the curve segment (the ``spiral") with initial point $M_i$ and terminal point $w_i$ which, going counter-clockwise around the inner boundary, intersects $l_i$ in $m$ points not including $M_i$ (and intersects itself zero times). Illustrated below for $m=3$:

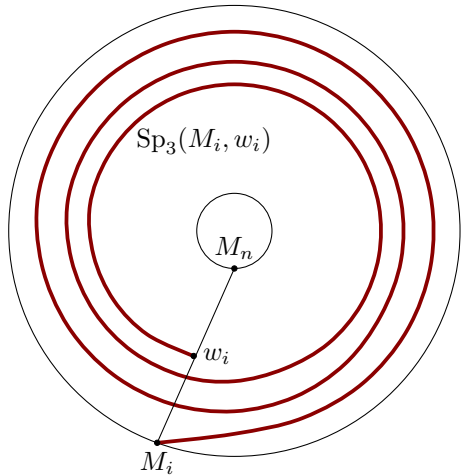
\begin{center}
\begin{tikzpicture}[scale=1]
\draw (0,0) circle (3cm)
   	circle (.5cm); 
\draw [DarkRed, line width=.5mm] plot [smooth, tension=0.6] coordinates { (250:3) (270:2.7) (290:2.65) (310:2.65) (330:2.65) (350:2.65) (10:2.65) (30:2.65) (50:2.65) (70:2.65) (90:2.65) (110:2.65) (130:2.65) (150:2.65) (170:2.65)(190:2.6) (210:2.55) (230:2.5) (250:2.45) (270:2.4) (290:2.35) (310:2.3) (330:2.25) (350:2.25) (10:2.25) (30:2.25) (50:2.25)(70:2.25)(90:2.25)(110:2.25)(130:2.25) (150:2.25) (170:2.25) (190:2.2) (210:2.15) (230:2.1)  (251:2.05) (270:2)(290:1.95) (310:1.95) (330:1.95) (350:1.95) (10:1.95)(30:1.95)(50:1.95) (70:1.95) (90:1.95)(110:1.95)(130:1.95) (150:1.95) (170:1.95)(190:1.9)(210:1.85)(230:1.8)(252:1.75)}; 
 \draw[fill] (250:3) circle(1pt) node[below, scale=1]{$M_i$};		
 \draw[black] (0cm,-.5cm) -- (250:3cm);
 \draw[fill] (252:1.75)circle(1pt) node[right, scale=1]{$w_i$};
 \draw[fill] (0cm,-.5cm) circle(1pt) node[above, scale=1]{$M_n$}; 
  \node at (-.4,1.2) {$\text{Sp}_3(M_i,w_i)$}; 
\end{tikzpicture}
\end{center}
  
Choose a point $q_i$ in $l_i - \{M_n\}$. Define $H(q_i,M_j)$ to be the curve segment which starts at $q_i$, goes counter-clockwise (possibly around the inner boundary), and ends at $M_j$. 

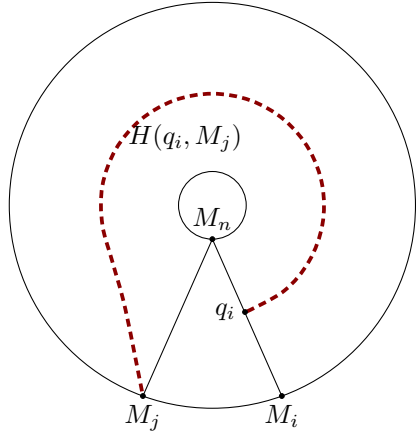
\begin{center}  
\begin{tikzpicture}[scale=.9]
    \draw (0,0) circle (3cm)
   	circle (.5cm);
   	\draw[black] (0cm,-.5cm) -- (250:3cm);
   	\draw[DarkRed, densely dashed, line width=.5mm] plot[smooth, tension=0.6] coordinates {(288:1.65)(310:1.65) (330:1.65)(350:1.65)  (10:1.65)(30:1.65) (50:1.65)(70:1.65)(90:1.65) (110:1.65)(130:1.65) (150:1.65)(170:1.65)(190:1.65) (210:1.7)(230:2) (250:3)} ; 
   	 	\draw[fill] (250:3) circle(1pt) node[below, scale=1]{$M_j$};		
   	\draw[fill] (287:1.65) circle(1pt)node[left, scale=1]{$q_i$};	
   	\draw[black] (0cm,-.5cm) -- (290:3cm);
   	\draw[fill] (290:3) circle(1pt) node[below, scale=1]{$M_i$};
   	\draw[fill] (0cm,-.5cm) circle(1pt) node[above, scale=1]{$M_n$};
   	 \node at (-.4,1) {$H(q_i,M_j)$}; 
\end{tikzpicture}
\end{center}

If $\alpha=(m^a,(m+1)^b,m^c)$ for $m>0$, then it is clear that $\gamma_\alpha^{\text{op}}=\text{Sp}_m(M_a,w_a)\cup H(w_a,M_{a+b+1})$ for some $w_a \in l_a-\{M_n,M_a\}$.  (Replace $M_{a+b+1}$ with $M_1$ in case $a+b+1=n$.) Moreover, if $\beta=\alpha+(\lambda,\dots,\lambda)$ for $\lambda \geq 0$, then $\gamma_\beta^{\text{op}}=\text{Sp}_{m+\lambda}(M_a,q_a)\cup H(q_a,M_{a+b+1})$ for some $q_a\in l_a-\{M_n,M_a\}$. If $m=0$, then we will say that $\gamma_\alpha^{\text{op}}=\text{Sp}_0(M_a,M_a)\cup H(M_a,M_{a+b+1})=H(M_a,M_{a+b+1})$. Here, $\text{Sp}_0(M_a,M_a)$ is degenerate and should be understood as the point $M_a$. Extend the definition of Int to apply to curve segments. With the understanding that $m$ may be zero, in which case $w_a$ should be replaced with $M_a$, note that $\text{Int}(\gamma_\alpha^{\text{op}})=\text{Int}(\text{Sp}_m(M_a,w_a),H(w_a,M_{a+b+1}))=m$. This holds whether $a+b+1<n$ or $a+b+1=n$ (see Example \ref{exa0}).

This description of curves in terms of spirals helps to understand the intersections occurring between two curves. Let $m>0$. Then
\begin{multline*} 
\gamma_\alpha^{\text{op}} \cap\gamma_\beta^{\text{op}}=(\text{Sp}_m(M_a,w_a)\cap\text{Sp}_{m+\lambda}(M_a,q_a)) \cup (\text{Sp}_m(M_a,w_a)\cap H(q_a,M_{a+b+1})) \\
\cup (H(w_a,M_{a+b+1}) \cap \text{Sp}_{m+\lambda}(M_a,q_a)) \cup (H(w_a,M_{a+b+1}) \cap H(q_a,M_{a+b+1})).
\end{multline*}
Therefore,
\begin{multline*} 
\text{Int}(\gamma_\alpha^{\text{op}},\gamma_\beta^{\text{op}})=\text{Int}(\text{Sp}_m(M_a,w_a),\text{Sp}_{m+\lambda}(M_a,q_a))+\text{Int}(\text{Sp}_m(M_a,w_a),H(q_a,M_{a+b+1})) \\
+\text{Int}(H(w_a,M_{a+b+1}),\text{Sp}_{m+\lambda}(M_a,q_a))+\text{Int}(H(w_a,M_{a+b+1}),H(q_a,M_{a+b+1})). \\
\end{multline*}

When drawing $\gamma_\alpha^{\text{op}} \cup \gamma_\beta^{\text{op}}$, it is always possible to choose $q_a$ (strictly) between $w_a$ and $M_n$, and to choose $w_a$ between the $(m-1)$th and the $m$th intersection of $\gamma_\beta^{\text{op}}$ with $l_a$ (such as in Example \ref{exa1}). We will always assume $\gamma_\alpha^{\text{op}}$ and $\gamma_\beta^{\text{op}}$ are drawn in this way, since any other way introduces additional pairwise intersections. Indeed, when drawn in this way,
\[\text{Sp}_m(M_a,w_a)\cap\text{Sp}_{m+\lambda}(M_a,q_a)=H(w_a,M_{a+b+1}) \cap H(q_a,M_{a+b+1})=\emptyset\]
(see Example \ref{exa1}). Therefore, to determine $\text{Int}(\gamma_\alpha^{\text{op}},\gamma_\beta^{\text{op}})$, we need only compute $\text{Int}(\text{Sp}_m(M_a,w_a),H(q_a,M_{a+b+1}))$ and $\text{Int}(H(w_a,M_{a+b+1}),\text{Sp}_{m+\lambda}(M_a,q_a))$.

In the case $m=0$, any choice of $q_a \in l_a-\{M_n,M_a\}$ will result in the same number of pairwise intersections between $\gamma_\alpha^{\text{op}}$ and $\gamma_\beta^{\text{op}}$, zero.

\begin{Exa}\label{exa0} The solid curve is $\text{Sp}_m(M_a,w_a)$ where $m=3$. If $a+b+1<n$, then $M_{a+b+1}$ lies strictly to the right of $M_a$. If $a+b+1=n$, then $M_1$ lies to the left of $M_a$ or $M_1 = M_a$. The two dashed curve segments are the two options for $H(w_a,M_{a+b+1})$.
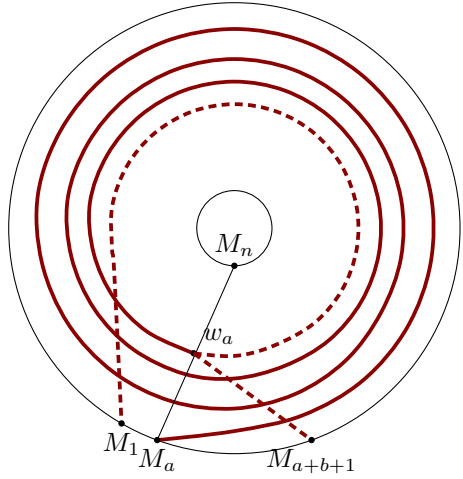
\begin{center}
\begin{tikzpicture}[scale=1]
\draw (0,0) circle (3cm)
   	circle (.5cm); 
\draw [DarkRed, line width=.5mm] plot [smooth, tension=0.6] coordinates { (250:3) (270:2.7) (290:2.65) (310:2.65) (330:2.65) (350:2.65) (10:2.65) (30:2.65) (50:2.65) (70:2.65) (90:2.65) (110:2.65) (130:2.65) (150:2.65) (170:2.65)(190:2.6) (210:2.55) (230:2.5) (250:2.45) (270:2.4) (290:2.35) (310:2.3) (330:2.25) (350:2.25) (10:2.25) (30:2.25) (50:2.25)(70:2.25)(90:2.25)(110:2.25)(130:2.25) (150:2.25) (170:2.25) (190:2.2) (210:2.15) (230:2.1)  (251:2.05) (270:2)(290:1.95) (310:1.95) (330:1.95) (350:1.95) (10:1.95)(30:1.95)(50:1.95) (70:1.95) (90:1.95)(110:1.95)(130:1.95) (150:1.95) (170:1.95)(190:1.9)(210:1.85)(230:1.8)(252:1.75)}; 
 \draw[fill] (250:3) circle(1pt) node[below, scale=1]{$M_a$};		
 \draw[black] (0cm,-.5cm) -- (250:3cm);
 \draw[fill] (252:1.75)circle(1pt) node[above right, scale=1]{$w_a$};
 \draw[fill] (0cm,-.5cm) circle(1pt) node[above, scale=1]{$M_n$}; 
  
	\draw[DarkRed, densely dashed, line width=.5mm] plot[smooth, tension=0.6] coordinates {(252:1.75)(270:1.7)(288:1.65)(310:1.65) (330:1.65)(350:1.65)  (10:1.65)(30:1.65) (50:1.65)(70:1.65)(90:1.65) (110:1.65)(130:1.65) (150:1.65)(170:1.65)(190:1.65) (200:1.7)(240:3)} ;  
	\draw[DarkRed, densely dashed, line width=.5mm] plot[smooth, tension=0.6] coordinates {(252:1.75)(290:3)};
   	\draw[fill] (290:3) circle(1pt) node[below, scale=1]{$M_{a+b+1}$};	
   	\draw[fill] (240:3) circle(1pt) node[below, scale=1]{$M_1$};
\end{tikzpicture}
\end{center}
\end{Exa}

 \begin{Exa}\label{exa1} Let $\alpha=121$ and $\beta=343=\alpha+222$. In this example, $m=1$, $\lambda=2$, $n=3$, $a=1$, and $b=1$. $\gamma_\alpha^{\text{op}}=\text{Sp}_1(M_1,w_1) \cup H(w_1,M_1)$ and $\gamma_\beta^{\text{op}}=\text{Sp}_3(M_1,q_1) \cup H(q_1,M_1)$ are shown. Note that $w_1$ is between the $0$th and $1$st intersection of $\gamma_\beta^{\text{op}}$ with $l_1$, which enables $\gamma_\alpha^{\text{op}}$ and $\gamma_\beta^{\text{op}}$ to be drawn in such a way that results in the fewest possible pairwise intersections.  The elements of $\gamma_\alpha^{\text{op}} \cap \gamma_\beta^{\text{op}}$ are highlighted.
 
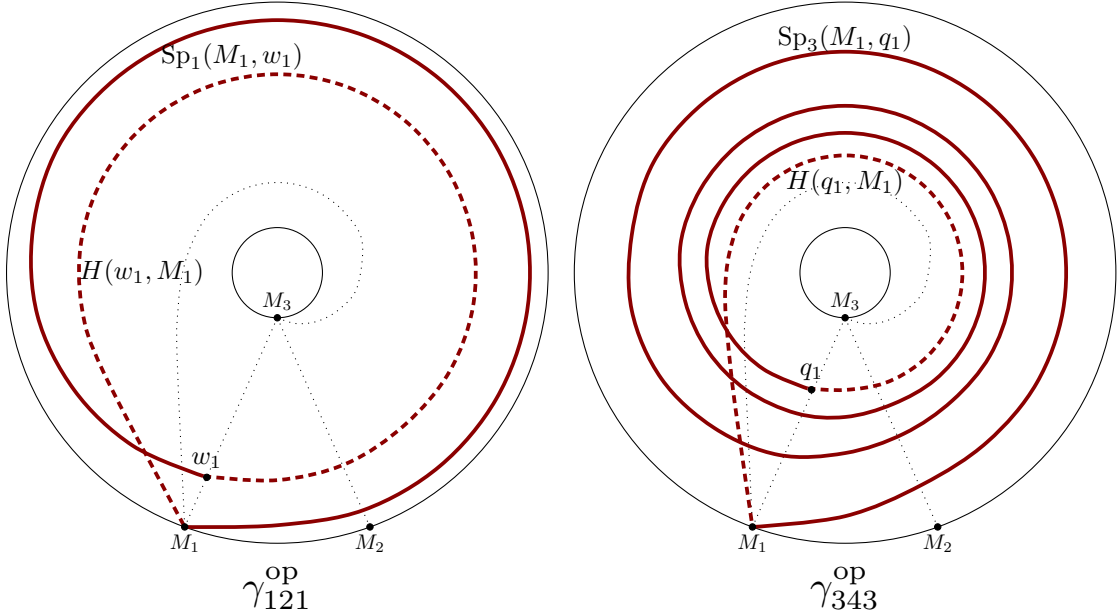
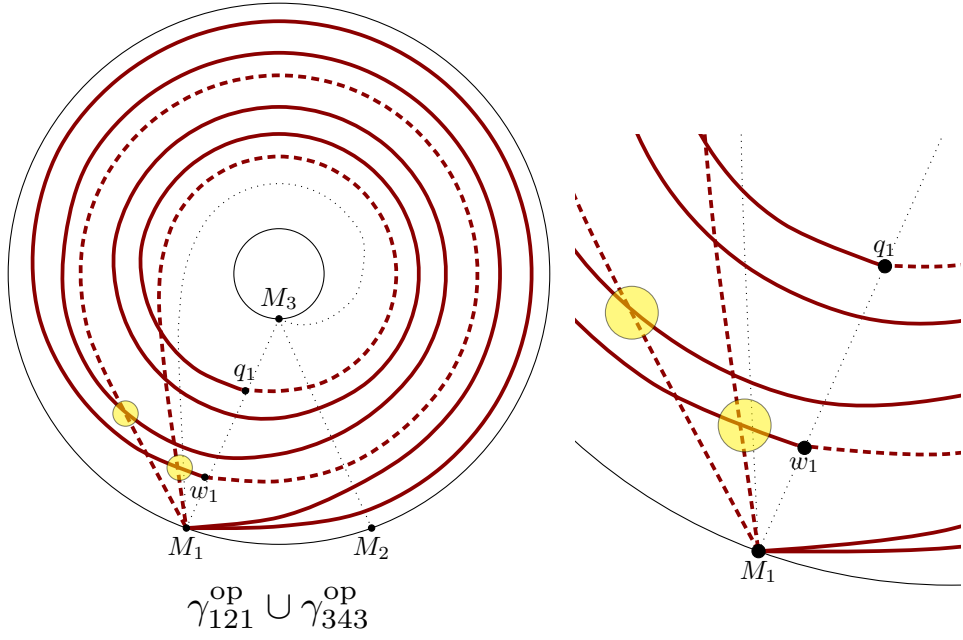
\begin{center}    
\begin{tikzpicture}[scale=1.2] 
    \draw (0,0) circle (3cm)
   	circle (.5cm);
   	\draw[dotted] (0cm,-.5cm) -- (250:3cm)  ;
   	\draw[dotted] (0cm,-.5cm) -- (290:3cm) ;
   	\draw [dotted] plot [smooth, tension=.6] coordinates { (270:.5) (290:.6)(310:.7) (330:.8) (350:.9) (20:1) (40:1)(60:1)(80:1)(100:1) (120:1)(140:1) (160:1.0) (180:1.05) (200:1.17) (220:1.45) (250:3)}; 	
   	\draw[DarkRed, densely dashed, line width=.5mm] plot[smooth, tension=0.6] coordinates {(250:2.4)(270:2.3)(290:2.25) (310:2.2) (330:2.2) (350:2.2) (10:2.2)(30:2.2)(50:2.2) (70:2.2)(90:2.2)(110:2.2)(130:2.2)(150:2.2) (170:2.2) (190:2.2) (210:2.22 ) (250:3)};
 \draw [DarkRed,  line width=0.5mm] plot [smooth, tension=0.6] coordinates { (250:3) (270:2.8) (290:2.8) (310:2.8) (330:2.8) (350:2.8) (10:2.8) (30:2.8) (50:2.8) (70:2.8) (90:2.8) (110:2.8) (130:2.8) (150:2.8) (170:2.75) (190:2.7) (210:2.6) (230:2.5)  (250:2.4)}; 
 \draw[fill] (251:2.4) circle(1pt) node[above, scale=1]{$w_1$};
 \node at (-1.5,0) {$H(w_1,M_1)$}; 
  \node at (-.5,2.4) {$\text{Sp}_1(M_1,w_1)$}; 
  \draw[fill] (0cm,-.5cm) circle(1pt) node[above, scale=.8]{$M_3$};
   	\draw[fill] (250:3) circle(1pt) node[below, scale=.8]{$M_1$};		
   	\draw[fill] (290:3) circle(1pt)node[below, scale=.8]{$M_2$};	
   	\node[scale=1.5] at (-0,-3.5 ) {$\gamma_{121}^{\text{op}}$}; 
\end{tikzpicture}
\hspace{1mm}
\begin{tikzpicture}[scale=1.2]
    \draw (0,0) circle (3cm)
   	circle (.5cm);
   	\draw[dotted] (0cm,-.5cm) -- (250:3cm)  ;
   	\draw[dotted] (0cm,-.5cm) -- (290:3cm) ;
   	\draw [dotted] plot [smooth, tension=.6] coordinates { (270:.5) (290:.6)(310:.7) (330:.8) (350:.9) (20:1) (40:1)(60:1)(80:1)(100:1) (120:1)(140:1) (160:1.0) (180:1.05) (200:1.17) (220:1.45) (250:3)}; 	   		
   	\draw [DarkRed, line width=.5mm] plot [smooth, tension=0.6] coordinates { (250:3) (270:2.7)  (290:2.5) (310:2.45) (330:2.45) (350:2.45) (10:2.45)(30:2.45)(50:2.45) (70:2.45)(90:2.45)(110:2.45)(130:2.45)(150:2.45) (170:2.4) (190:2.4)(210:2.35) (230:2.25)(250:2.15)(270:2) (290:1.9) (310:1.85) (330:1.85) (350:1.85)(10:1.85) (30:1.85) (50:1.85)(70:1.85)(90:1.85)(110:1.85)(130:1.85) (150:1.85) (170:1.85) (190:1.8) (210:1.75) (230:1.7)  (251:1.65) 
(270:1.6)(290:1.55) (310:1.55) (330:1.55) (350:1.55) (10:1.55)(30:1.55)(50:1.55) (70:1.55) (90:1.55)(110:1.55)(130:1.55) (150:1.55) (170:1.55)(190:1.5)(210:1.45)(230:1.4)(252:1.35)}; 
   	\draw[DarkRed, densely dashed, line width=.5mm] plot[smooth, tension=0.6] coordinates {(252:1.35)(270:1.3)(290:1.3)(310:1.3) (330:1.3)(350:1.3)  (10:1.3)(30:1.3) (50:1.3)(70:1.3)(90:1.3) (110:1.3)(130:1.3) (150:1.3)(170:1.3)(190:1.35) (210:1.5)(230:1.9) (250:3)} ; 
\draw[fill] (254:1.35) circle(1pt) node[above, scale=1]{$q_1$};
 \node at (0,1) {$H(q_1,M_1)$}; 
  \node at (0,2.6) {$\text{Sp}_3(M_1,q_1)$}; 
  	\draw[fill] (0cm,-.5cm) circle(1pt) node[above, scale=.8]{$M_3$};
   	\draw[fill] (250:3) circle(1pt) node[below, scale=.8]{$M_1$};		
   	\draw[fill] (290:3) circle(1pt)node[below, scale=.8]{$M_2$};	
   	 	\node[scale=1.5] at (-0,-3.5 ) {$\gamma_{343}^{\text{op}}$}; 
\end{tikzpicture}
\end{center}
\vspace{3mm}
\begin{center}      
\begin{tikzpicture}[scale=1.2]
    \draw (0,0) circle (3cm)
   	circle (.5cm);
   	\draw[dotted] (0cm,-.5cm) -- (250:3cm)  ;
   	\draw[dotted] (0cm,-.5cm) -- (290:3cm) ;
   	\draw [dotted] plot [smooth, tension=.6 ] coordinates { (270:.5) (290:.6)(310:.7) (330:.8) (350:.9) (20:1) (40:1)(60:1)(80:1)(100:1) (120:1)(140:1) (160:1.0) (180:1.05) (200:1.17) (220:1.45) (250:3)}; 
   	 \draw [DarkRed, line width=.5mm] plot [smooth, tension=0.6] coordinates { (250:3) (270:2.8) (290:2.8) (310:2.8) (330:2.8) (350:2.8) (10:2.8) (30:2.8) (50:2.8) (70:2.8) (90:2.8) (110:2.8) (130:2.8) (150:2.8) (170:2.75) (190:2.7) (210:2.6) (230:2.5)  (250:2.4)}; 	
   	\draw[DarkRed,densely dashed, line width=.5mm] plot[smooth, tension=0.6] coordinates {(250:2.4)(270:2.3)(290:2.25) (310:2.2) (330:2.2) (350:2.2) (10:2.2)(30:2.2)(50:2.2) (70:2.2)(90:2.2)(110:2.2)(130:2.2)(150:2.2) (170:2.2) (190:2.2) (210:2.22 ) (250:3)};
   	\draw [DarkRed, line width=.5mm] plot [smooth, tension=0.6] coordinates { (250:3) (270:2.7)  (290:2.5) (310:2.45) (330:2.45) (350:2.45) (10:2.45)(30:2.45)(50:2.45) (70:2.45)(90:2.45)(110:2.45)(130:2.45)(150:2.45) (170:2.4) (190:2.4)(210:2.35) (230:2.25)(250:2.15)(270:2) (290:1.9) (310:1.85) (330:1.85) (350:1.85)(10:1.85) (30:1.85) (50:1.85)(70:1.85)(90:1.85)(110:1.85)(130:1.85) (150:1.85) (170:1.85) (190:1.8) (210:1.75) (230:1.7)  (251:1.65) 
(270:1.6)(290:1.55) (310:1.55) (330:1.55) (350:1.55) (10:1.55)(30:1.55)(50:1.55) (70:1.55) (90:1.55)(110:1.55)(130:1.55) (150:1.55) (170:1.55)(190:1.5)(210:1.45)(230:1.4)(252:1.35)}; 
   	\draw[DarkRed, densely dashed, line width=.5mm] plot[smooth, tension=0.6] coordinates {(252:1.35)(270:1.3)(290:1.3)(310:1.3) (330:1.3)(350:1.3)  (10:1.3)(30:1.3) (50:1.3)(70:1.3)(90:1.3) (110:1.3)(130:1.3) (150:1.3)(170:1.3)(190:1.35) (210:1.5)(230:1.9) (250:3)} ; 
 \draw[fill] (254:1.35) circle(1pt) node[above, scale=1]{$q_1$};
 \draw[fill] (250:2.4) circle(1pt) node[below, scale=1]{$w_1$};
 \draw[fill=yellow, opacity=0.5](-1.1,-2.15 ) circle(4 pt);  
 \draw[fill=yellow, opacity=0.5](-1.7,-1.55 ) circle(4 pt);  
 	\draw[fill] (0cm,-.5cm) circle(1pt) node[above, scale=1]{$M_3$};
   	\draw[fill] (250:3) circle(1pt) node[below, scale=1]{$M_1$};		
   	\draw[fill] (290:3) circle(1pt)node[below, scale=1]{$M_2$};
   \node[scale=1.5] at (-0,-3.7 ) {$\gamma_{121}^{\text{op}} \cup \gamma_{343}^{\text{op}}$}; 
\end{tikzpicture}
\hspace{1mm}
\begin{tikzpicture}[scale=2.5]
\clip (-2,-3.3) rectangle (0.05,-.6);
\draw (0,0) circle (3cm)
   	circle (.5cm);
\draw[dotted] (0cm,-.5cm) -- (250:3cm)  ;
\draw[dotted] (0cm,-.5cm) -- (290:3cm) ;	
\draw [dotted] plot [smooth, tension=.8] coordinates { (270:.5) (290:.6)(310:.7) (330:.8) (350:.9) (20:1) (40:1)(60:1)(80:1)(100:1) (120:1)(140:1) (160:1.0) (180:1.05) (200:1.17) (220:1.45) (250:3)}; 
 \draw [DarkRed,line width=0.5mm] plot [smooth, tension=0.6] coordinates { (250:3) (270:2.8) (290:2.8) (310:2.8) (330:2.8) (350:2.8) (10:2.8) (30:2.8) (50:2.8) (70:2.8) (90:2.8) (110:2.8) (130:2.8) (150:2.8) (170:2.75) (190:2.7) (210:2.6) (230:2.5)  (250:2.4)}; 
 \draw[DarkRed, densely dashed, line width=.5mm] plot[smooth, tension=0.6] coordinates {(250:2.4)(270:2.3)(290:2.25) (310:2.2) (330:2.2) (350:2.2) (10:2.2)(30:2.2)(50:2.2) (70:2.2)(90:2.2)(110:2.2)(130:2.2)(150:2.2) (170:2.2) (190:2.2) (210:2.22 ) (250:3)};	
\draw [DarkRed, line width=.5mm] plot [smooth, tension=0.6 ] coordinates { (250:3) (270:2.7)  (290:2.5) (310:2.45) (330:2.45) (350:2.45) (10:2.45)(30:2.45)(50:2.45) (70:2.45)(90:2.45)(110:2.45)(130:2.45)(150:2.45) (170:2.4) (190:2.4)(210:2.35) (230:2.25)(250:2.15)(270:2) (290:1.9) (310:1.85) (330:1.85) (350:1.85)(10:1.85) (30:1.85) (50:1.85)(70:1.85)(90:1.85)(110:1.85)(130:1.85) (150:1.85) (170:1.85) (190:1.8) (210:1.75) (230:1.7)  (251:1.65) 
(270:1.6)(290:1.55) (310:1.55) (330:1.55) (350:1.55) (10:1.55)(30:1.55)(50:1.55) (70:1.55) (90:1.55)(110:1.55)(130:1.55) (150:1.55) (170:1.55)(190:1.5)(210:1.45)(230:1.4)(254:1.35)}; 
   	\draw[DarkRed,densely dashed, line width=.5mm] plot[smooth, tension=0.6] coordinates {(254:1.35)(270:1.3)(290:1.3)(310:1.3) (330:1.3)(350:1.3)  (10:1.3)(30:1.3) (50:1.3)(70:1.3)(90:1.3) (110:1.3)(130:1.3) (150:1.3)(170:1.3)(190:1.35) (210:1.5)(230:1.9) (250:3)} ; 
 \draw[fill] (254.8:1.35) circle(1pt) node[above, scale=1]{$q_1$};
 \draw[fill] (251:2.4) circle(1pt) node[below, scale=1]{$w_1$};
 \draw[fill=yellow, opacity=0.5](-1.1,-2.15 ) circle(4 pt);  
 \draw[fill=yellow, opacity=0.5](-1.7,-1.55 ) circle(4 pt);  
 \draw[fill] (0cm,-.5cm) circle(1pt) node[above, scale=1]{$M_3$};
   	\draw[fill] (250:3) circle(1pt) node[below, scale=1]{$M_1$};		
   	\draw[fill] (290:3) circle(1pt)node[below, scale=1]{$M_2$};	
\end{tikzpicture}
\end{center} 
\end{Exa}

We also present an analogous description of $F(\alpha)$ in terms of spirals which will be used in the proof of Theorem \ref{type12}. Let $\alpha=(m^a,(m+1)^b,m^c)$ and $\beta=\alpha+(\lambda,\dots,\lambda)$, for $\lambda \geq 0$, be given and suppose $m>0$. That portion of $F(\alpha)$ which is created through steps (1)-(4) has initial point $p_{a+b}$ and terminal point $w_{a+b}$, for some  $w_{a+b}$ in $r_{a+b} - \{p_{a+b}\}$. Define $\text{Sp}_m (p_{a+b},w_{a+b})$ to be that segment of $F(\alpha)$.  Define $H(w_{a+b},B)$  to be the segment of $F(\alpha)$ formed by step (5). Hence $F(\alpha)=\text{Sp}_m (p_{a+b},w_{a+b}) \cup H(w_{a+b},B)$. Similarly, $F(\beta)=\text{Sp}_{m+\lambda} (p_{a+b},q_{a+b}) \cup H(q_{a+b},B)$ for some $q_{a+b}$ in $r_{a+b}-\{p_{a+b}\}$.  If $m=0$, then by definition say that $F(\alpha)=\text{Sp}_0 (p_{a+b},p_{a+b}) \cup H(p_{a+b},B)=H(p_{a+b},B)$ (the segment which results from step (5) alone).  Note that $\text{Int}(F(\alpha))=\text{Int}(\text{Sp}_m (p_{a+b},w_{a+b}),H(w_{a+b},B))=m$, with the understanding that if $m=0$ then $w_{a+b}$ should be replaced with $p_{a+b}$ and the degenerate spiral $\text{Sp}_0 (p_{a+b},p_{a+b})$ should be interpreted as the point $p_{a+b}$.

When drawing $F(\alpha) \cup F(\beta)$, we can always choose $q_{a+b}$ to lie above $w_{a+b}$ (meaning that $\text{Im}(w_{a+b})<\text{Im}(q_{a+b})$) and choose $w_{a+b}$ to lie between the $(m-1)$th and the $m$th intersection of $F(\beta)$ with $r_{a+b}$ (as in Example \ref{planeint}).  We will always assume $F(\alpha)$ and $F(\beta)$ are drawn in this way since the resulting curves are representatives of the homotopy classes of $F(\alpha)$ and $F(\beta)$, respectively, which have the fewest pairwise intersections. Indeed, it ensures that $\text{Sp}_m (p_{a+b},w_{a+b})\cap \text{Sp}_{m+\lambda} (p_{a+b},q_{a+b})=\emptyset$ and $H(w_{a+b},B) \cap H(q_{a+b},B)=\emptyset$.

\begin{Exa}\label{planeint}  $F(\alpha)=F(121)=\text{Sp}_1(p_2,w_2) \cup H(w_2,B)$ and $F(\beta)=F(232)=\text{Sp}_2(p_2,q_2) \cup H(q_2, B)$ are shown, with their intersections highlighted. Note that $w_2$ is strictly between the 0th and 1st intersection of $F(\beta)$ with $r_2$.
\begin{center} 
 \begin{tikzpicture}[scale=2]
\fill[Linen] (-1,0) rectangle (5,3);
\draw[-] (-1,0) -- (5,0);
\draw[densely dotted] (0,1.5) -- (0,3) node[below left, scale=.8]{$r_1$}; 
\draw[densely dotted] (2,1.5) -- (2,3) node[below left, scale=.8]{$r_2$}; 
\draw[densely dotted] (4,1.5) -- (4,3) node[below left, scale=.8]{$r_3$};
\draw [DarkRed,  line width=.5mm] plot [smooth, tension=0.6] coordinates {(2,1.5)(2.6,1.4) (3.5,1)(4.3,1.1)(4.45,1.6) (4.1,2.1)(3.5,2.3)(2,2.4) (1.5,2.4)(1,2.4)(.5,2.35) (0,2.2) (-.3 ,1.9) (-.4,1.6)(-.3,1.2) (.2,1) (1,1.2) (2,1.8)(2.5,1.7) (3.5,1.2 )(4.2,1.3)(4.1,1.8 )(3.5,2.1)(2,2.2) (1,2.2) (0,2)(-.2,1.55) (-.1,1.3)(.2,1.2) (1,1.4)(2,2)}; 
 \draw[DarkRed, densely dashed, line width=.5mm] plot [smooth, tension=0.6] coordinates{(2,2) (2.2,2) (2.5,1.9)(2.7,1.5)(2.5,.8)(2,0)};
  \draw [DarkRed, line width=0.5mm] plot [smooth, tension=0.6] coordinates {(2,1.5)(3.5,.75)(4.3,.8)(4.7,1.2) (4.7,1.7)(4.5,2.1)(4.1,2.4) (3.5,2.6)(2.5,2.7) (1,2.7) (-.1,2.45) (-.7,1.7)(-.55,1)(.1,.75) (1,1) (2,1.65) }; 
   \draw[DarkRed, densely dashed, line width=.5mm] plot [smooth, tension=0.6] coordinates{(2,1.65) (2.3,1.55) (2.3,1)(2,0)};
   \draw[fill] (0,1.5) circle(1 pt) node[below, scale=.8] {$p_1$}; 
\draw[fill] (2,1.5) circle(1 pt) node[below, scale=.8] {$p_2$};
\draw[fill] (4,1.5) circle(1 pt) node[below, scale=.8] {$p_3$};
\draw[fill] (2,0) circle(1 pt) node[below, scale=1] {$B$}; 
\draw[fill] (2,1.65) circle(1 pt) node[left, scale=1] {$w_2$};
\draw[fill] (2,2) circle(1 pt) node[left, scale=1] {$q_2$};
 \draw[fill=yellow, opacity=0.5](2.33,1.45) circle(2 pt);  
  \draw[fill=yellow, opacity=0.5](2.65,1.16) circle(2 pt);  
 \end{tikzpicture} 
 \end{center}  
\end{Exa}

\begin{Thm}[\ref{mainthm} (2)]\label{type12} Let $\alpha=(m^a,(m+1)^b,m^c)$ and $\beta$ be given such that $\beta-\alpha \in \{0\} \cup \text{Im}_{+}$. Then $\text{Int} (F(\alpha), F(\beta))=\text{Int}(\gamma_\alpha^{\text{op}}, \gamma_\beta^{\text{op}})$.
\end{Thm}

\begin{proof} First note that if $m=0$ then it is easy to directly check the conclusion of the theorem, since $\text{Int} (F(\alpha), F(\beta))=\text{Int}(\gamma_\alpha^{\text{op}}, \gamma_\beta^{\text{op}})=0$ for any $\beta$. So suppose $m>0$. Let $0\leq \lambda$ be given, so that $\gamma_\alpha^{\text{op}}=\text{Sp}_m(M_a,w_a)\cup H(w_a,M_{a+b+1})$ for some $w_a \in l_a-\{M_a,M_n\}$ and $\gamma_\beta^{\text{op}}=\text{Sp}_{m+\lambda}(M_a, q_a)\cup H(q_a, M_{a+b+1})$ for $q_a \in l_a-\{M_a,M_n\}$. Now suppose $y \in \text{Sp}_{m+\lambda}(M_a,q_a) \cap H(w_a,M_{a+b+1})$. Let $y'$ be the next (according to the orientation of  $H(w_a,M_{a+b+1})$) intersection occurring between $H(w_a,M_{a+b+1})$ and any other curve. That other curve must be $\text{Sp}_m(M_a,w_a)$ because of the way we have chosen to draw $\gamma_\alpha^{\text{op}}$ and $\gamma_\beta^{\text{op}}$ (see Example \ref{intnum}). Hence $y' \in \text{Sp}_m(M_a,w_a)\cap H(w_a,M_{a+b+1})$. $y \mapsto y'$ in fact defines a bijection, which implies

\[\text{Int}(\text{Sp}_{m+\lambda}(M_a,q_a), H(w_a,M_{a+b+1}))=\text{Int}(\text{Sp}_m(M_a,w_a),H(w_a,M_{a+b+1}))=m.\]
 
Note that this is true whether $a+b+1 <n$ or $a+b+1 =n$ (see Example \ref{intnum}). Now let $z \in\text{Sp}_m(M_a,w_a)\cap H(q_a,M_{a+b+1})$.  Let $z'$ be the previous intersection between $\text{Sp}_m(M_a,w_a)$ and any other curve. That curve must be $H(w_a,M_{a+b+1})$ (see Example \ref{intnum0}), so $z' \in \text{Sp}_m(M_a,w_a) \cap H(w_a,M_{a+b+1})$. Since $z \mapsto z'$ is a bijection, 

\[\text{Int}(\text{Sp}_m(M_a,w_a),H(q_a,M_{a+b+1})) =\text{Int}(\text{Sp}_m(M_a,w_a),H(w_a,M_{a+b+1}))=m.\]

Again, this holds true whether $a+b+1<n$ or $a+b+1=n$. It follows that 
\begin{multline*} 
\text{Int}(\gamma_\alpha^{\text{op}},\gamma_\beta^{\text{op}})=\text{Int}(\text{Sp}_m(M_a,w_a),\text{Sp}_{m+\lambda}(M_a,q_a))+\text{Int}(\text{Sp}_m(M_a,w_a),H(q_a,M_{a+b+1})) \\
+\text{Int}(H(w_a,M_{a+b+1}),\text{Sp}_{m+\lambda}(M_a,q_a))+\text{Int}(H(w_a,M_{a+b+1}),H(q_a,M_{a+b+1})) \\
= 0+m+m+0 \\
=2m.\\ 
\end{multline*} 

Now let $F(\alpha)=\text{Sp}_m (p_{a+b},w_{a+b})\cup H(w_{a+b},B)$ and $F(\beta)=\text{Sp}_{m+\lambda} (p_{a+b},q_{a+b}) \cup H(q_{a+b},B)$ for some $w_{a+b}, q_{a+b} \in r_{a+b}-\{p_{a+b}\}$.  Now, let $y\in  \text{Sp}_{m+\lambda}(p_{a+b},q_{a+b}) \cap H(w_{a+b},B)$. Let $y'$ be the next (according to the orientation of $H(w_{a+b},B)$) intersection of $H(w_{a+b},B)$ with any other curve. That other curve must be $\text{Sp}_{m} (p_{a+b},w_{a+b})$ because of the way $F(\alpha)$ and $F(\beta)$ are drawn (see Example \ref{intnum2}). Hence $y' \in  \text{Sp}_{m} (p_{a+b},w_{a+b})\cap H(w_{a+b},B)$. $y \mapsto y'$ defines a bijection between $  \text{Sp}_{m+\lambda} (p_{a+b},q_{a+b})\cap H(w_{a+b},B)$ and $\text{Sp}_{m} (p_{a+b},w_{a+b})\cap H(w_{a+b},B)$. Therefore, 

\[\text{Int}(\text{Sp}_{m+\lambda} (p_{a+b},q_{a+b}),H(w_{a+b},B))=\text{Int}(\text{Sp}_m (p_{a+b},w_{a+b}),H(w_{a+b},B))=m.\]

 Similarly, let $z\in \text{Sp}_m (p_{a+b},w_{a+b}) \cap H(q_{a+b},B)$. Let $z'$ be the previous intersection between $\text{Sp}_m (p_{a+b},w_{a+b})$ and any other curve. That other curve must be $H(w_{a+b},B)$ (see Example \ref{intnum2}), and therefore $z' \in \text{Sp}_m (p_{a+b},w_{a+b}) \cap H(w_{a+b},B)$. $z \mapsto z'$ gives a bijection between $\text{Sp}_m (p_{a+b},w_{a+b})\cap H(q_{a+b},B)$ and $\text{Sp}_m (p_{a+b},w_{a+b}) \cap H(w_{a+b},B)$. Therefore, 
 
\[\text{Int}(\text{Sp}_m (p_{a+b},w_{a+b}), H(q_{a+b},B))=\text{Int}(\text{Sp}_m (p_{a+b},w_{a+b}),H(w_{a+b},B))=m.\]

So by a computation analogous to that for $\gamma_\alpha^{\text{op}}$ and $\gamma_\beta^{\text{op}}$, $\text{Int}(F(\alpha),F(\beta))=2m$.

\end{proof}

\begin{Exa}\label{intnum} Let $m=1$ and $\lambda=2$. The dashed spiral is  $\text{Sp}_1(M_a,w_a)$ and the solid spiral is $\text{Sp}_3(M_a,q_a)$. The two dotted curve segments show the two choices for $H(w_a, M_{a+b+1})$.
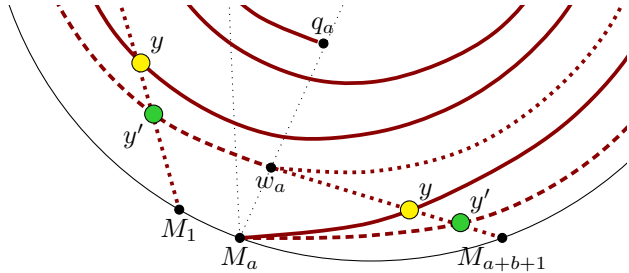
\begin{center} 
\begin{tikzpicture}[scale=1.7]
\clip (-3,-3.8) rectangle (2,-1);
 \draw (0,0) circle (3cm)
   	circle (.5cm);
   	\draw[dotted] (0cm,-.5cm) -- (250:3cm)  ;
   	\draw [dotted] plot [smooth, tension=.8] coordinates { (270:.5) (290:.6)(310:.7) (330:.8) (350:.9) (20:1) (40:1)(60:1)(80:1)(100:1) (120:1)(140:1) (160:1.0) (180:1.05) (200:1.17) (220:1.45) (250:3)}; 
   	 \draw [DarkRed,densely dashed, line width=.5mm] plot [smooth, tension=0.6] coordinates { (250:3) (270:2.8) (290:2.8) (310:2.8) (330:2.8) (350:2.8) (10:2.8) (30:2.8) (50:2.8) (70:2.8) (90:2.8) (110:2.8) (130:2.8) (150:2.8) (170:2.75) (190:2.7) (210:2.6) (230:2.5)  (251:2.4)}; 
 	\draw[DarkRed,dotted, line width=.5mm] plot[smooth, tension=0.6] coordinates {(251:2.4)(270:2.3)(290:2.25) (310:2.2) (330:2.2) (350:2.2) (10:2.2)(30:2.2)(50:2.2) (70:2.2)(90:2.2)(110:2.2)(130:2.2)(150:2.2) (170:2.2) (190:2.2) (210:2.22 ) (240:3)};	
 	 \draw[DarkRed, dotted,line width=.5mm] plot[smooth, tension=0.8] coordinates {  (250:2.4)(290:3)};
   	\draw [DarkRed, line width=.5mm] plot [smooth, tension=0.6 ] coordinates { (250:3) (270:2.7)  (290:2.5) (310:2.45) (330:2.45) (350:2.45) (10:2.45)(30:2.45)(50:2.45) (70:2.45)(90:2.45)(110:2.45)(130:2.45)(150:2.45) (170:2.4) (190:2.4)(210:2.35) (230:2.25)(250:2.15)(270:2) (290:1.9) (310:1.85) (330:1.85) (350:1.85)(10:1.85) (30:1.85) (50:1.85)(70:1.85)(90:1.85)(110:1.85)(130:1.85) (150:1.85) (170:1.85) (190:1.8) (210:1.75) (230:1.7)  (251:1.65) 
(270:1.6)(290:1.55) (310:1.55) (330:1.55) (350:1.55) (10:1.55)(30:1.55)(50:1.55) (70:1.55) (90:1.55)(110:1.55)(130:1.55) (150:1.55) (170:1.55)(190:1.5)(210:1.45)(230:1.4)(252:1.35)};  

 \draw[fill] (254:1.35) circle(1pt) node[above, scale=1]{$q_a$};
 \draw[fill] (251:2.4) circle(1pt) node[below, scale=1]{$w_a$};   
 \draw[fill=yellow] (.3,-2.6 ) circle(2 pt) node[above right, scale=1] {$y$}; 
\draw[fill=LimeGreen] (-1.7,-1.85 )  circle(2 pt) node[below left, scale=1] {$y'$}; 
\draw[fill=yellow] (-1.8,-1.45  ) circle(2 pt) node[above right, scale=1] {$y$}; 
\draw[fill=LimeGreen] (.7,-2.7)  circle(2 pt) node[above right, scale=1] {$y'$}; 
 	\draw[fill] (0cm,-.5cm) circle(1pt) node[above, scale=.8]{$M_3$};
   	\draw[fill] (250:3) circle(1pt) node[below, scale=1]{$M_a$};	
   	\draw[fill] (290:3) circle(1pt) node[below, scale=1]{$M_{a+b+1}$};	
   	\draw[fill] (240:3) circle(1pt) node[below, scale=1]{$M_1$};
\end{tikzpicture}
\end{center}
\end{Exa}

\begin{Exa}\label{intnum0} Let the curves be as in Example \ref{exa1}. Then  $\text{Sp}_m(M_a,w_a)=\text{Sp}_1(M_1,w_1)$ and $H(q_a,M_{a+b+1})= H(q_1, M_1)$. There is only one intersection $z$ between them.  
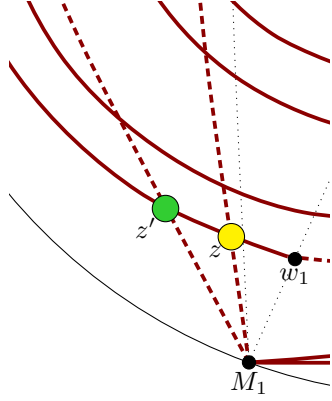
\begin{center}
\begin{tikzpicture}[scale=2.5]
\clip (-2.3,-3.3) rectangle (-.6,-.9);
\draw (0,0) circle (3cm)
   	circle (.5cm);
\draw[dotted] (0cm,-.5cm) -- (250:3cm)  ;
\draw[dotted] (0cm,-.5cm) -- (290:3cm) ;	
\draw [dotted] plot [smooth, tension=.8] coordinates { (270:.5) (290:.6)(310:.7) (330:.8) (350:.9) (20:1) (40:1)(60:1)(80:1)(100:1) (120:1)(140:1) (160:1.0) (180:1.05) (200:1.17) (220:1.45) (250:3)}; 
 \draw [DarkRed,line width=0.5mm] plot [smooth, tension=0.6] coordinates { (250:3) (270:2.8) (290:2.8) (310:2.8) (330:2.8) (350:2.8) (10:2.8) (30:2.8) (50:2.8) (70:2.8) (90:2.8) (110:2.8) (130:2.8) (150:2.8) (170:2.75) (190:2.7) (210:2.6) (230:2.5)  (250:2.4)}; 
 \draw[DarkRed,densely dashed, line width=.5mm] plot[smooth, tension=0.6] coordinates {(250:2.4)(270:2.3)(290:2.25) (310:2.2) (330:2.2) (350:2.2) (10:2.2)(30:2.2)(50:2.2) (70:2.2)(90:2.2)(110:2.2)(130:2.2)(150:2.2) (170:2.2) (190:2.2) (210:2.22 ) (250:3)};	
\draw [DarkRed, line width=.5mm] plot [smooth, tension=0.6 ] coordinates { (250:3) (270:2.7)  (290:2.5) (310:2.45) (330:2.45) (350:2.45) (10:2.45)(30:2.45)(50:2.45) (70:2.45)(90:2.45)(110:2.45)(130:2.45)(150:2.45) (170:2.4) (190:2.4)(210:2.35) (230:2.25)(250:2.15)(270:2) (290:1.9) (310:1.85) (330:1.85) (350:1.85)(10:1.85) (30:1.85) (50:1.85)(70:1.85)(90:1.85)(110:1.85)(130:1.85) (150:1.85) (170:1.85) (190:1.8) (210:1.75) (230:1.7)  (251:1.65) 
(270:1.6)(290:1.55) (310:1.55) (330:1.55) (350:1.55) (10:1.55)(30:1.55)(50:1.55) (70:1.55) (90:1.55)(110:1.55)(130:1.55) (150:1.55) (170:1.55)(190:1.5)(210:1.45)(230:1.4)(252:1.35)}; 
   	\draw[DarkRed,densely dashed, line width=.5mm] plot[smooth, tension=0.6] coordinates {(252:1.35)(270:1.3)(290:1.3)(310:1.3) (330:1.3)(350:1.3)  (10:1.3)(30:1.3) (50:1.3)(70:1.3)(90:1.3) (110:1.3)(130:1.3) (150:1.3)(170:1.3)(190:1.35) (210:1.5)(230:1.9) (250:3)} ; 
 \draw[fill] (252:1.35) circle(1pt) node[above, scale=1]{$q_1$};
 \draw[fill] (251:2.4) circle(1pt) node[below, scale=1]{$w_1$};
 \draw[fill=yellow](-1.12,-2.15 ) circle(2 pt) node[below left,scale=1]{$z$};  
 \draw[fill=LimeGreen](-1.47,-2) circle(2 pt)node[below left, scale=1]{$z'$} ; 
 \draw[fill] (0cm,-.5cm) circle(1pt) node[above, scale=1]{$M_3$};
   	\draw[fill] (250:3) circle(1pt) node[below, scale=1]{$M_1$};		
   	\draw[fill] (290:3) circle(1pt)node[below, scale=1]{$M_2$};	
\end{tikzpicture}
\end{center}
\end{Exa}

\begin{Exa}\label{intnum2} Let $m=1$ and $\lambda=1$ as in Example \ref{planeint}. 
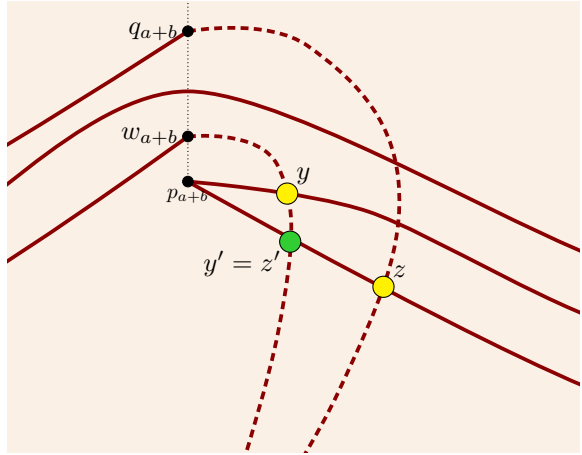
\begin{center}
 \begin{tikzpicture}[scale=4]
 \clip (1.4,.6) rectangle (3.3,2.1);  
\fill[Linen] (-1,0) rectangle (5,3);
\path (-1,0) -- (5,0);
\draw[densely dotted] (0,1.5) -- (0,3) node[below left, scale=.8]{$r_1$}; 
\draw[densely dotted] (2,1.5) -- (2,3) node[below left, scale=.8]{$r_2$}; 
\draw[densely dotted] (4,1.5) -- (4,3) node[below left, scale=.8]{$r_3$};
 \draw [DarkRed, line width=0.5mm] plot [smooth, tension=0.6] coordinates {(2,1.5)(3.5,.75)(4.3,.8)(4.7,1.2) (4.7,1.7)(4.5,2.1)(4.1,2.4) (3.5,2.6)(2.5,2.7) (1,2.7) (-.1,2.45) (-.7,1.7)(-.55,1)(.1,.75) (1,1) (2,1.65) }; 
   \draw[DarkRed, densely dashed, line width=.5mm] plot [smooth, tension=0.7] coordinates{(2,1.65) (2.3,1.55) (2.3,1)(2,0)}; 
\draw [DarkRed, line width=.5mm] plot [smooth, tension=0.5] coordinates {(2,1.5)(2.6,1.4) (3.5,1)(4.3,1.1)(4.45,1.6) (4.1,2.1)(3.5,2.3)(2,2.4) (1.5,2.4)(1,2.4)(.5,2.35) (0,2.2) (-.3 ,1.9) (-.4,1.6)(-.3,1.2) (.2,1) (1,1.2) (2,1.8) (3.5,1.2 )(4.2,1.3)(4.1,1.8 )(3.5,2.1)(2,2.2) (1,2.2) (0,2)(-.2,1.55) (-.1,1.3)(.2,1.2) (1,1.4)(2,2)}; 
 \draw[DarkRed,densely dashed, line width=.5mm] plot [smooth, tension=0.7] coordinates{(2,2)(2.4,1.95)(2.7,1.5)(2.5,.8)(2,0)};
    \node at (-2,2.2) {$\text{Sp}_2(p_2,q_2)$} ;
     \draw [DarkRed,dotted, line width=.5mm] (-2.5,1.5) -- (-1.5,1.5);  
     \node at (-2,1.7) {$H(q_2,B)$} ;  
      \draw [DarkRed, densely dashed, line width=.5mm] (-2.5,1) -- (-1.5,1);  
      \node at (-2,1.2) {$\text{Sp}_1(p_2,w_2)$} ;
      \draw [DarkRed,dash pattern = on 3mm off .5mm, line width=.5mm] (-2.5,.5) -- (-1.5,.5);  
      \node at (-2,.7) {$H(w_2,B)$} ; 
       \draw[fill=yellow](2.33,1.46) circle(1pt) node[above right,scale=1]{$y$};  
        \draw[fill=yellow]  (2.65,1.15) circle(1pt) node[above right, scale=1] {$z$}; 
        \draw[fill=LimeGreen]  (2.34,1.3) circle(1pt) node[below left, scale=1] {$y'=z'$}; 
         \draw[fill] (0,1.5) circle(1 pt) node[below, scale=.8] {$p_1$}; 
\draw[fill] (2,1.5) circle(.5 pt) node[below, scale=.8] {$p_{a+b}$};
\draw[fill] (4,1.5) circle(1 pt) node[below, scale=.8] {$p_3$};
\draw[fill] (2,0) circle(1 pt) node[below, scale=1] {$B$}; 
\draw[fill] (2,1.65) circle(.5 pt) node[left, scale=1] {$w_{a+b}$};
\draw[fill] (2,2) circle(.5 pt) node[left, scale=1] {$q_{a+b}$}; 
      \end{tikzpicture}
      \end{center} 
\end{Exa}

We are able to easily show that $\text{Int} (F(\alpha), F(\beta))=\text{Int}(\gamma_\alpha^{\text{op}}, \gamma_\beta^{\text{op}})$ because the descriptions of the planar curves and of the annulus curves in terms of spirals are extremely similar. This similarity can be explained through a natural deformation of $F(\alpha) \cup F(\beta)$ into  $\gamma_\alpha^{\text{op}} \cup \gamma_\beta^{\text{op}}$. In $\Gamma$, $p_{a+b}$ is the common initial point of $F(\alpha)$ and $F(\beta)$ and $B$ is the common terminal point. In $(\mathcal{S},\mathcal{M})^{\text{op}}$, $M_{a}$ is the common initial point of $\gamma_\alpha^{\text{op}}$ and $\gamma_\beta^{\text{op}}$. If $a+b+1<n$, then the common terminal point is $M_{a+b+1}$, which must lie to the right of $M_a$.  To visualize the deformation, pull $p_{a+b}$ down and to the left so that $B$ lies to the right of it. Then identify $p_{a+b}$ with $M_a$ and  $B$ with $M_{a+b+1}$ (Example \ref{singlecurve}). If $a+b+1=n$, then the common endpoint of $\gamma_\alpha^{\text{op}}$ and $\gamma_\beta^{\text{op}}$ is $M_1$, which must lie to the left of (or be equal to) the initial point $M_a$. In this case, carry $p_{a+b}$ clockwise over $F(\alpha) \cup F(\beta)$ and then identify it with $M_a$ (Example \ref{degeneratecase}). Identify $B$ with $M_{1}$. This deformation results in the identifications (cf. examples \ref{singlecurve} and \ref{degeneratecase}): $\text{Sp}_m (p_{a+b},w_{a+b})$ with $\text{Sp}_m(M_a,w_a)$, $H(w_{a+b},B)$ with $ H(w_a,M_{a+b+1})$, $\text{Sp}_{m+\lambda} (p_{a+b},q_{a+b})$ with $ \text{Sp}_{m+\lambda}(M_a,q_a)$, and $H(q_{a+b},B)$ with $ H(q_a,M_{a+b+1})$ (where spirals are allowed to be degenerate). In particular, $F(\alpha)$ is deformed into $\gamma_\alpha^{\text{op}}$ and $F(\beta)$ is deformed into $\gamma_\beta^{\text{op}}$. No intersections are removed nor introduced throughout this process.

\begin{Exa}\label{singlecurve} Let $\alpha=11211$ and $\beta=22322$. $F(11211)=\text{Sp}_1 (p_3,w_3) \cup H(w_3,B)$ gets deformed into $\gamma_{11211}^{\text{op}}=\text{Sp}_1 (M_2,w_2)\cup H(w_2,M_4)$. $F(22322)=\text{Sp}_2(p_3,q_3)\cup H(q_3,B)$ gets deformed into $\gamma_{22322}^{\text{op}}=\text{Sp}_2(M_2,q_2)\cup H(q_2,M_4)$.   
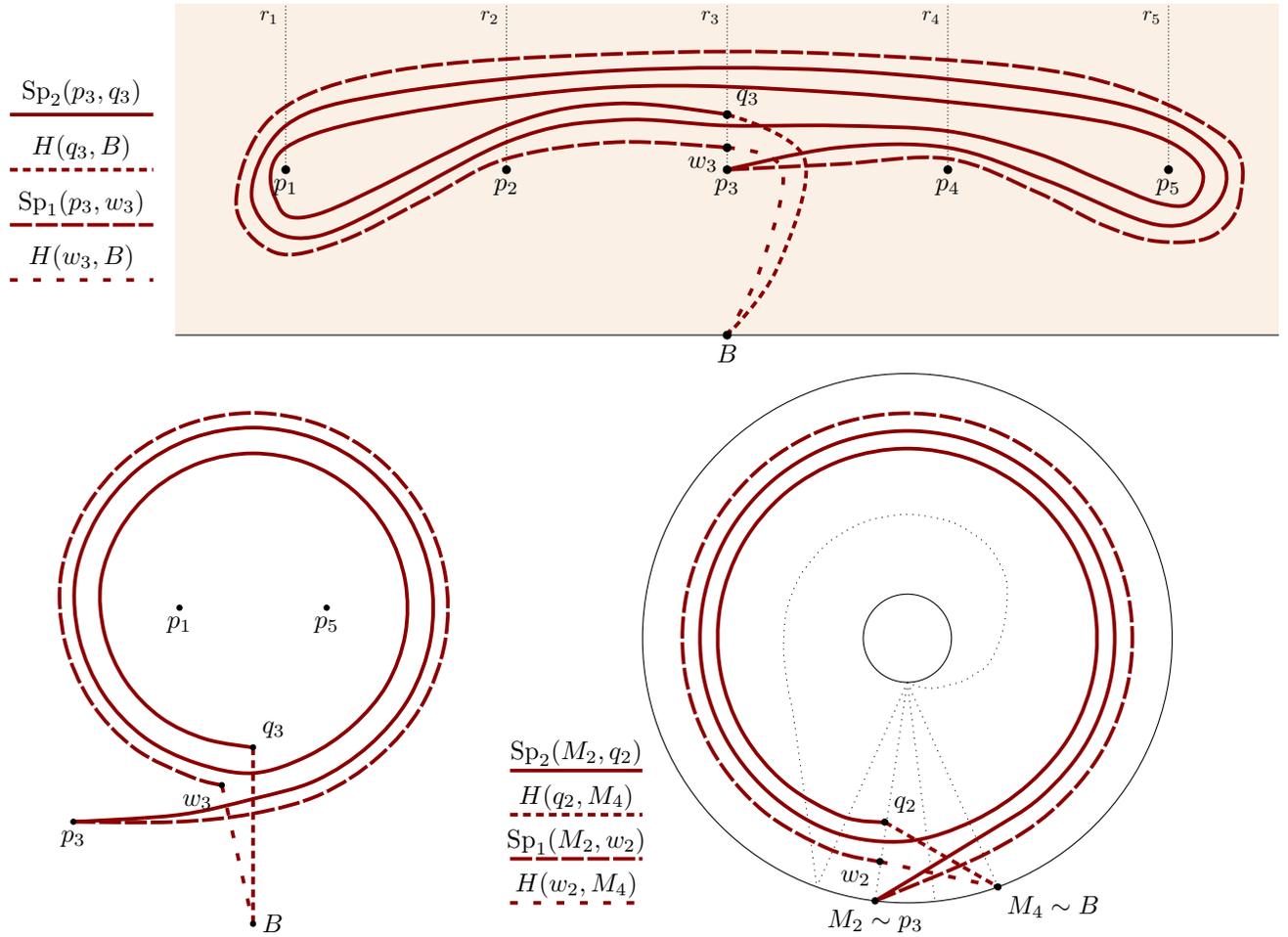
\begin{center}    
\begin{tikzpicture}[scale=1.5]
\fill[Linen] (-1,0) rectangle (9,3);
\draw[-] (-1,0) -- (9,0);
\draw[densely dotted] (0,1.5) -- (0,3) node[below left, scale=.8]{$r_1$}; 
\draw[densely dotted] (2,1.5) -- (2,3) node[below left, scale=.8]{$r_2$}; 
\draw[densely dotted] (4,1.5) -- (4,3) node[below left, scale=.8]{$r_3$}; 
\draw[densely dotted] (6,1.5) -- (6,3) node[below left, scale=.8]{$r_4$}; 
\draw[densely dotted] (8,1.5) -- (8,3) node[below left, scale=.8]{$r_5$}; 
\draw [DarkRed, dash pattern = on 3mm off .5mm,line width=.5mm] plot [smooth, tension=0.6] coordinates {(4, 1.5) (5,1.55) (6,1.6)  (6.8, 1.3)(7.4,1) (7.9,.85)(8.4,.9) (8.65,1.25)(8.5,1.8)(7.2,2.35) (5,2.55) (3,2.55)(1,2.4) (-.1,2)(-.45,1.2)(-.1,.75)(.5,.85) (1.3,1.25)(2,1.6)(3,1.75)(4,1.7)} ;  
\draw [DarkRed,loosely dashed, line width=.5mm] plot [smooth, tension=0.6] coordinates {(4,1.7)(4.5,1.5)(4.3,.6) (4,0)};    
\draw [DarkRed, line width=.5mm] plot [smooth, tension=0.6] coordinates  {(4, 1.5) (5,1.7) (6,1.7)  (6.9, 1.4)(7.9,1) (8.4,1.15) (8.5,1.55)(8,1.95) (7,2.2) (5,2.4) (3,2.4)(1,2.2) (0,1.9)(-.3,1.2)(.3,.9) (2,1.75)(3,1.95)(4,1.9) (5,1.9) (6,1.85) (6.8, 1.65)(7.4,1.4) (7.9,1.2)(8.2,1.2) (8.3,1.5)(7.9,1.8) (7,2) (5,2.2) (3,2.25)(1,2) (0,1.7)(-.1,1.25)(.3,1.1)(2,1.9)(3,2.1)(4,2)  } ; 
\draw [DarkRed,densely dashed, line width=.5mm] plot [smooth, tension=0.6] coordinates {(4,2)(4.7,1.6)(4.45,.6) (4,0)};   
\draw[fill] (0,1.5) circle(1 pt) node[below, scale=1] {$p_1$}; 
\draw[fill] (2,1.5) circle(1 pt) node[below, scale=1] {$p_2$};
\draw[fill] (4,1.5) circle(1 pt) node[below, scale=1] {$p_3$};
\draw[fill] (6,1.5) circle(1 pt) node[below, scale=1] {$p_4$};
\draw[fill] (8,1.5) circle(1 pt) node[below, scale=1] {$p_5$};  
\draw[fill] (4,1.7) circle(1 pt) node[below left, scale=1] {$w_3$};
\draw[fill] (4,2) circle(1 pt) node[above right, scale=1] {$q_3$};
\draw[fill] (4,0) circle(1 pt) node[below, scale=1] {$B$}; 
 \draw [DarkRed,line width=.5mm] (-2.5,2) -- (-1.2,2);   
    \node at (-1.85,2.2) {$\text{Sp}_2 (p_3,q_3)$} ;  
     \draw [DarkRed,densely dashed,line width=.5mm] (-2.5,1.5) -- (-1.2,1.5);  
     \node at (-1.85,1.7) {$H(q_3,B)$} ;  
 \draw [DarkRed,dash pattern = on 3mm off .5mm, line width=.5mm] (-2.5,1) -- (-1.2,1);   
    \node at (-1.85,1.2) {$\text{Sp}_1 (p_3,w_3)$} ;  
     \draw [DarkRed,loosely dashed,line width=.5mm] (-2.5,.5) -- (-1.2,.5);  
     \node at (-1.85,.7) {$H(w_3,B)$} ;   
\end{tikzpicture} 

\begin{tikzpicture}[scale=1]
\path(-3,-4.5) -- (3,-4.5);   
   	\draw [DarkRed, dash pattern = on 3mm off .5mm, line width=.5mm] plot [smooth, tension=0.6] coordinates {(230:3.8)(250:3.1) (270:2.8) (290:2.7) (310:2.65) (330:2.65) (350:2.65) (10:2.65) (30:2.65) (50:2.65) (70:2.65) (90:2.65) (110:2.65) (130:2.65) (150:2.65) (170:2.65)(190:2.6) (210:2.55) (230:2.5) (250:2.45) (260:2.45)};  
   	\draw[DarkRed,loosely dashed, line width=.5mm] plot[smooth, tension=0.8] coordinates { (260:2.45)(270:4.3)} ;  
 \draw [DarkRed, line width=.5mm] plot [smooth, tension=0.6] coordinates {(230:3.8)(250:3) (270:2.6) (290:2.5) (310:2.45) (330:2.45) (350:2.45) (10:2.45) (30:2.45) (50:2.45) (70:2.45) (90:2.45) (110:2.45) (130:2.45) (150:2.45) (170:2.45)(190:2.4) (210:2.35) (230:2.3) (250:2.25) (270:2.25)(290:2.15)(310:2.1)(330:2.1)(350:2.1)(10:2.1)(30:2.1)(50:2.1)(70:2.1)(90:2.1)(110:2.1)(130:2.1)(150:2.1)(170:2.1)(190:2.05)(210:2)(230:1.95)(250:1.9)(270:1.9)};  
 \draw[DarkRed,densely dashed,line width=.5mm] plot[smooth, tension=0.8] coordinates { (270:1.9)(270:4.3)} ;   
\draw[fill](270:4.3) circle(1 pt) node[right, scale=1] {$B$};
\draw[fill] (230:3.8)  circle(1 pt) node[below, scale=1] {$p_3$};
\draw[fill] (-1,0) circle(1 pt) node[below, scale=1] {$p_1$}; 
\draw[fill] (1,0) circle(1 pt) node[below, scale=1] {$p_5$}; 
\draw[fill]  (260:2.45)circle(1 pt) node[below left, scale=1] {$w_3$};
\draw[fill]  (270:1.9)circle(1 pt) node[above right, scale=1] {$q_3$};
 
\end{tikzpicture}
\hspace{1mm}
\begin{tikzpicture}[scale=1.2]
   \draw (0,0) circle (3cm)
   	circle (.5cm);
   	\draw[dotted] (0cm,-.5cm) -- (250:3cm) ;
   	\draw[dotted] (0cm,-.5cm) -- (290:3cm) ;
   	\draw[dotted] (0cm,-.5cm) -- (263:3cm) ;
   	\draw[dotted] (0cm,-.5cm) -- (276:3cm) ;
   	\draw [dotted] plot [smooth, tension=.8] coordinates { (270:.5) (290:.6)(310:.7) (330:.85) (350:1) (20:1.2) (40:1.35)(60:1.4)(80:1.4)(100:1.4) (120:1.4)(140:1.4) (160:1.4) (180:1.4) (200:1.45) (220:1.65) (250:3)}; 	 
   	\draw [DarkRed,dash pattern = on 3mm off .5mm, line width=.5mm] plot [smooth, tension=0.6] coordinates { (263:3) (290:2.6) (310:2.55) (330:2.55) (350:2.55) (10:2.55) (30:2.55) (50:2.55) (70:2.55) (90:2.55) (110:2.55) (130:2.55) (150:2.55) (170:2.55) (190:2.55) (210:2.55) (230:2.55) (250:2.55)(263:2.55)};
  	\draw [DarkRed,loosely dashed, line width=0.5mm] plot [smooth, tension=0.8] coordinates {(263:2.55)(270:2.6)(290:3) }; 
  	\draw [DarkRed, line width=0.5mm] plot [smooth, tension=0.6] coordinates { (263:3) (290:2.4) (310:2.35) (330:2.35) (350:2.35) (10:2.35) (30:2.35) (50:2.35) (70:2.35) (90:2.35) (110:2.35) (130:2.35) (150:2.35) (170:2.35) (190:2.35) (210:2.35) (230:2.35) (250:2.35)(270:2.3)(290:2.2)(310:2.15)(330:2.15) (350:2.15) (10:2.15) (30:2.15) (50:2.15) (70:2.15) (90:2.15) (110:2.15) (130:2.15) (150:2.15) (170:2.15) (190:2.15) (210:2.15) (230:2.15) (250:2.15)(263:2.1)};
  	\draw [DarkRed,densely dashed, line width=0.5mm] plot [smooth, tension=0.8] coordinates {(263:2.1)(290:3) }; 
  	\draw[fill]  (263:2.55)circle(1 pt) node[below left, scale=1] {$  w_2$};
  	
  	\draw[fill]  (263:2.1)circle(1 pt) node[above right, scale=1] {$ q_2$};
  		\draw[fill] (263:3) circle(1pt) node[below, scale=1]{$M_2 \sim p_3$};		
   	\draw[fill] (290:3) circle(1pt)node[below right, scale=1]{$M_4 \sim B$};	
   		\draw [DarkRed,line width=.5mm] (-4.5,-1.5) -- (-3,-1.5);   
    \node at (-3.75,-1.3) {$\text{Sp}_2 (M_2,q_2)$} ;  
     \draw [DarkRed,densely dashed,line width=.5mm] (-4.5,-2) -- (-3,-2);  
     \node at (-3.75,-1.8) {$H(q_2,M_4)$} ;  
   	\draw [DarkRed, dash pattern = on 3mm off .5mm, line width=.5mm] (-4.5,-2.5) -- (-3,-2.5);   
    \node at (-3.75,-2.3) {$\text{Sp}_1 (M_2,w_2)$} ;  
     \draw [DarkRed,loosely dashed, line width=.5mm] (-4.5,-3) -- (-3,-3);  
     \node at (-3.75,-2.8) {$H(w_2,M_4)$} ;  
\end{tikzpicture}
\end{center} 
\end{Exa}

\begin{Exa}\label{degeneratecase} Let $\alpha=010$ and $\beta=343$ and deform $F(\alpha) \cup F(\beta)$ into $\gamma_\alpha^{\text{op}} \cup \gamma_\beta^{\text{op}}$. In this case, $a+b+1=n$ so $p_2$ is carried over $F(343) \cup F(010)$ before being identified with  $M_a=M_1$. $B$ is identified with $M_1$.
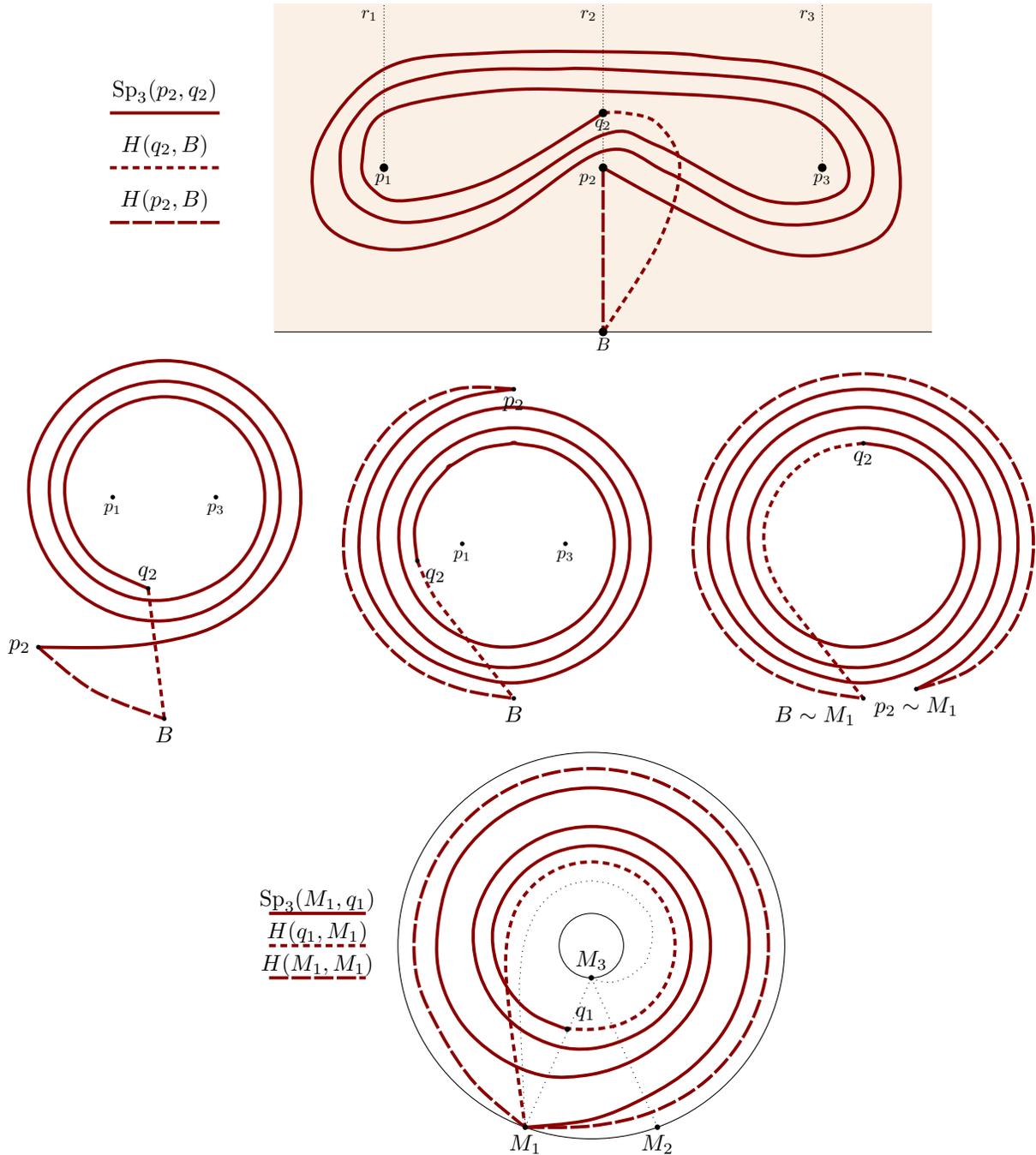
\begin{center}
\begin{tikzpicture}[scale=1.7]
\fill[Linen] (-1,0) rectangle (5,3);
\draw[-] (-1,0) -- (5,0);
\draw[densely dotted] (0,1.5) -- (0,3) node[below left, scale=.8]{$r_1$}; 
\draw[densely dotted] (2,1.5) -- (2,3) node[below left, scale=.8]{$r_2$}; 
\draw[densely dotted] (4,1.5) -- (4,3) node[below left, scale=.8]{$r_3$};
\draw [DarkRed, line width=.5mm] plot [smooth, tension=0.6] coordinates {(2,1.5)(3.5,.75)(4.4,.8)(4.7,1.2) (4.6,1.9)(4.1,2.3)(3.6,2.45)(3.1,2.5)(2.5,2.55)  (1,2.55) (0,2.4) (-.6,1.7)(-.55,1)(.1,.75) (1,1) (2,1.65)(2.6,1.45) (3.5,1)(4.3,1.1)(4.45,1.6) (4.1,2.1)(3.5,2.3)(2,2.4) (1.5,2.4)(1,2.4)(.5,2.35) (0,2.2) (-.3 ,1.9) (-.4,1.6)(-.3,1.2) (.2,1) (1,1.2) (2,1.8)(2.5,1.7) (3.5,1.2 )(4.2,1.3)(4.1,1.8 )(3.5,2.1)(2,2.2) (1,2.2) (0,2)(-.2,1.55) (-.1,1.3)(.2,1.2) (1,1.4)(2,2)}; 
 \draw[DarkRed,densely dashed, line width=.5mm] plot [smooth, tension=0.6] coordinates{(2,2) (2.2,2) (2.5,1.9)(2.7,1.5)(2.5,.8)(2,0)};
 \draw [DarkRed,dash pattern = on 3mm off .5mm, line width=.5mm] plot [smooth, tension=0.8] coordinates {(2,1.5) (2,0) };
 \draw[fill] (0,1.5) circle(1 pt) node[below, scale=.8] {$p_1$}; 
\draw[fill] (2,1.5) circle(1 pt) node[below left, scale=.8] {$p_2$};
\draw[fill] (4,1.5) circle(1 pt) node[below, scale=.8] {$p_3$};
\draw[fill] (2,0) circle(1 pt) node[below, scale=.8] {$B$};  
  \draw[fill] (2,2) circle(1 pt) node[below, scale=.8] {$q_2$};
  \draw [DarkRed, line width=.5mm] (-2.5,2) -- (-1.5,2);   
 \node at (-2,2.2) {$\text{Sp}_3(p_2,q_2)$} ;
 \draw [DarkRed,densely dashed, line width=.5mm] (-2.5,1.5) -- (-1.5,1.5);  
 \node at (-2,1.7) {$H(q_2,B)$} ; 
 \draw [DarkRed,dash pattern = on 3mm off .5mm, line width=.5mm] (-2.5,1) -- (-1.5,1);  
 \node at (-2,1.2) {$H(p_2,B)$} ; 
 \end{tikzpicture}
 
 \begin{tikzpicture}[scale=.8]
\path(-3,-4.5) -- (3,-4.5);   
   	\draw [DarkRed, line width=.5mm] plot [smooth, tension=0.6 ] coordinates {(230:3.8)(250:3.1) (270:2.8) (290:2.7) (310:2.65) (330:2.65) (350:2.65) (10:2.65) (30:2.65) (50:2.65) (70:2.65) (90:2.65) (110:2.65) (130:2.65) (150:2.65) (170:2.65)(190:2.6) (210:2.55) (230:2.5) 	(250:2.45)(270:2.4) (290:2.35) (310:2.3) (330:2.25) (350:2.25) (10:2.25) (30:2.25) (50:2.25)(70:2.25)(90:2.25)(110:2.25)(130:2.25) (150:2.25) (170:2.25) (190:2.2) (210:2.15) (230:2.1)  (251:2.05) (270:2)(290:1.95) (310:1.95) (330:1.95) (350:1.95) (10:1.95)(30:1.95)(50:1.95) (70:1.95) (90:1.95)(110:1.95)(130:1.95) (150:1.95) (170:1.95)(190:1.9)(210:1.85)(230:1.8) (260:1.8)}; 
   	\draw[DarkRed,densely dashed,  line width=.5mm] plot[smooth, tension=0.6] coordinates { (260:1.8)(270:4.3)} ;  
   	 \draw [DarkRed,dash pattern = on 3mm off .5mm, line width=.5mm] plot [smooth, tension=0.6] coordinates {   (230:3.8) (250:4)(270:4.3)}; 
\draw[fill](270:4.3) circle(1 pt) node[below, scale=1] {$B$};
\draw[fill] (230:3.8)  circle(1 pt) node[left, scale=1] {$p_2$};
\draw[fill] (-1,0) circle(1 pt) node[below, scale=.8] {$p_1$}; 
\draw[fill] (1,0) circle(1 pt) node[below, scale=.8] {$p_3$}; 
\draw[fill] (260:1.8) circle(1 pt) node[above, scale=1] {$q_2$};
\end{tikzpicture}
\hspace{1mm}
\begin{tikzpicture}[scale=.8]
\path (-3,-4) -- (3,-4);   
   	\draw [DarkRed, line width=.5mm] plot [smooth, tension=0.6] coordinates {(90:3) (110:3) (130:3) (150:3) (170:3) (190:3.0) (210:2.9) (230:2.8)(250:2.7) (270:2.7) (290:2.65) (310:2.65) (330:2.65) (350:2.65) (10:2.65) (30:2.65) (50:2.65) (70:2.65) (90:2.65) (110:2.65) (130:2.65) (150:2.65) (170:2.65)(190:2.6) (210:2.55) (230:2.5) (250:2.45)(270:2.4) (290:2.35) (310:2.3) (330:2.25) (350:2.25) (10:2.25) (30:2.25) (50:2.25)(70:2.25)(90:2.25)(110:2.25)(130:2.25) (150:2.25) (170:2.25) (190:2.2) (210:2.15) (230:2.1)   (251:2.05)(270:2)(290:1.95) (310:1.95) (330:1.95) (350:1.95) (10:1.95)(30:1.95)(50:1.95) (70:1.95) (90:1.95) (90:1.95)(110:1.95)(130:1.95)(130:1.95)(150:1.95) (170:1.95)(190:1.9)}; 
   	\draw[DarkRed,densely dashed, line width=.5mm] plot[smooth, tension=0.8] coordinates { (190:1.9)(210:1.85)(230:1.9)  (270:3)} ; 
   	 \draw [DarkRed,dash pattern = on 3mm off .5mm, line width=.5mm] plot [smooth, tension=0.6] coordinates {  (90:3) (110:3.2)(130:3.3)(150:3.3)(170:3.3) (190:3.3) (210:3.2) (230:3.1) (250:3)(270:3)}; 
\draw[fill](270:3) circle(1 pt) node[below, scale=1] {$B$};
\draw[fill] (90:3) circle(1 pt) node[below, scale=1] {$p_2$};
\draw[fill] (-1,0) circle(1 pt) node[below, scale=.8] {$p_1$}; 
\draw[fill] (1,0) circle(1 pt) node[below, scale=.8] {$p_3$}; 
\draw[fill] (190:1.9)circle(1 pt) node[below right, scale=1] {$q_2$};
\end{tikzpicture} 
\hspace{1mm}
\begin{tikzpicture}[scale=.8]
\path(-3,-4) -- (3,-4) ; 
   \draw [DarkRed, line width=.5mm] plot [smooth, tension=0.6] coordinates {(290:3)(310:3)(330:3)(350:3)(10:3)(30:3)(50:3) (70:3) (90:3) (110:3) (130:3) (150:3) (170:3) (190:3.0) (210:2.9) (230:2.8)(250:2.7) (270:2.7) (290:2.65) (310:2.65) (330:2.65) (350:2.65) (10:2.65) (30:2.65) (50:2.65) (70:2.65) (90:2.65) (110:2.65) (130:2.65) (150:2.65) (170:2.65)(190:2.6) (210:2.55) (230:2.5) (250:2.45)(270:2.4) (290:2.35) (310:2.3) (330:2.25) (350:2.25) (10:2.25) (30:2.25) (50:2.25)(70:2.25)(90:2.25)(110:2.25)(130:2.25) (150:2.25) (170:2.25) (190:2.2) (210:2.15) (230:2.1)    (251:2.05)(270:2)(290:1.95) (310:1.95) (330:1.95) (350:1.95) (10:1.95)(30:1.95)(50:1.95) (70:1.95) (90:1.95) }; 
   	\draw[DarkRed,densely dashed, line width=.5mm] plot[smooth, tension=0.6] coordinates {(90:1.95)(110:1.95)(130:1.95) (150:1.95) (170:1.95)(190:1.9)(210:1.85)(230:1.9)  (270:3)} ; 
   	 \draw [DarkRed,dash pattern = on 3mm off .5mm, line width=.5mm] plot [smooth, tension=0.6] coordinates { (290:3)(310:3.2)(330:3.3)(350:3.3)(10:3.3)(30:3.3)(50:3.3) (70:3.3)  (90:3.3) (110:3.3)(130:3.3)(150:3.3)(170:3.3) (190:3.3) (210:3.2) (230:3.1) (250:3)(270:3)}; 
   	 \draw[fill]( 270:3) circle(1 pt) node[below left, scale=1] {$B\sim M_1$};
   	 \draw[fill]  (290:3) circle(1 pt) node[below, scale=1] {$p_2 \sim M_1$};
   	 \draw[fill] (90:1.95)circle(1 pt) node[below, scale=1] {$q_2$};
\end{tikzpicture} 

\begin{tikzpicture}[scale=1]
    \draw (0,0) circle (3cm)
   	circle (.5cm);
   	\draw[dotted] (0cm,-.5cm) -- (250:3cm)  ;
   	\draw[dotted] (0cm,-.5cm) -- (290:3cm) ;	
   	\draw [dotted] plot [smooth, tension=.8] coordinates { (270:.5) (290:.6)(310:.7) (330:.8) (350:.9) (20:1) (40:1)(60:1)(80:1)(100:1) (120:1)(140:1) (160:1.0) (180:1.05) (200:1.17) (220:1.45) (250:3)}; 	  
\draw [DarkRed, line width=.5mm] plot [smooth, tension=0.6] coordinates { (250:3) (270:2.7)  (290:2.5) (310:2.45) (330:2.45) (350:2.45) (10:2.45)(30:2.45)(50:2.45) (70:2.45)(90:2.45)(110:2.45)(130:2.45)(150:2.45) (170:2.4) (190:2.4)(210:2.35) (230:2.25)(250:2.15)(270:2) (290:1.9) (310:1.85) (330:1.85) (350:1.85)(10:1.85) (30:1.85) (50:1.85)(70:1.85)(90:1.85)(110:1.85)(130:1.85) (150:1.85) (170:1.85) (190:1.8) (210:1.75) (230:1.7)  (251:1.65) 
(270:1.6)(290:1.55) (310:1.55) (330:1.55) (350:1.55) (10:1.55)(30:1.55)(50:1.55) (70:1.55) (90:1.55)(110:1.55)(130:1.55) (150:1.55) (170:1.55)(190:1.5)(210:1.45)(230:1.4)(252:1.35)}; 
   	\draw[DarkRed,densely dashed, line width=.5mm] plot[smooth, tension=0.6] coordinates {(252:1.35)(270:1.3)(290:1.3)(310:1.3) (330:1.3)(350:1.3)  (10:1.3)(30:1.3) (50:1.3)(70:1.3)(90:1.3) (110:1.3)(130:1.3) (150:1.3)(170:1.3)(190:1.35) (210:1.5)(230:1.9) (250:3)} ; 
 \draw [DarkRed,dash pattern = on 3mm off .5mm, line width=.5mm] plot [smooth, tension=0.6] coordinates { (250:3) (270:2.8) (290:2.75) (310:2.75) (330:2.75) (350:2.75) (10:2.75) (30:2.75) (50:2.75) (70:2.75) (90:2.75) (110:2.75) (130:2.75) (150:2.75) (170:2.75) (190:2.75) (210:2.7) (230:2.8)(250:3)};
 \draw[fill] (0cm,-.5cm) circle(1pt) node[above, scale=1]{$M_3$};
   	\draw[fill] (250:3) circle(1pt) node[below, scale=1]{$M_1$};		
   	\draw[fill] (290:3) circle(1pt)node[below, scale=1]{$M_2$};	  
   	\draw[fill] (254:1.35) circle(1pt) node[above right, scale=1]{$q_1$};
  \draw [DarkRed, line width=.5mm] (-5,.5) -- (-3.5,.5);   
    \node at (-4.25,.7) {$\text{Sp}_3(M_1,q_1)$} ;
     \draw [DarkRed,densely dashed, line width=.5mm] (-5,0) -- (-3.5,0);  
     \node at (-4.25,.2) {$H(q_1,M_1)$} ; 
      \draw [DarkRed,dash pattern = on 3mm off .5mm, line width=.5mm] (-5,-.5) -- (-3.5,-.5);  
     \node at (-4.25,-.3) {$H(M_1,M_1)$} ; 
\end{tikzpicture}  
 \end{center}

 \end{Exa}

\begin{Rmk} 
Suppose $\alpha$ and $\beta$ are roots of the first type such that for $\alpha$, $a=a_0$, $b=b_0$, and for $\beta$, $a=a'$, $b=b'$. One can always deform $F(\alpha)$ into $\gamma_\alpha^{\text{op}}$ and $F(\beta)$ into $\gamma_\beta^{\text{op}}$ separately by making the identifications $M_a \sim p_{a+b}$ and $M_{a+b+1} \sim B$ (or $M_{1} \sim B$) in each case. However, $F(\alpha)\cup F(\beta)$ cannot be deformed into $\gamma_\alpha^{\text{op}} \cup \gamma_\beta^{\text{op}}$ with these identifications unless $\beta-\alpha \in \{0\} \cup \text{Im}_{+}$.

 If $F(\alpha)\cup F(\beta)$ is deformed into $\gamma_\alpha^{\text{op}} \cup \gamma_\beta^{\text{op}}$, then $M_{a_0+b_0+1}\sim B$ and $ M_{a'+b'+1}\sim B$, which implies $a_0+b_0=a'+b'$. Therefore, $p_{a_0+b_0} = p_{a'+b'}$, which implies $M_{a_0} =M_{a'}$. Hence $a_0=a'$ and so $b_0=b'$. Now observe that $a_0=a'$ and $b_0=b'$ if and only if $\beta-\alpha \in \{0\} \cup \text{Im}_{+}$.

 This natural deformation of $F(\alpha)\cup F(\beta)$ into $\gamma_\alpha^{\text{op}} \cup \gamma_\beta^{\text{op}}$ is what explains the equality $\text{Int} (F(\alpha), F(\beta))=\text{Int}(\gamma_\alpha^{\text{op}}, \gamma_\beta^{\text{op}})$. Indeed, one can check that if $\beta-\alpha \notin \{0\} \cup \text{Im}_{+}$ then in general this equality need not hold.

 \end{Rmk}
  
\section{Proof of Theorem \ref{mainthm} for Roots of Type 2}
Fix $n \geq 3$ and nonnegative integers $m,a,b,c$ such that $a+b+c=n$, $a+c \neq n$, $a\geq 1$, and $c\geq 1$. Then 
\[
\alpha=(m+1,\dots,m+1,m,\dots,m,m+1,\dots,m+1)=((m+1)^a,m^b,(m+1)^c)
\]
is a positive real root of type 2 which is not a real Schur root, and all non-Schur roots of type 2 are of this form. 

\begin{Thm}[\ref{mainthm}(1)] Fix $n \geq 3$ and nonnegative integers $m,a,b,c$ such that $a+b+c=n$, $a+c \neq n$, $a\geq 1$, and $c\geq 1$, and consider the root $\alpha=((m+1)^a,m^b,(m+1)^c)$. Then $R(F(\alpha))=\alpha$.
\end{Thm}

\begin{proof}
We simultaneously construct $F(\alpha)$ and prove $R(F(\alpha))=\alpha$.
\begin{enumerate} 
\item Let $F(\alpha)$ start at $p_1$. Going to the right, intersect the rays $r_x$ for $1<x\leq a$, and then for $a+b<x\leq n$. \\
In terms of constructing $R(F(\alpha))$, step (1) corresponds to creating the root $\sum_{j=1}^{a} \alpha_j+\alpha_n$.
\item Going to the left and under the part of $F(\alpha)$ which is already drawn, intersect the rays $r_x$ for $a+b<x<n$.\\
The result of this step is now the root $\sum_{j=1}^{a} \alpha_j+\sum_{k=a+b+1}^{n} \alpha_k$.  
\item Going to the right and over the part of $F(\alpha)$ which is already drawn, intersect the rays $r_x$ for $1\leq x \leq a+b $.\\
The result is  $\sum_{j=1}^{a-1} \alpha_j+\sum_{k=a+b}^{n} \alpha_k$.
\item Going to the left and under the part of $F(\alpha)$ drawn in the preceding step, intersect the rays $r_x$ for $1\leq x<a+b$. \\
The result of this step is the root $(1,\dots,1)+\sum_{j=1}^{a} \alpha_j$.
\item Going to the right, intersect the rays $r_x$ for $a+b<x\leq n$, and then, going to the left, intersect the rays $r_x$ for $a+b<x<n$ (similarly to steps (1) and (2)). \\
The result is the root $(1,\dots,1)+\sum_{j=1}^{a} \alpha_j+\sum_{k=a+b+1}^{n} \alpha_k$. 
\item Repeat steps (3)-(5) $m$ times in total. \\
The final result is the root $(m,\dots,m)+\sum_{j=1}^{a} \alpha_j+\sum_{k=a+b+1}^{n} \alpha_k$, which is indeed $\alpha$. 
\end{enumerate}
\end{proof}

To obtain $\gamma_\alpha^{\text{op}}$, 
\begin{enumerate} 
\item Let $\gamma_\alpha^{\text{op}}$ start at $M_{a+b}$ and, going counter-clockwise around the inner boundary, intersect $l_{a+b}$ in $m$ points, not including $M_{a+b}$.
\item End at $M_{a+1}$.
\end{enumerate}

As in Section 3, a description of the curves in terms of spirals will be used to prove Theorem \ref{mainthm}(2). In the preceding section, we explicitly counted the elements of $\gamma_\alpha^{\text{op}} \cap \gamma_\beta^{\text{op}}$ and of $F(\alpha) \cap F(\beta)$ to prove the equality $ \text{Int}(\gamma_\alpha^{\text{op}},\gamma_\beta^{\text{op}})=\text{Int}(F(\alpha),F(\beta))$. We then explained that the equality was due to an intersection-preserving deformation between $\gamma_\alpha^{\text{op}} \cup \gamma_\beta^{\text{op}}$ and $F(\alpha) \cup F(\beta)$. The concept of the proof of Theorem \ref{type22} is exactly the same as that of Theorem \ref{type12}, but we will omit the explicit computation and simply explain the deformation. 

First, let $m>0$ and consider $\alpha=((m+1)^a,m^b,(m+1)^c)$. Then $\gamma_\alpha^{\text{op}}=\text{Sp}_m(M_{a+b},w_{a+b})\cup H(w_{a+b},M_{a+1})$ for some $w_{a+b} \in l_{a+b}-\{M_n, M_{a+b}\}$. If $\beta=\alpha+(\lambda,\dots,\lambda)$ for $\lambda \geq 0$,  then $\gamma_\beta^{\text{op}}=\text{Sp}_{m+\lambda}(M_{a+b},q_{a+b})\cup H(q_{a+b},M_{a+1})$ for some $q_{a+b}$ in $l_{a+b}-\{M_n, M_{a+b}\}$. In the case $m=0$, let us say $\gamma_\alpha^{\text{op}}=\text{Sp}_0(M_{a+b},M_{a+b})\cup  H(M_{a+b},M_{a+1})= H(M_{a+b},M_{a+1})$. $\text{Sp}_0(M_{a+b},M_{a+b})$ should be interpreted as the point $M_{a+b}$.

Again assume $m>0$. When drawing $\gamma_\alpha^{\text{op}} \cup \gamma_\beta^{\text{op}}$, we can always choose $q_{a+b}$ to be between $w_{a+b}$ and $M_n$, and $w_{a+b}$ to be between the $(m-1)$th and $m$th intersection of $\gamma_\beta^{\text{op}}$ with $l_{a+b}$. This amounts to choosing representatives of the respective homotopy classes which have the fewest pairwise intersections, and so we will always assume it is done. In particular, it ensures that $ \text{Sp}_m(M_{a+b},w_{a+b}) \cap \text{Sp}_{m+\lambda}(M_{a+b},q_{a+b})=H(w_{a+b},M_{a+1}) \cap H(q_{a+b},M_{a+1})=\emptyset$.  If $m=0$, then any choice of $q_{a+b}\in l_{a+b}-\{M_n, M_{a+b}\}$ results in zero pairwise intersections between $\gamma_\alpha^{\text{op}}$ and $\gamma_\beta^{\text{op}}$.

When $m>0$, the segment of $F(\alpha)$ which results from steps (3)-(6) has some initial point $w_1 \in r_1-\{p_1\}$. Denote by  $H(p_1,w_1)$ the segment of $F(\alpha)$ between $p_1$ and $w_1$, inclusive. Call the segment of $F(\alpha)$ which results from steps (3)-(6) $\text{Sp}_m(w_1,B)$. So $F(\alpha)=H(p_1,w_1) \cup \text{Sp}_m(w_1,B)$. Suppose $\beta=\alpha+(\lambda,\dots,\lambda)$ for $\lambda \geq 0$. Then $F(\beta)=H(p_1,q_1)\cup \text{Sp}_{m+\lambda}(q_1,B)$ where $q_1 \in r_1 -\{p_1\}$ can be chosen so that $\text{Im}(p_1)<\text{Im}(q_1)<\text{Im}(w_1)$. This enables us to draw $F(\alpha)\cup F(\beta)$ in such a way that  $H(p_1,w_1) \cap H(p_1,q_1)=\emptyset$ and $\text{Sp}_m(w_1,B) \cap \text{Sp}_{m+\lambda}(q_1,B) = \emptyset$. If $m=0$, then the process for creating $F(\alpha)$ terminates after step (2). In this case, by definition write $F(\alpha)=H(p_1,B) \cup \text{Sp}_0(B,B)=H(p_1,B)$. We interpret $\text{Sp}_0(B,B)$ to be the point $B$. $F(\alpha)$ and $F(\beta)$ will always have zero pairwise intersections in this case.

\begin{Exa}\label{spiral} Let $\alpha=2112$. Then $F(\alpha)=H(p_1,w_1) \cup \text{Sp}_1(w_1,B)$. 
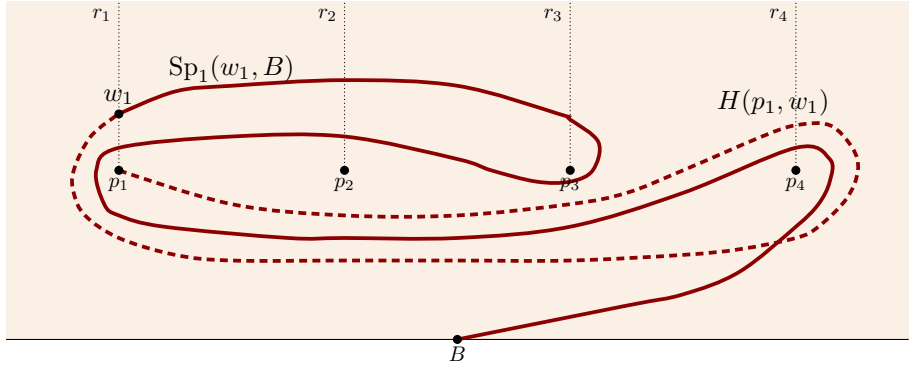
 \begin{center}
\begin{tikzpicture}[scale=1.5]   
\fill[Linen] (-1,0) rectangle (7,3);
\draw[densely dotted] (0,1.5) -- (0,3) node[below left, scale=.8]{$r_1$}; 
\draw[densely dotted] (2,1.5) -- (2,3) node[below left, scale=.8]{$r_2$}; 
\draw[densely dotted] (4,1.5) -- (4,3) node[below left, scale=.8]{$r_3$}; 
\draw[densely dotted] (6,1.5) -- (6,3) node[below left, scale=.8]{$r_4$}; 
\draw[-] (-1,0) -- (7,0);
\draw [DarkRed,densely dashed, line width=.5mm] plot [smooth, tension=0.6] coordinates {(0,1.5)(1,1.2)(2,1.1)(3,1.1)(4,1.2)(4.5,1.3)(5,1.5) (5.7,1.8)(6,1.9) (6.3,1.9)(6.5,1.7)(6.55,1.5)(6.45,1.25)(6.2,1)(6,.9)(5.5,.8)(5,.75)(4,.7)(3,.7)(2,.7)(1,.7)(.5,.75)(0,.9)(-.3,1.1)(-.4,1.3)(-.4,1.5)(-.25,1.8)(0,2)};
\node[scale=1]  at (5.8,2.1){$H(p_1,w_1)$}; 
\draw [DarkRed, line width=.5mm] plot [smooth, tension=0.6]
 coordinates{(0,2)(.5,2.2)(1,2.25)(2,2.3)(3,2.25)(3.9,2)(4,1.95) (4.25,1.75)(4.2,1.5)(4,1.4)(3.7,1.4)(3.3,1.5) (3,1.6) (2,1.8)(1,1.8)(0,1.7)(-.2,1.5)(-.13,1.2)(0,1.1) (.4,1)(1.5,.9)(2,.9)(3,.9)(4,1)(5,1.3)(6,1.7)(6.3,1.6)(6.3,1.4)(6.2,1.2)(6,1)(5.5,.6)(5,.4)(4.5,.3)(3,0) };
 \node[scale=1]  at (1,2.4){$\text{Sp}_1(w_1,B)$}; 
 \draw[fill] (3,0) circle(1 pt) node[below, scale=.8] {$B$};
 \draw[fill] (0,1.5) circle(1 pt) node[below, scale=.8] {$p_1$}; 
\draw[fill] (2,1.5) circle(1 pt) node[below, scale=.8] {$p_2$};
\draw[fill] (4,1.5) circle(1 pt) node[below, scale=.8] {$p_3$};
\draw[fill] (6,1.5) circle(1 pt) node[below, scale=.8] {$p_4$};  
\draw[fill] (0,2) circle(1 pt) node[above, scale=1] {$w_1$};  
\end{tikzpicture}
\end{center}
\end{Exa}

\begin{Thm}[\ref{mainthm}(2)]\label{type22} Let $\alpha$ and $\beta$ be given such that $\alpha=((m+1)^a,m^b,(m+1)^c)$ and $\beta-\alpha\in \{0\} \cup \text{Im}_{+}$. Then $\text{Int} (F(\alpha),F(\beta))=\text{Int}(\gamma_\alpha^{\text{op}},\gamma_\beta^{\text{op}})$.
\end{Thm}

\begin{proof}
Let $0\leq \lambda$ be given, so that $F(\alpha)=H(p_1,w_1) \cup \text{Sp}_m(w_1,B)$ and $F(\beta)=H(p_1,q_1)\cup \text{Sp}_{m+\lambda}(q_1,B)$. We understand that $m$ could be zero, in which case $H(p_1,w_1)=H(p_1,B)$ and $\text{Sp}_m(w_1,B)$ degenerates into the point $B$. Recall that $\gamma_\alpha^{\text{op}}=\text{Sp}_m(M_{a+b},w_{a+b})\cup H(w_{a+b},M_{a+1})$ and $\gamma_\beta^{\text{op}}=\text{Sp}_{m+\lambda}(M_{a+b},q_{a+b})\cup H(q_{a+b},M_{a+1})$. If $m=0$ then the expression for $\gamma_\alpha^{\text{op}}$ is understood to be $H(M_{a+b},M_{a+1})$.
Like for roots of the first type, there is a natural deformation of  $F(\alpha) \cup F(\beta)$ into $\gamma_\alpha^{\text{op}} \cup \gamma_\beta^{\text{op}}$ which preserves intersections. Identify $p_1$ with the common terminal point $M_{a+1}$ of $\gamma_\alpha^{\text{op}}$ and $\gamma_\beta^{\text{op}}$ by carrying $p_1$ counter-clockwise over $F(\alpha)\cup F(\beta)$ (see Example \ref{deform}). Identify $B$ with the common initial point $M_{a+b}$.  In the end, $H(p_1,w_1)$ is identified with $H(w_{a+b}, M_{a+1})$ and $\text{Sp}_m(w_1,B)$ with  $\text{Sp}_m(M_{a+b},w_{a+b})$, so $F(\alpha)$ is deformed into $\gamma_\alpha^{\text{op}}$. Similarly, $H(p_1,q_1)$ is identified with $H(q_{a+b},M_{a+1})$ and $\text{Sp}_{m+\lambda}(q_1,B)$ with $\text{Sp}_{m+\lambda}(M_{a+b},q_{a+b})$ (see Example \ref{deform}), so $F(\beta)$ is deformed into $\gamma_\beta^{\text{op}}$. This deformation does not introduce nor remove intersections between the two curves, so we conclude that $\text{Int}(F(\alpha), F(\beta))=\text{Int}(\gamma_\alpha^{\text{op}}, \gamma_\beta^{\text{op}}) $.
\end{proof}

\begin{Rmk} Because $a+1 \leq a+b < n$, the terminal point $M_{a+1}$ lies always to the left of the initial point $M_{a+b}$ on $(\mathcal{S},\mathcal{M})^{\text{op}}$'s outer boundary. By contrast, for roots of type 1, the terminal point could be either on the left or on the right. This is the reason we have to distinguish two different versions of the deformation for roots of type 1, but not for roots of type 2.
\end{Rmk}

\begin{Exa}\label{deform} Let $\alpha=110011$ and $\beta= 221122$. In this example, $m=0$, $\lambda=1$, $a=2$, and $b=2$. $F(\alpha)=H(p_1,B)$ and $F(\beta)=H(p_1,q_1) \cup \text{Sp}_1(q_1,B)$. In the deformation of $F(\alpha) \cup F(\beta)$ into $\gamma_\alpha^{\text{op}} \cup \gamma_\beta^{\text{op}}$, $p_1$ is carried over $F(\alpha) \cup F(\beta)$ counter-clockwise to be identified with $M_3$. 
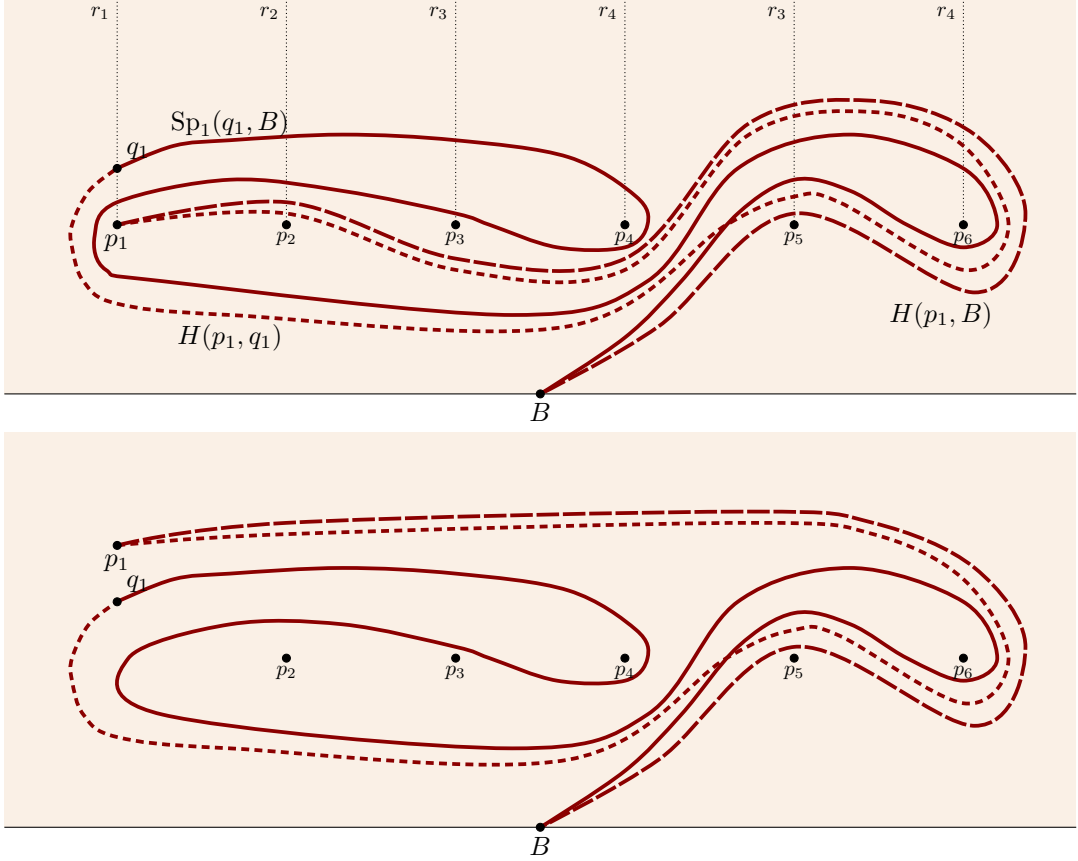
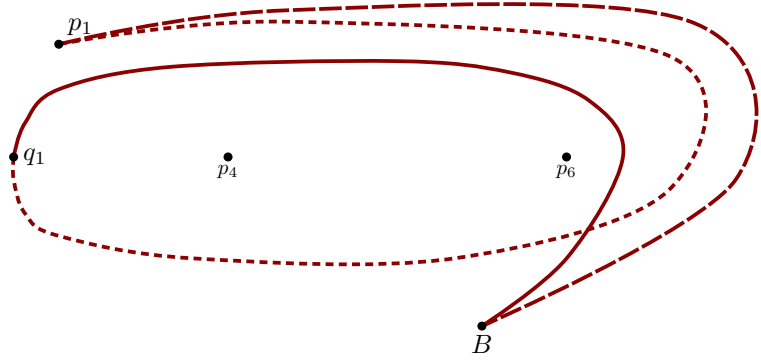
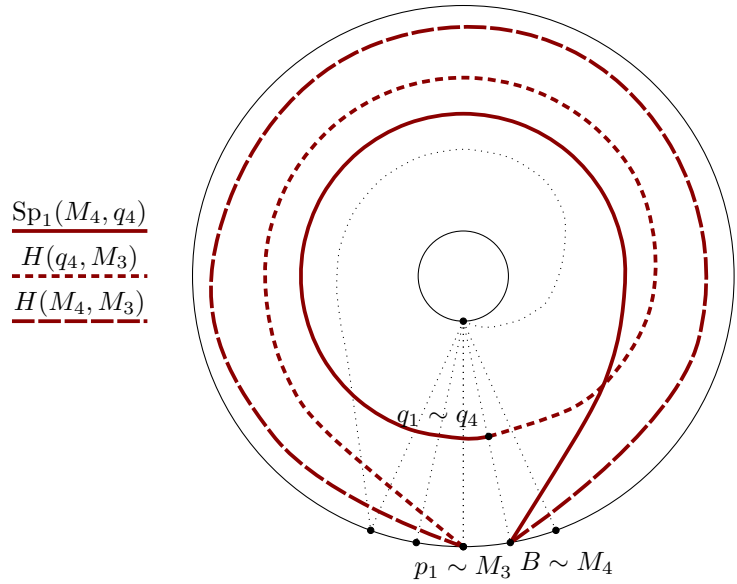
\begin{center}
\begin{tikzpicture}[scale=1.5]
\fill[Linen] (-1,0) rectangle (8.5,3.5);

\draw[-] (-1,0) -- (8.5,0);
\draw[densely dotted] (0,1.5) -- (0,3.5) node[below left, scale=.8]{$r_1$}; 
\draw[densely dotted] (1.5,1.5) -- (1.5,3.5) node[below left, scale=.8]{$r_2$}; 
\draw[densely dotted] (3,1.5) -- (3,3.5) node[below left, scale=.8]{$r_3$}; 
\draw[densely dotted] (4.5,1.5) -- (4.5,3.5) node[below left, scale=.8]{$r_4$}; 
\draw[densely dotted] (6,1.5) -- (6,3.5) node[below left, scale=.8]{$r_3$}; 
\draw[densely dotted] (7.5,1.5) -- (7.5,3.5) node[below left, scale=.8]{$r_4$}; 
\draw [DarkRed,densely dashed, line width=.5mm] plot [smooth, tension=0.6] coordinates {(0,1.5)(1.5,1.6)(3,1.1)(4.5,1.1)(5.6,2.3) (6.7 ,2.5)(7.5,2.2)(7.9,1.5)(7.5,1.1) (6.4,1.7)( 6,1.75 )(5.4,1.5)(4,.6)(1,.7)(0,.8)(-.3,1)(-.4,1.3)(-.4,1.5)(-.25,1.8)(0,2)};
\draw [DarkRed, line width=.5mm] plot [smooth, tension=0.6]
 coordinates{(0,2)(.5,2.2)(1,2.25)(2,2.3)(3,2.25)(4,2.1) (4.6,1.75)(4.7,1.5)(4.5,1.3)(3.9,1.3)(3.3,1.5) (3,1.6) (2,1.8)(1,1.9)(0,1.7)(-.2,1.4) (-.1,1.1)(.4,1)(3.5,.7) (4.7,1)(5.5,2) (6.5,2.3)(7.5,2)(7.8,1.5)(7.5,1.3)(7,1.5)(6.5,1.8)(6,1.9)(5.5,1.6)(5,1)(4.5,.5)(3.75,0) };
 \draw[fill] (0,2 ) circle(1 pt) node[above right, scale=1] {$q_1$} ; 
\draw[DarkRed,dash pattern = on 3mm off .5mm, line width=0.5mm] plot [smooth, tension=0.6] coordinates{(0,1.5)(1.5,1.7)(3,1.2)(4.5,1.2)(5.6,2.4)(6.7,2.6)(7.6,2.3)(8.05,1.6)(7.6,.9 )(6,1.6)(4.8,.6) (3.75,0)};
\draw[fill] (0,1.5) circle(1 pt) node[below, scale=1] {$p_1$}; 
\draw[fill] (1.5,1.5) circle(1 pt) node[below, scale=.8] {$p_2$};
\draw[fill] (3,1.5) circle(1 pt) node[below, scale=.8] {$p_3$};
\draw[fill] (4.5,1.5) circle(1 pt) node[below, scale=.8] {$p_4$};
\draw[fill] (6,1.5) circle(1 pt) node[below, scale=.8] {$p_5$};
\draw[fill] (7.5,1.5) circle(1 pt) node[below, scale=.8] {$p_6$};
\draw[fill] (3.75,0) circle(1 pt) node[below, scale=1] {$B$}; 
 \node[scale=1]  at (1,2.4){$\text{Sp}_1(q_1,B)$}; 
  \node[scale=1]  at (1,.5){$H(p_1,q_1)$}; 
   \node[scale=1]  at (7.3,.7){$H(p_1,B)$}; 
\end{tikzpicture} 
\vspace{5mm}  
\begin{tikzpicture}[scale=1.5]
\fill[Linen] (-1,0) rectangle (8.5,3.5);

\draw[-] (-1,0) -- (8.5,0);

\draw [DarkRed,densely dashed, line width=.5mm] plot [smooth, tension=0.6] coordinates {(0,2.5)(1.5,2.6)(5.6,2.7) (6.7 ,2.6)(7.5,2.2)(7.9,1.5)(7.5,1.1) (6.4,1.7)( 6,1.75 )(5.4,1.5)(4,.6)(1,.7)(0,.8)(-.3,1)(-.4,1.3)(-.4,1.5)(-.25,1.8)(0,2)};
\draw [DarkRed, line width=.5mm] plot [smooth, tension=0.6]
 coordinates{(0,2)(.5,2.2)(1,2.25)(2,2.3)(3,2.25)(4,2.1) (4.6,1.75)(4.7,1.5)(4.5,1.3)(3.9,1.3)(3.3,1.5) (3,1.6) (2,1.8)(1,1.8) (.1,1.5)(.4,1)(3.5,.7) (4.7,1)(5.5,2) (6.5,2.3)(7.5,2)(7.8,1.5)(7.5,1.3)(7,1.5)(6.5,1.8)(6,1.9)(5.5,1.6)(5,1)(4.5,.5)(3.75,0) };
 \draw[fill] (0,2) circle(1 pt) node[above right, scale=1] {$q_1$} ; 
\draw[DarkRed,dash pattern = on 3mm off .5mm, line width=0.5mm] plot [smooth, tension=0.6] coordinates{(0,2.5)(1.5,2.7)(5.6,2.8)(6.7,2.7)(7.6,2.3)(8.05,1.6)(7.6,.9)(6,1.6)(4.8,.6) (3.75,0)};
\draw[fill] (0,2.5) circle(1 pt) node[below, scale=1] {$p_1$}; 
\draw[fill] (1.5,1.5) circle(1 pt) node[below, scale=.8] {$p_2$};
\draw[fill] (3,1.5) circle(1 pt) node[below, scale=.8] {$p_3$};
\draw[fill] (4.5,1.5) circle(1 pt) node[below, scale=.8] {$p_4$};
\draw[fill] (6,1.5) circle(1 pt) node[below, scale=.8] {$p_5$};
\draw[fill] (7.5,1.5) circle(1 pt) node[below, scale=.8] {$p_6$};
\draw[fill] (3.75,0) circle(1 pt) node[below, scale=1] {$B$}; 
\end{tikzpicture} 
\vspace{5mm}  
\begin{tikzpicture}[scale=1.5]

\draw [DarkRed,densely dashed, line width=.5mm] plot [smooth, tension=0.6] coordinates {(0,2.5)(2,2.7)(5.4,2.4) (5.4,1.2)(3.5,.6)(1.2,.6)(0,.8)(-.3,1)(-.4,1.3)(-.4,1.5)};
\draw [DarkRed, line width=.5mm] plot [smooth, tension=0.6]
 coordinates{(-.4,1.5)(-.3,1.8)(0,2.1)(1,2.3 ) ( 3,2.35)(4,2.25 ) (4.7,2)(5,1.5)(4.5,.6)(3.75,0) };
\draw[DarkRed,dash pattern = on 3mm off .5mm, line width=0.5mm] plot [smooth, tension=0.6] coordinates{(0,2.5)(2,2.8)(5.6,2.7)(6,1.3)(3.75,0)};
\draw[fill] (1.5,1.5) circle(1 pt) node[below, scale=.8] {$p_4$};
\draw[fill] (4.5,1.5) circle(1 pt) node[below, scale=.8] {$p_6$};
\draw[fill] (3.75,0) circle(1 pt) node[below, scale=1] {$B$}; 
 \draw[fill] (0,2.5) circle(1 pt) node[above right, scale=1] {$p_1$} ; 
  \draw[fill] (-.4,1.5) circle(1 pt) node[right, scale=1] {$q_1$} ; 
\end{tikzpicture} 
\vspace{5mm}

\begin{tikzpicture}[scale=1.2] 
\centering
    \draw (0,0) circle (3cm)
   	circle (.5cm);
   	\draw[dotted] (0cm,-.5cm) -- (250:3cm)  ; 
   	\draw[dotted] (0cm,-.5cm) -- (260:3cm)  ;
   	\draw[dotted] (0cm,-.5cm) -- (270:3cm)  ;
   	\draw[dotted] (0cm,-.5cm) -- (280:3cm)  ;
   	\draw[dotted] (0cm,-.5cm) -- (290:3cm) ; 
   	\draw [dotted] plot [smooth, tension=.8] coordinates { (270:.5) (290:.6)(310:.7) (330:.85) (350:1) (20:1.2) (40:1.35)(60:1.4)(80:1.4)(100:1.4) (120:1.4)(140:1.4) (160:1.4) (180:1.4) (200:1.45) (220:1.65) (250:3)}; 	 
   	\draw[dotted] (0cm,-.5cm) -- (0cm,-3cm);
   	\draw [DarkRed,dash pattern = on 3mm off .5mm, line width=0.5mm] plot [smooth, tension=0.75] coordinates { (280:3)(345:2.7)(50:2.75) (110:2.75) (170:2.75) (220:2.75)(270:3)};
   	\draw [DarkRed, line width=0.5mm] plot [smooth, tension=0.6] coordinates { (280:3)  (310:2.1) (330:1.9) (350:1.8) (10:1.8) (30:1.8) (50:1.8) (70:1.8) (90:1.8) (110:1.8) (130:1.8) (150:1.8) (170:1.8) (190:1.8) (210:1.8) (230:1.8) (250:1.8)(270:1.8)(279:1.8) };
 \draw [DarkRed,densely dashed, line width=0.5mm] plot [smooth, tension=0.6] coordinates {	(279:1.8) (310:1.9) (330:2) (350:2.1) (10:2.15)(30:2.2)(50:2.2) (70:2.2)(90:2.2)(110:2.2)(130:2.2)(150:2.2) (170:2.2) (190:2.2)(210:2.25)(230:2.3)(270:3) };
   \draw[fill] (279:1.8)  circle(1 pt) node[above left,scale=1] {$q_1 \sim q_4$};  	
\draw[fill] (280:3) circle(1pt)node[below right, scale=1]{$B \sim M_4$};		
   	\draw[fill] (0cm,-3cm) circle(1pt) node[below, scale=1]{$p_1 \sim M_{3}$};   
   	  \draw[fill] (0cm,-.5cm) circle(1pt) ;
   	\draw[fill] (250:3) circle(1pt) ;		
   	\draw[fill] (0cm,-3cm) circle(1pt) ;  
   	\draw[fill] (290:3) circle(1pt);
   	\draw[fill] (260:3) circle(1pt);
   	\draw [DarkRed,densely dashed, line width=.5mm] (-5,0) -- (-3.5,0);   
    \node at (-4.25,.2) {$H(q_4,M_3)$} ;
     \draw [DarkRed, line width=.5mm] (-5,.5) -- (-3.5,.5);  
     \node at (-4.25,.7) {$\text{Sp}_1(M_4,q_4)$} ; 
      \draw [DarkRed,dash pattern = on 3mm off .5mm, line width=.5mm] (-5,-.5) -- (-3.5,-.5);  
     \node at (-4.25,-.3) {$H(M_4,M_3)$} ;  
   	\end{tikzpicture} 
\end{center}
\end{Exa}
 
 \begin{Rmk}
It is not possible to apply this deformation to a pair of curves which correspond to $(\alpha,\beta)$, for $\beta-\alpha \notin \{0\} \cup \text{Im}_{+}$. The deformation relies on making the identifications $M_{a+1} \sim p_1$ and $M_{a+b} \sim B$. Let $\alpha$ and $\beta$ be arbitrary roots of type 2 such that for $\alpha$, $a=a_0$, $b=b_0$, and for $\beta$, $a=a'$, $b=b'$. If $F(\alpha)\cup F(\beta)$ is successfully deformed into $\gamma_\alpha^{\text{op}} \cup \gamma_\beta^{\text{op}}$, then $M_{a_0+b_0}=M_{a'+b'}$, which implies $a_0+b_0=a'+b'$. But since $M_{a_0+1}=M_{a'+1}$, $a_0=a'$. It follows that $b_0=b'$, and therefore $\beta-\alpha \in \{0\} \cup \text{Im}_{+}$.
\end{Rmk}

\section{Proof of Theorem \ref{mainthm} for Schur Roots}
Fix $m\geq 0$, $n\geq 3$ and $a, b \geq 1$ such that $a+b=n$ and consider a Schur root 

\[\alpha=(m+1,\dots,m+1,m,\dots,m)=((m+1)^a,m^b).
\]

To construct $F(\alpha)$,
\begin{enumerate}
\item Start at $p_a$ and, going counter-clockwise, intersect all rays $r_x$ for $1\leq x< a$.
\item Going counter-clockwise, intersect all rays $r_x$, for $1\leq x \leq n$, $m$ times.
\item Shift the initial point to the left $m$ places from $p_a$ to $p_{a-m}$ (where $a-m$ is computed mod $n$ and we set $p_0=p_n$).
\end{enumerate}

\begin{Exa} Some examples for $n=3$. In order to illustrate the effect of step (3), the result of only steps (1) and (2) is shown with a dotted curve in each case.
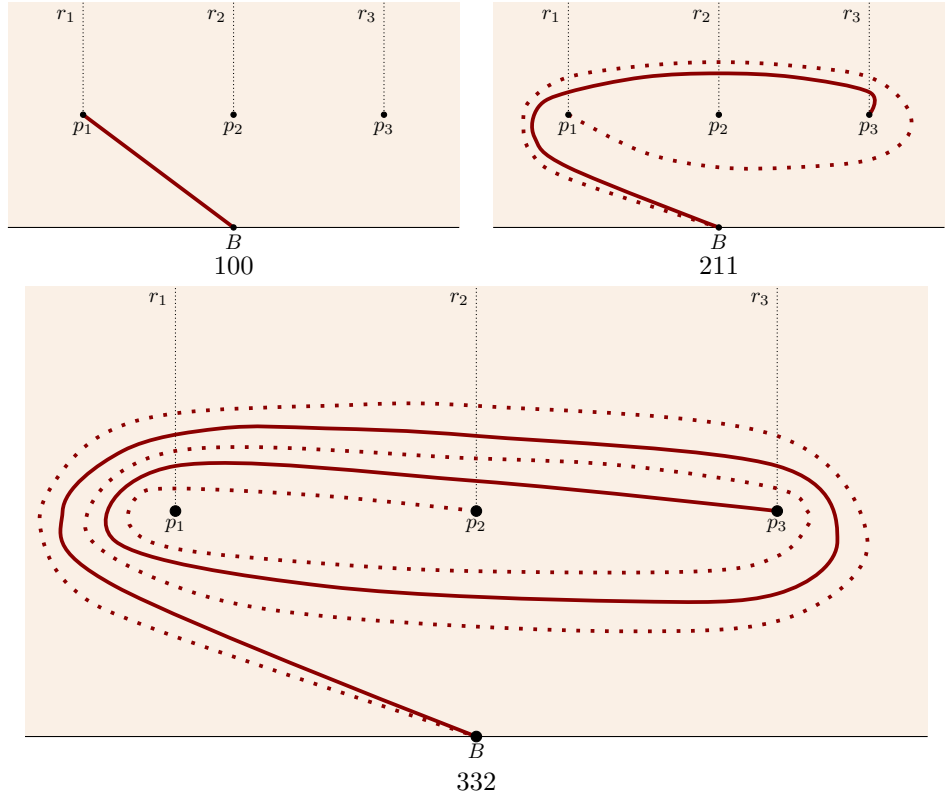
\vspace{2mm}
\begin{center}
\begin{tikzpicture}[scale=1]  
\fill[Linen] (-1,0) rectangle (5,3);
\draw[-] (-1,0) -- (5,0);
\draw[densely dotted] (0,1.5) -- (0,3) node[below left, scale=.8]{$r_1$}; 
\draw[densely dotted] (2,1.5) -- (2,3) node[below left, scale=.8]{$r_2$}; 
\draw[densely dotted] (4,1.5) -- (4,3) node[below left, scale=.8]{$r_3$}; 
\draw [DarkRed, line width=0.5mm] plot [smooth, tension=0.6] coordinates { (0,1.5) (2,0)}; 
\node at (2,-.5) {$100$}; 
\draw[fill] (0,1.5) circle(1 pt) node[below, scale=.8] {$p_1$}; 
\draw[fill] (2,1.5) circle(1 pt) node[below, scale=.8] {$p_2$};
\draw[fill] (4,1.5) circle(1 pt) node[below, scale=.8] {$p_3$};
\draw[fill] (2,0) circle(1 pt) node[below, scale=.8] {$B$}; 
\end{tikzpicture}
\hspace{2mm}
\begin{tikzpicture}[scale=1]
\fill[Linen] (-1,0) rectangle (5,3); 
\draw[-] (-1,0) -- (5,0);
\draw[densely dotted] (0,1.5) -- (0,3) node[below left, scale=.8]{$r_1$}; 
\draw[densely dotted] (2,1.5) -- (2,3) node[below left, scale=.8]{$r_2$}; 
\draw[densely dotted] (4,1.5) -- (4,3) node[below left, scale=.8]{$r_3$}; 
\draw [DarkRed,line width=0.5mm] plot [smooth, tension=0.6] coordinates { (4,1.5)  (4,1.8) (3,2) (2,2.05) (1,2) (0,1.8) (-.4,1.6) (-.45,1.2) (0,.8) (2,0)}; 
\node at (2,-.5) {$211$}; 
\draw [DarkRed, loosely dotted, line width=.5mm] plot [smooth, tension=0.6] coordinates {(0,1.5) (1,1) (2,.8) (3,.8) (4,.9) (4.5,1.2) (4.5, 1.6) (4,2) (3,2.15) (2,2.2) (1,2.15) (0,2) (-.5,1.7) (-.6,1.4) (-.5,1)(0,.65) (2,0)} ; 
\draw[fill] (0,1.5) circle(1 pt) node[below, scale=.8] {$p_1$}; 
\draw[fill] (2,1.5) circle(1 pt) node[below, scale=.8] {$p_2$};
\draw[fill] (4,1.5) circle(1 pt) node[below, scale=.8] {$p_3$};
\draw[fill] (2,0) circle(1 pt) node[below, scale=.8] {$B$}; 
\end{tikzpicture}

\begin{tikzpicture}[scale=2] 
\fill[Linen] (-1,0) rectangle (5,3);
\draw[-] (-1,0) -- (5,0);
\draw[densely dotted] (0,1.5) -- (0,3) node[below left, scale=.8]{$r_1$}; 
\draw[densely dotted] (2,1.5) -- (2,3) node[below left, scale=.8]{$r_2$}; 
\draw[densely dotted] (4,1.5) -- (4,3) node[below left, scale=.8]{$r_3$}; 
\draw [DarkRed, line width=0.5mm] plot [smooth, tension=0.6] coordinates { (4,1.5) (2,1.7) (0,1.8) (-.4,1.3) (1,1) (3,.9)(4,.95) (4.4,1.3) (4,1.8) (2,2) (1,2.05) (0.3,2.05) (-.35,1.9)(-.75,1.5) (-.4,1) (2,0)}; 
\draw [DarkRed, loosely dotted, line width=.5mm] plot [smooth, tension=0.6] coordinates {(2,1.5) (1,1.6) (0,1.65)(-.25,1.6)(-.3,1.4)(0,1.25)(1,1.15)(2,1.1) (3.5,1.1)(4,1.2)(4.2,1.4) (4.2,1.5) (4,1.65) (3,1.8) (2,1.85) (0,1.9)  (-.6,1.4) (0,1)(1,.8)(3,.7) (4.1,.8) (4.6,1.3)  (4,2) (2,2.2) (1,2.2) (0,2.15) (-.5,2)(-.8,1.7) (-.9,1.4) (-.7,1)(-.1,.7) (2,0)}; 
\node at (2,-.3) {$332$};
\draw[fill] (2,0) circle(1 pt) node[below, scale=.8] {$B$}; 
\draw[fill] (0,1.5) circle(1 pt) node[below, scale=.8] {$p_1$}; 
\draw[fill] (2,1.5) circle(1 pt) node[below, scale=.8] {$p_2$};
\draw[fill] (4,1.5) circle(1 pt) node[below, scale=.8] {$p_3$};
\end{tikzpicture}
\end{center}  
\end{Exa}

Let $\widehat{F(\alpha)}$ denote the curve which is formed through steps (1) and (2) of constructing $F(\alpha)$. Observe that 

\begin{align*} 
S(\widehat{F(\alpha)}) &= \cdots s_n s_1 s_2 \cdots s_{n-1}s_n s_1 \cdots s_{a-1}\alpha_a \\
&=(s_1\cdots s_n)^m(s_1\cdots s_{a-1})\alpha_a.
\end{align*}

Step (3) means removing the first (from the right, starting with $\alpha_a$) $m$ terms of $S(\widehat{F(\alpha)})$. In other words, $S(F(\alpha))$ is  $S(\widehat{F(\alpha)})$ with the first $m$ terms removed. When $n$ and $a$ are fixed, we can consider the set of expressions $\{S(F(\alpha))\thinspace|\thinspace \alpha=((m+1)^a,m^b)\}$ as $m$ varies. To emphasize the dependence on $m$, we introduce a new notation $\mathcal{S}_m$ for $S(F(\alpha))$. We will denote by $\mathcal{R}_m$ the reflection with the same indices.

\begin{Exa}\label{sequence} Fix $n=4$ and $a=2$.

\begin{center}
\begin{tabular}{llll}
	 & $\alpha$ & $\mathcal{S}_m$ & $\mathcal{R}_m$ \\
	$m=0$ & 1100 & $s_1\alpha_2$ & $s_1s_2$\\
	$m=1$ & 2211 & $s_1s_2s_3s_4\alpha_1$ & $s_1s_2s_3s_4s_1$ \\
	$m=2$ & 3322 & $s_1s_2s_3s_4s_1s_2s_3\alpha_4$ & $s_1s_2s_3s_4s_1s_2s_3s_4$ \\
	$m=3$ & 4433 & $s_1s_2s_3s_4s_1s_2s_3s_4s_1s_2\alpha_3$ & $s_1s_2s_3s_4s_1s_2s_3s_4s_1s_2s_3$\\
\end{tabular}
\end{center}
\end{Exa}

\begin{Thm}[\ref{mainthm}(1)] Fix $m\geq 0$, $n\geq 3$ and $a, b \geq 1$ such that $a+b=n$ and consider the root $\alpha=((m+1)^a,m^b)$. Then $R(F(\alpha))=\alpha$.
\end{Thm}
  
\begin{proof}  
Fix $n$ and $a$. We will prove by induction on $m$ that $\mathcal{S}_m$ is an expression for $\alpha$.  First of all, when $m=0$, step (1) in forming $\mathcal{S}_m$ creates the root $\sum_{k=1}^{a} \alpha_k$. The remaining steps make no change to that root, so $\mathcal{S}_m$ is an expression for $((m+1)^a,m^b)=(1^a,0^b)$, as desired. Now suppose that $\mathcal{S}_m$ is an expression for $((m+1)^a,m^b)$ for some arbitrary but fixed $m$. We want to show that $\mathcal{S}_{m+1}$ expresses $((m+2)^a,(m+1)^b)$.
 
 Observe that $\mathcal{S}_{m+1}=(s_1\cdots s_n) \mathcal{S}'_m$, where $\mathcal{S}'_m$ is the same sequence as $\mathcal{S}_m$ but with the rightmost term removed. But $\mathcal{S}'_m = (s_1 \cdots s_n)^d  (*) \alpha_j$ for some $d \geq 0$ and $1 \leq j \leq n$, and where $(*)$ represents the sequence of $j-1$ simple reflections $s_{1}\cdots s_{j-1}$ (it is empty if $j=1$). It follows that 

\[\mathcal{S}_{m+1}=(s_1 \cdots s_n)^{d+1}(*)\alpha_j=(s_1 \cdots s_n)^d (s_1 \cdots s_j \cdots s_n) (*) \alpha_j = (s_1 \cdots s_n)^d (*) (s_j \cdots s_n) (*) \alpha_j .\]

Note that the indices of $(s_{j}\cdots s_n)(*)$ form a sequence of $n$ consecutive integers modulo $n$, starting with $j$ (which may be $n$, in which case $s_{j}\cdots s_n$ is the single reflection $s_n$). Let $\mathcal{H}$ represent $(s_{j}\cdots s_n)(*)$ with the two leftmost terms removed and $\alpha_j$ added onto the right; the indices of $\mathcal{H}$ are the $n-1$ consecutive integers starting with $j+2$. Then

\begin{align*}
\mathcal{S}_{m+1}&= (s_1 \cdots s_n)^d (*)( s_j \cdots s_n) (*) \alpha_j \\
&= \mathcal{R}'_m (s_{j+1}\cdots s_n)(*)\alpha_j\\
&= \mathcal{R}_m \mathcal{H}.
\end{align*}

In words, $\mathcal{S}_{m+1}$ is $\mathcal{R}_m$ with the next $n-1$ terms added onto the right, as Example \ref{sequence} illustrates.  
 
 Now note that $\mathcal{S}_m=\mathcal{R}'_m \alpha_{j+1}$ is the same as $\mathcal{R}'_m s_{j+1}(-\alpha_{j+1})=\mathcal{R}_m(-\alpha_{j+1})$. It follows from our characterization of the simple reflections in the proof of Proposition \ref{realroot} that if $\mathcal{R}$ is any reflection and $\mathbf{v}$ any real root such that $\mathcal{R}(\mathbf{v}) = \mathbf{w}$, then $\mathcal{R}(\mathbf{\lambda} + \mathbf{v})=\mathbf{\lambda} + \mathbf{w}$, for $\lambda=(\lambda,\dots,\lambda)$ and $\lambda \in \mathbb{Z}$. 
$\mathcal{H}$ is an expression for the root $(1,\dots,1)-\alpha_{j+1}$.  Therefore, $\mathcal{R}_m (\mathcal{H})=(1,\dots,1)+\mathcal{R}_m (-\alpha_{j+1})=(1,\dots,1)+ \mathcal{S}_m$,  which is $(1,\dots,1)+((m+1)^a,m^b)$ by hypothesis. Hence $\mathcal{S}_{m+1}=((m+2)^a,(m+1)^b)$, as desired. It follows by induction that $\mathcal{S}_m=S(F(\alpha))$ expresses $\alpha=((m+1)^a,m^b)$ for all choices of $n,a,m$. 
\end{proof} 

Now, to obtain $\gamma_\alpha^{\text{op}}$, perform the steps:
\begin{enumerate}
\item Start at $M_n$ and, going counter-clockwise around the inner boundary, intersect $l_n$ in $m$ points, not including $M_a$.
\item End at $M_{a+1}$. (If $a+1=n$, end at $M_1$.)
\end{enumerate}
We can describe $\gamma_\alpha^{\text{op}}$ in terms of spirals. Suppose $m>0$ and pick $w_{n}$ in $l_{n}-\{M_1, M_n\}$. Let $H(M_n,w_n)$ be the curve segment with endpoints $M_n$ and $w_n$ which goes counter-clockwise around $(\mathcal{S},\mathcal{M})^{\text{op}}$'s inner boundary once. Let $\text{Sp}_m(w_n,M_{a+1})$ be the curve segment with endpoints $w_n$ and $M_{a+1}$ which, going counter-clockwise, intersects $l_{a+1}$ $m$ times (the final intersection being $M_{a+1}$). Then $\gamma_\alpha^{\text{op}}=H(M_n,w_n) \cup \text{Sp}_m(w_n,M_{a+1})$. If $\beta=\alpha +(\lambda,\dots,\lambda)$ for $\lambda \geq 0$, then $\gamma_\beta^{\text{op}}=H(M_n,q_n) \cup \text{Sp}_{m+\lambda}(q_n,M_{a+1})$ for some $q_n \in l_{n}-\{M_1, M_n\}$ which can be chosen strictly between $w_n$ and $M_n$. In case $m=0$, then $\gamma_\alpha^{\text{op}}=H(M_n,M_{a+1})\cup \text{Sp}_m(M_{a+1},M_{a+1})=H(M_n,M_{a+1})$.
 
\begin{Exa}\label{Schur} Let $\alpha=322$. Then $\gamma_\alpha^{\text{op}}= H(M_3,w_3)\cup \text{Sp}_2(w_3,M_2)$.
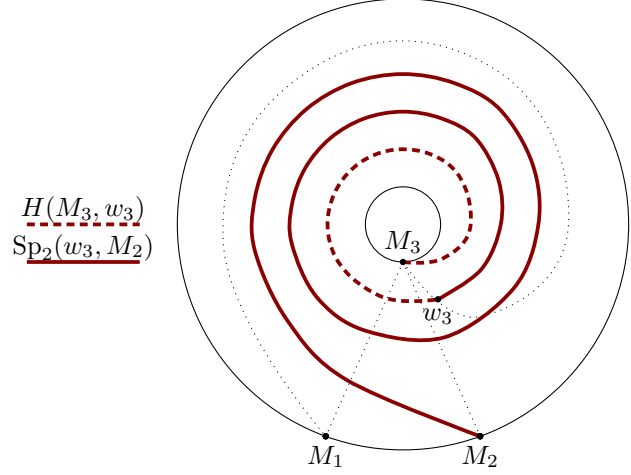
\begin{center}
\begin{tikzpicture}[scale=1]
\draw (0,0) circle (3cm) 
   	circle (.5cm);
   	\draw[dotted] (0cm,-.5cm) -- (250:3cm)  ;
   	\draw[dotted] (0cm,-.5cm) -- (290:3cm) ;
   	\draw [dotted] plot [smooth, tension=0.8] coordinates { (270:.5) (310:1.6) (360:2.2) (60:2.4) (120:2.4) (185:2.4)(250:3)}; 	 
 	\draw [DarkRed,densely dashed, line width=0.5mm] plot [smooth, tension=0.7] coordinates { (270:.5) (320: .7) (360:.9) (50:1.) (100:1.) (150:1.) (200:1.) (252:1.)(295:1.1)};
 	\draw [DarkRed, line width=0.5mm] plot [smooth, tension=0.6] coordinates {(295:1.1)(330:1.2)(360:1.3) (25:1.4) (50:1.5)(70:1.5)(90:1.5)(110:1.5)(130:1.5)(150:1.5)(170:1.5)(190:1.5) (220:1.5) (250:1.5)(290:1.6) (330:1.7)(360:1.8)(25:1.9)(50:2)(70:2)(90:2)(110:2)(130:2)(150:2)(170:2)(190:2)(220:2)(250:2.2)(290:3)}; 
  
   	\draw[fill] (0cm,-.5cm) circle(1pt) node[above] {$M_3$};
   	\draw[fill] (250:3) circle(1pt) circle(1pt) node[below] {$M_1$};
   	\draw[fill] (290:3) circle(1pt) circle(1pt) node[below] {$M_2$};	
\draw[fill] (295:1.1) circle(1pt) node[below] {$w_3$}; 
	\draw [DarkRed, line width=.5mm] (-5,-.5) -- (-3.5,-.5);   
    \node at (-4.25,-.3) {$\text{Sp}_2(w_3,M_2)$} ;
     \draw [DarkRed,densely dashed, line width=.5mm] (-5,0) -- (-3.5,0);  
     \node at (-4.25,.2) {$H(M_3,w_3)$} ; 
      
\end{tikzpicture}
\end{center}
\end{Exa}

We also give a description of the planar curves in terms of spirals. Let $a$, $1 \leq a < n$, be given. Let $E$ denote the curve which starts at $p_a$, and, going counter-clockwise, first intersects the rays $r_x$ for $1 \leq x <a$ one time and then the rays $r_x$ for $1 \leq x \leq n$ $m>0$ times. If $q$ and $w$ are points on $E$, such that $q$ is before (according to the orientation of $E$) $w$, denote by $E(q,w)$ the segment of $E$ between $q$ and $w$, inclusive. If $q=w$, $E(q,w)$ will symbolize the point $q$.  In particular, if $q$ and $w$ are two distinct intersections of $E$ with $r_1$, where $q$ possibly equals $p_1$, then $E(q,w)$ shall be called a spiral of $E$. 

If $w_{1}$ is the initial point of the first spiral of $E$ (the first intersection of $E$ with $r_1$, possibly equal to $p_1$) and $w_{z}$ is the terminal point of the last spiral of $E$, then the set $\{\text{union of all the spirals of $E$}\} \cup E(w_{z},B) $ shall be denoted by $\text{Sp}_m(w_{1},B)$. The particular segment $E(p_a,w_{1}) $ shall be denoted instead by $H(p_a,w_{1})$. Hence $E=H(p_a,w_1)\cup \text{Sp}_m(w_1,B)$. Note that if $a=1$ then $H(p_a,w_1)=H(p_1,p_1)$ should be understood as the point $p_1$.

 In case $m=0$, denote $E=H(p_a,B)\cup \text{Sp}_0(B,B)=H(p_a,B)$. As usual, interpret the degenerate spiral $\text{Sp}_0(B,B)$ as the point $B$. We still denote the first intersection of $E$ with $r_1$ by $w_1$. 
 
 Now let $0\leq \lambda$ and $1 \leq b <n$ be given. Suppose $F$ is a curve which starts at $p_b$, and, going counter-clockwise, first intersects one time the rays $r_x$ for $1 \leq x <b$ and then intersects $m+\lambda$ times the rays $r_x$ for $1 \leq x \leq n$. If $m>0$, then $F=H(p_b,q_1)\cup \text{Sp}_{m+\lambda}(q_1,B)$, where $q_1$ denotes the first intersection of $F$ with $r_1$. If $m=\lambda=0$, then denote $F=H(p_b,B)$, but still denote the first intersection between $F$ and $r_1$ by $q_1$.

The natural way to draw $E \cup F$ in $\Gamma$ depends on the values of $a$, $b$, $m$, and $\lambda$. First suppose $1\leq b \leq a$. Then, if either $m$ or $\lambda$ is larger than zero, $F$ intersects $r_1$ at least twice. Denote the second point of intersection by $q_2$. Draw $E$ and $F$ in such a way that  $\text{Im}(q_1)\leq \text{Im}(w_1) < \text{Im}(q_2)$, and such that $r_1$ has no other intersections with either $E$ or $F$ between $q_1$ and $q_2$.  In case $\lambda=m=0$, there does not exist a $q_2$, but still draw $\text{Im}(q_1)\leq \text{Im}(w_1)$.  $\text{Im}(q_1)=\text{Im}(w_1)=\text{Im}(p_1)$ if and only if $a=b=1$ (see Example \ref{schurint}). Now suppose $1\leq a <b$. If $m>0$ then $E$ intersects $r_1$ in at least two points $w_1$ and $w_2$. Draw $E \cup F$ in such a way that $\text{Im}(w_1)<\text{Im}(q_1)<\text{Im}(w_2)$ and such that $r_1$ contains no other intersections with either $E$ or $F$ between $w_1$ and $w_2$. If $m=0$ then let $\text{Im}(w_1)<\text{Im}(q_1)<\text{Im}(q_2)$, if $q_2$ exists. In all cases, drawing $E \cup F$ in the natural way coincides with choosing representatives of the respective homotopy classes of $E$ and $F$ which have the fewest pairwise intersections, so we will always assume it is done.

\begin{Exa}\label{schurint} Left picture: $E=H(p_2, B)$ and $F=H(p_2, q_1) \cup \text{Sp}_2(q_1,B)$. Right picture: $E=H(p_1,B)$ and $F=\text{Sp}_2(p_1,B)$.
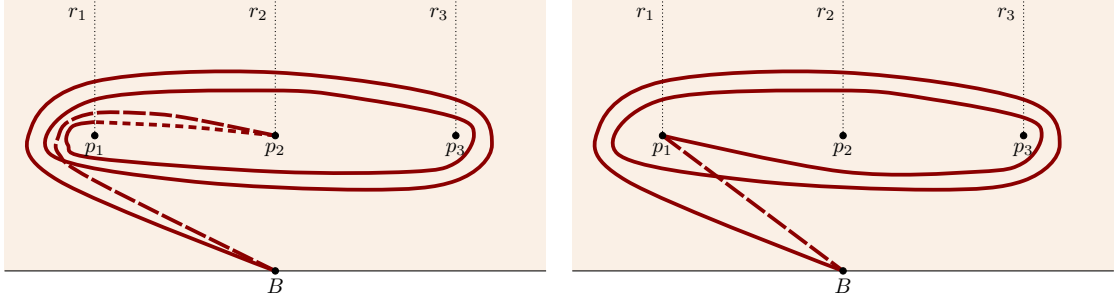
\begin{center}
\begin{tikzpicture}[scale=1.2]   
\fill[Linen] (-1,0) rectangle (5,3);

\draw[-] (-1,0) -- (5,0);
\draw[densely dotted] (0,1.5) -- (0,3) node[below left, scale=.8]{$r_1$}; 
\draw[densely dotted] (2,1.5) -- (2,3) node[below left, scale=.8]{$r_2$}; 
\draw[densely dotted] (4,1.5) -- (4,3) node[below left, scale=.8]{$r_3$}; 
\draw [DarkRed,dash pattern = on 3mm off .5mm, line width=0.5mm] plot [smooth, tension=0.6] coordinates { (2,1.5)(1,1.7)(.5,1.75) (0,1.75) (-.3,1.65)(-.4,1.5) (-.4,1.3)(0,1) (2,0)}; 
\draw [DarkRed,densely dashed, line width=0.5mm] plot [smooth, tension=0.6] coordinates {(2,1.5) (1,1.6) (0,1.65)};
\draw [DarkRed, line width=0.5mm] plot [smooth, tension=0.6] coordinates { (0,1.65)(-.25,1.6)(-.3,1.4)(0,1.25)(2,1.1) (3.5,1.1)(4,1.2)  (4.2,1.5) (4,1.7)(3,1.9)(2,2) (0,1.9) (-.5,1.3) (1,1) (3,.9)(4,1) (4.4,1.4) (4,1.9) (2,2.2)  (0,2.1) (-.7,1.6) (-.4,1) (2,0)};  
\draw[fill] (0,1.5) circle(1 pt) node[below, scale=.8] {$p_1$}; 
\draw[fill] (2,1.5) circle(1 pt) node[below, scale=.8] {$p_2$};
\draw[fill] (4,1.5) circle(1 pt) node[below, scale=.8] {$p_3$};
\draw[fill] (2,0) circle(1 pt) node[below, scale=.8] {$B$};  
\end{tikzpicture}
\hspace{1mm}
\begin{tikzpicture}[scale=1.2] 
\fill[Linen] (-1,0) rectangle (5,3);
\draw[-] (-1,0) -- (5,0);
\draw[densely dotted] (0,1.5) -- (0,3) node[below left, scale=.8]{$r_1$}; 
\draw[densely dotted] (2,1.5) -- (2,3) node[below left, scale=.8]{$r_2$}; 
\draw[densely dotted] (4,1.5) -- (4,3) node[below left, scale=.8]{$r_3$}; 
\draw [DarkRed,dash pattern = on 3mm off .5mm, line width=0.5mm] plot [smooth, tension=0.6] coordinates { (0,1.5) (2,0)}; 
\draw [DarkRed, line width=0.5mm] plot [smooth, tension=0.6] coordinates {  (0,1.5)(2,1.1) (3.5,1.1)(4,1.2)  (4.2,1.5) (4,1.7)(3,1.9)(2,2) (0,1.9) (-.5,1.3) (1,1) (3,.9)(4,1) (4.4,1.4) (4,1.9) (2,2.2)  (0,2.1) (-.7,1.6) (-.4,1) (2,0)}; 
\draw[fill] (0,1.5) circle(1 pt) node[below, scale=.8] {$p_1$}; 
\draw[fill] (2,1.5) circle(1 pt) node[below, scale=.8] {$p_2$};
\draw[fill] (4,1.5) circle(1 pt) node[below, scale=.8] {$p_3$};
\draw[fill] (2,0) circle(1 pt) node[below, scale=.8] {$B$}; 
\end{tikzpicture}
\end{center} 
\end{Exa}

The next lemma shows that the intersections between two curves $E$ and $F$ result from one curve having more spirals than the other.

\begin{Lem}\label{lemma}
Let $1 \leq \lambda$ be given. Suppose $E=H(p_{a},w_1) \cup \text{Sp}_m(w_1, B)$ and  $F=H(p_{b},q_1) \cup \text{Sp}_{m+\lambda} (q_1, B)$ for $0\leq m$. Then $\text{Int}(E,F)=\lambda-1$ if $b \leq a$ and $\text{Int}(E,F)=\lambda$ if $b >a$.

\end{Lem}  

\begin{proof} 

Let $b \leq a$. 
Suppose the first spiral of $F$ is $F(q_1,q_2)$ for two points $q_1,q_2 \in r_1 \cap F$. Because of the way we have chosen to draw $E$ and $F$,  $F(q_1,q_2) \cap E = \emptyset$. After the point $q_2$, $F$ has $m+\lambda-1$ more spirals. After $w_1$, $E$ has $m$ more spirals. Let $q_{m+2}$ be the point which concludes the $(m+1)$th spiral of $F$ and let $w_{m+1}$ be the point which concludes the $m$th spiral of $E$. The segments $F(q_2, q_{m+2})$ and $E(w_1,w_{m+1})$ can be drawn parallel to one another and so $F(q_2, q_{m+2}) \cap E(w_1,w_{m+1}) = \emptyset$. (Note that if $m=0$, $q_{m+2}=q_2$ and $w_{m+1}=w_1$, so these two segments are points.) Since $E$ has no more spirals after this, it must cross the remainder of $F$ to reach $B$. The number of times it crosses is equal to the number of spirals $F$ still has left, $\lambda-1$. It is clear that this accounts for all intersections which cannot be removed by drawing the curves in another way, so $\text{Int}(E,F) =\lambda-1$. 

Now let $b>a$. The segments $F(q_1, q_{m+1})$ and $E(w_1,w_{m+1})$ (the first $m$ spirals of each curve) can be drawn parallel to each other and therefore have no pairwise intersections. After this, $F$ has $\lambda$ more spirals and $E$ has none. So, in order to reach $B$, $E$ must cross $F$ $\lambda$ different times. Hence $\text{Int}(E,F) =\lambda$. 
\end{proof}

\begin{Thm}[\ref{mainthm}(2)] Let $\alpha$ and $\beta$ be given such that $\alpha=((m+1)^a,m^b)$ and $\beta-\alpha \in \{(\lambda,\dots,\lambda) | 0\leq \lambda < n\}$. Then $\text{Int}(F(\alpha),F(\beta))=\text{Int}(\gamma_\alpha^{\text{op}}, \gamma_\beta^{\text{op}})$.
\end{Thm}

\begin{proof}
Let $0\leq \lambda <n$ be given. Then $\widehat{F(\alpha)}=H(p_a,w_1) \cup \text{Sp}_m(w_1,B)$ and $\widehat{F(\beta)}=H(p_a, q_1) \cup \text{Sp}_{m+\lambda}(q_1,B)$ (where we understand that some terms could be degenerate).  There is a deformation of $\widehat{F(\alpha)}\cup \widehat{F(\beta)}$ into $\gamma_\alpha^{\text{op}} \cup \gamma_\beta^{\text{op}}$ which preserves intersections. Identify the common initial point $p_a$ of $\widehat{F(\alpha)}$ and $\widehat{F(\beta)}$ with $M_n$ and the common terminal point $B$ with $M_{a+1}$ (if $a+1=n$, then with $M_1$). $M_n$, is, of course, the common initial point of $\gamma_\alpha^{\text{op}}$ and $\gamma_\beta^{\text{op}}$, and $M_{a+1}$ is their common terminal point. This results in the following identifications (see Example \ref{deformschur}): $H(p_a,w_1)$ with $H(M_n,w_n)$ and $\text{Sp}_m(w_1,B)$ with $\text{Sp}_m(w_n,M_{a+1})$. Hence $\widehat{F(\alpha)}$ is deformed into $\gamma_\alpha^{\text{op}}$. The situation is similar for $\widehat{F(\beta)}$. Because no intersections are introduced nor removed throughout this process, we conclude that $\text{Int}(\widehat{F(\alpha)},\widehat{F(\beta)})=\text{Int}(\gamma_\alpha^{\text{op}},\gamma_\beta^{\text{op}})$.
  
We now wish to show that $\text{Int}(\widehat{F(\alpha)},\widehat{F(\beta)})=\text{Int}(F(\alpha),F(\beta))$. If $\lambda=0$, then $\widehat{F(\alpha)}$ and $\widehat{F(\beta)}$ are homotopic curves with initial point $p_a$ and $F(\alpha)$ and $F(\beta)$ are homotopic curves with inital point $p_{a-m}$. Clearly then $\text{Int}(\widehat{F(\alpha)},\widehat{F(\beta)})=0=\text{Int}(F(\alpha),F(\beta))$.

Let $\lambda>0$. We will proceed in two steps. Denote by $\widehat{F(\beta)}'$ the result of shifting the initial point of $\widehat{F(\beta)}$ to the left $m$ places. Of course, $F(\alpha)$ is the result of shifting the initial point of $\widehat{F(\alpha)}$ to the left $m$ places. We first claim that $\text{Int}(\widehat{F(\alpha)},\widehat{F(\beta)})=\text{Int}(F(\alpha),\widehat{F(\beta)}')$. $\beta$ was chosen so that $\widehat{F(\beta)}$ has $\lambda$ more spirals than $\widehat{F(\alpha)}$, so Lemma \ref{lemma} implies that $\text{Int}(\widehat{F(\alpha)},\widehat{F(\beta)})=\lambda-1$.  Shifting the common endpoint of $\widehat{F(\alpha)}$ and $\widehat{F(\beta)}$ removes the same number of spirals from $\widehat{F(\alpha)}$ as it does from $\widehat{F(\beta)}$.   It follows that $\widehat{F(\beta)}'$ has $\lambda$ more spirals than $F(\alpha)$, so Lemma \ref{lemma} again implies that $\text{Int}(F(\alpha),\widehat{F(\beta)}')=\lambda-1$.  
 
The second claim is that $\text{Int}(F(\alpha),\widehat{F(\beta)}')=\text{Int}(F(\alpha),F(\beta))$, where, recall, $F(\beta)$ is the result of shifting the initial point of $\widehat{F(\beta)}'$ back $\lambda$ places from $p_{a-m}$ to $p_{a-m-\lambda}$.   If $a-m-\lambda<a-m$ (as always, the indices are mod $n$), then this operation removes no spirals from $\widehat{F(\beta)}'$ and hence the difference in spirals between $F(\alpha)$ and $F(\beta)$ remains the same as the difference between $F(\alpha)$ and $\widehat{F(\beta)}'$, $\lambda$. So Lemma \ref{lemma} implies that $\text{Int}(F(\alpha),F(\beta))=\text{Int}(F(\alpha),\widehat{F(\beta)}')=\lambda-1$. Here we have used the assumption that $\lambda<n$. For if $\lambda \geq n$, then it could happen that $a-m-\lambda<a-m$ and a positive number of spirals are removed from $\widehat{F(\beta)}'$.
If $a-m < a-m -\lambda$, then this operation removes exactly one spiral from  $\widehat{F(\beta)}'$. Hence the difference in spirals between $F(\beta)$ and $F(\alpha)$ is $\lambda-1$. Lemma \ref{lemma} then implies that $\text{Int} (F(\alpha),F(\beta))=\lambda-1$. Thus, in any case, $\text{Int} (F(\alpha),F(\beta))=\text{Int}(F(\alpha),\widehat{F(\beta)}')=\lambda-1$.
\end{proof}

\begin{Exa}\label{deformschur}$\widehat{F(\alpha)}\cup \widehat{F(\beta)}$ is deformed into $\gamma_\alpha^{\text{op}} \cup \gamma_\beta^{\text{op}}$ for $\alpha=111000$ and $\beta=333222$. Here, $a=3$, $m=0$, and $\lambda=2$.
 
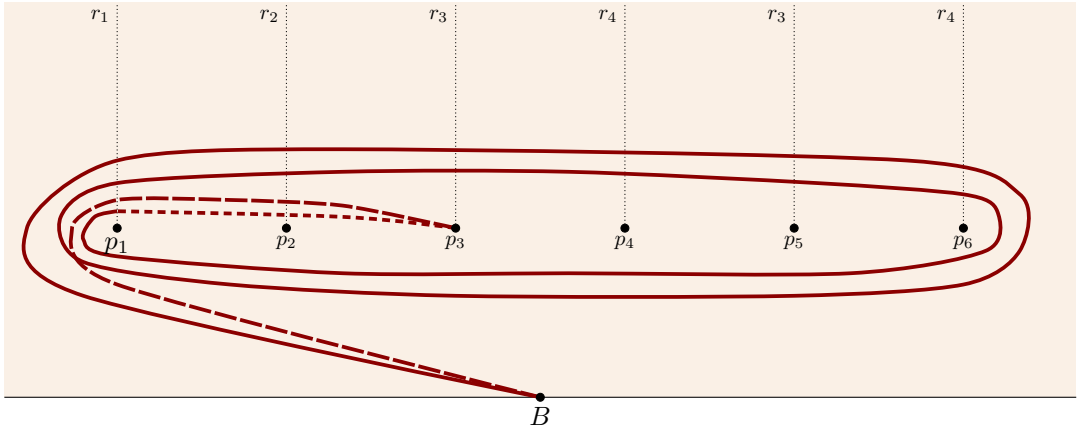
\begin{center} 
\begin{tikzpicture}[scale=1.5]
\fill[Linen] (-1,0) rectangle (8.5,3.5);

\draw[-] (-1,0) -- (8.5,0);
\draw[densely dotted] (0,1.5) -- (0,3.5) node[below left, scale=.8]{$r_1$}; 
\draw[densely dotted] (1.5,1.5) -- (1.5,3.5) node[below left, scale=.8]{$r_2$}; 
\draw[densely dotted] (3,1.5) -- (3,3.5) node[below left, scale=.8]{$r_3$}; 
\draw[densely dotted] (4.5,1.5) -- (4.5,3.5) node[below left, scale=.8]{$r_4$}; 
\draw[densely dotted] (6,1.5) -- (6,3.5) node[below left, scale=.8]{$r_3$}; 
\draw[densely dotted] (7.5,1.5) -- (7.5,3.5) node[below left, scale=.8]{$r_4$}; 
\draw [DarkRed,dash pattern = on 3mm off .5mm, line width=0.5mm] plot [smooth, tension=0.45] coordinates { (3,1.5)(2,1.7)(1,1.75) (0,1.75)  (-.4,1.5) (0,1) (3.75,0)}; 
\draw [DarkRed,densely dashed, line width=0.5mm] plot [smooth, tension=0.6] coordinates {(3,1.5) (2,1.6) (0,1.65)};
\draw [DarkRed, line width=0.5mm] plot [smooth, tension=0.5] coordinates { (0,1.65)(-.2,1.6)(-.3,1.4)(0,1.25)(2,1.1) (4,1.1)(6.5,1.1)  (7.7,1.3)(7.8,1.65) (7.5,1.8)(6.2,1.9)(4,2)(2,2) (0,1.9) (-.5,1.6) (-.3,1.2) (1,1) (3,.9)(6,.9) (7.5,1) (8,1.3)(8,1.8)(7,2.1)(2,2.2)  (0,2.1) (-.8,1.5) (-.3,.9) (3.75,0)};  
\draw[fill] (0,1.5) circle(1 pt) node[below, scale=1] {$p_1$}; 
\draw[fill] (1.5,1.5) circle(1 pt) node[below, scale=.8] {$p_2$};
\draw[fill] (3,1.5) circle(1 pt) node[below, scale=.8] {$p_3$};
\draw[fill] (4.5,1.5) circle(1 pt) node[below, scale=.8] {$p_4$};
\draw[fill] (6,1.5) circle(1 pt) node[below, scale=.8] {$p_5$};
\draw[fill] (7.5,1.5) circle(1 pt) node[below, scale=.8] {$p_6$};
\draw[fill] (3.75,0) circle(1 pt) node[below, scale=1] {$B$}; 
  
\end{tikzpicture}  

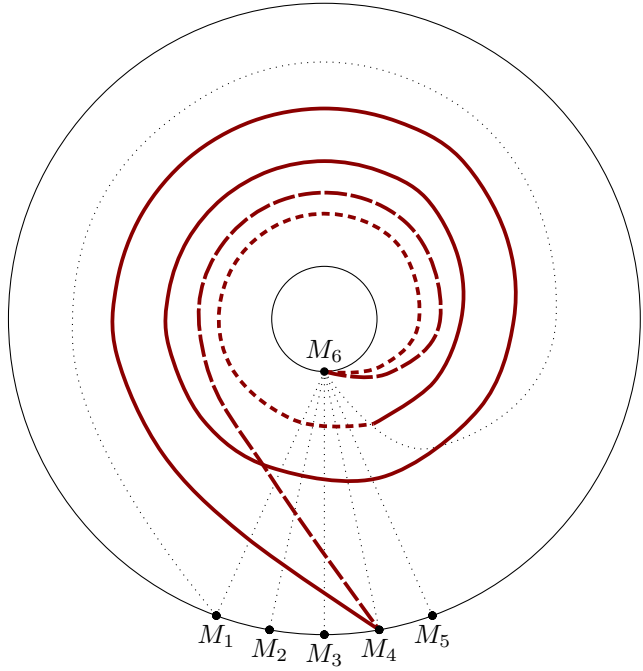
\begin{tikzpicture}[scale=1.4]
\draw (0,0) circle (3cm) 
   	circle (.5cm);
   	\draw[dotted] (0cm,-.5cm) -- (250:3cm)  ;
    \draw[dotted] (0cm,-.5cm) -- (260:3cm) ;
    \draw[dotted] (0cm,-.5cm) -- (270:3cm) ;
    \draw[dotted] (0cm,-.5cm) -- (280:3cm) ;
   	\draw[dotted] (0cm,-.5cm) -- (290:3cm) ;
   	\draw [dotted] plot [smooth, tension=0.8] coordinates { (270:.5) (310:1.6) (360:2.2) (60:2.4) (120:2.4) (185:2.4)(250:3)}; 	 
    \draw [DarkRed,dash pattern = on 3mm off .5mm, line width=0.5mm] plot [smooth, tension=0.7] coordinates {(270:.5)(320:.8)(360:1.1)(50:1.2)(100:1.2)(150:1.2)(200:1.2)(252:1.6)(280:3)};
 	\draw [DarkRed,densely dashed, line width=0.5mm] plot [smooth, tension=0.7] coordinates { (270:.5) (320: .7) (360:.9) (50:1.) (100:1.) (150:1.) (200:1.) (252:1.)(295:1.1)};
 	\draw [DarkRed, line width=0.5mm] plot [smooth, tension=0.6] coordinates {(295:1.1)(330:1.2)(360:1.3) (25:1.4) (50:1.5)(70:1.5)(90:1.5)(110:1.5)(130:1.5)(150:1.5)(170:1.5)(190:1.5) (220:1.5) (250:1.5)(290:1.6) (330:1.7)(360:1.8)(25:1.9)(50:2)(70:2)(90:2)(110:2)(130:2)(150:2)(170:2)(190:2)(220:2)(250:2.2)(280:3)}; 
  
   	\draw[fill] (0cm,-.5cm) circle(1pt) node[above] {$M_6$};
   	\draw[fill] (250:3) circle(1pt) circle(1pt) node[below] {$M_1$};
    \draw[fill] (260:3) circle(1pt) circle(1pt) node[below] {$M_2$};
    \draw[fill] (270:3) circle(1pt) circle(1pt) node[below] {$M_3$};
    \draw[fill] (280:3) circle(1pt) circle(1pt) node[below] {$M_4$};
   	\draw[fill] (290:3) circle(1pt) circle(1pt) node[below] {$M_5$};

\end{tikzpicture}
\end{center}
\end{Exa}

\begin{Rmk}
We remark again that the deformation is highly dependent upon the values of $a$ and $b$. Suppose the deformation works for arbitrary Schur roots $\alpha$ and $\beta$ such that for $\alpha$, $a=a_0$, $b=b_0$, and for $\beta$, $a=a'$, $b=b'$. Since $p_{a_0} \sim M_n \sim p_{a'}$, $a_0=a'$. But since $c=0$ for Schur roots, this forces $b_0=b'$.
 \end{Rmk}
 
The only roots we have not yet considered are the Schur roots of the form
\[\alpha^{-1}=(m^b,(m+1)^a)\]
for some $\alpha=((m+1)^a,m^b)$ as in the beginning of this section. We define the plane curve $F(\alpha^{-1})$ for $\alpha^{-1}$ to be the geometric reflection of $F(\alpha)$ in the vertical line through $B$. Thus, $S(F(\alpha^{-1}))$ is $S(F(\alpha))$ but with the numbers $k$ and $(n+1)-k$, $1 \leq k \leq n$, interchanged in the indices. For instance, if $S(F(\alpha))$ is $s_1s_2s_3s_4s_1\alpha_2$, then $S(F(\alpha^{-1}))$ is $s_4s_3s_2s_1s_4\alpha_3$. This implies that the $((n+1)-k)$th component of $R(F(\alpha^{-1}))$ is the $k$th component of $R(F(\alpha))$, which precisely means that $R(F(\alpha^{-1}))=\alpha^{-1}$. 

On the other hand, the annulus curve $\gamma_{\alpha^{-1}}^{\text{op}}$ for $\alpha^{-1}$ starts at $M_n$, and, going clockwise around the inner boundary, intersects $l_1$ in $m$ points not including $M_n$; then it ends at $M_{b}$. 

Let $\beta$ be such that $\beta-\alpha \in \{(\lambda,\dots,\lambda) \thinspace | \thinspace 0\leq \lambda < n\}$. It is clear then that  $\text{Int}(F(\alpha), F(\beta))=\text{Int}(F(\alpha^{-1}), F(\beta^{-1}))$ and  that $\text{Int}(\gamma_\alpha^{\text{op}}, \gamma_\beta^{\text{op}})=\text{Int}(\gamma_{\alpha^{-1}}^{\text{op}},\gamma_{\beta^{-1}}^{\text{op}})$. It follows that $\text{Int}(F(\alpha^{-1}),F(\beta^{-1}))=\text{Int}(\gamma_{\alpha^{-1}}^{\text{op}},\gamma_{\beta^{-1}}^{\text{op}})$.

\end{document}